\documentclass[11pt]{amsart}
\usepackage[top=3cm, bottom=2.5cm, left=2.5cm, right=2.5cm]{geometry}                
\usepackage[english]{babel}

\usepackage{graphicx}
\usepackage{amssymb}
\usepackage{tikz}
\usetikzlibrary{arrows,decorations.pathmorphing,cd,positioning}
\usetikzlibrary{shapes.geometric,shapes.misc,patterns}
\usepackage[colorlinks=true]{hyperref}
\usepackage{enumerate}
\usepackage{mathrsfs}

\usepackage{mathpazo}
\usepackage[scaled=.95]{helvet}
\usepackage{courier}

\usepackage{caption,subfig}
\numberwithin{equation}{section}
\numberwithin{figure}{section}
\numberwithin{table}{section}

\newtheorem{theorem}{Theorem}[section]
\newtheorem{proposition}[theorem]{Proposition}
\newtheorem{lemma}[theorem]{Lemma}
\newtheorem{corollary}[theorem]{Corollary}
\newtheorem{conjecture}[theorem]{Conjecture}
\newtheorem*{theoremIntro}{Theorem}

\theoremstyle{definition}
\newtheorem{definition}[theorem]{Definition}

\theoremstyle{remark}
\newtheorem{remark}[theorem]{Remark}
\newtheorem{example}[theorem]{Example}
\newtheorem*{exampleIntro}{Example}

\DeclareMathOperator{\Proj}{\textnormal{Proj}}
\DeclareMathOperator{\Spec}{\textnormal{Spec}}

\DeclareMathOperator{\Ht}{\textnormal{Ht}}

\newcommand{\GF}{\textnormal{GF}}
\newcommand{\kk}{\mathbb{K}}
\newcommand{\PP}{\mathbb{P}}
\newcommand{\QQ}{\mathbb{Q}}
\newcommand{\RR}{\mathbb{R}}

\newcommand{\ZZ}{\mathbb{Z}}

\newcommand{\In}[1]{\textnormal{in}_{#1}}

\newcommand{\sat}{\textnormal{sat}}

\newcommand{\Hilb}[2]{\textnormal{\bf Hilb}_{#1}^{#2}}
\newcommand{\St}[2]{\textnormal{\bf St}_{#1}^{#2}}
\newcommand{\Mf}[1]{\textnormal{\bf Mf}_{#1}}
\newcommand{\SI}[2]{\textnormal{\textsf{S}}_{#1}^{#2}}
\newcommand{\HS}[2]{\textnormal{\textsf{H}}_{#1}^{#2}}

\newcommand{\Sk}[2]{\mathscr{G}_{#1}^{#2}}
\newcommand{\DG}[3]{\mathscr{G}_{#1}^{#2}(#3)}

\newcommand{\MC}{\mathcal{M}\hspace{-0.75pt}\mathcal{C}}
\newcommand{\SC}{\mathcal{S}\mathcal{C}}

\newcommand{\xx}{{\mathbf{x}}}
\newcommand{\ff}{{\mathbf{f}}}
\newcommand{\aaa}{{\underline{a}}}
\newcommand{\bbb}{{\underline{b}}}
\newcommand{\ccc}{{\underline{c}}}
\newcommand{\ooo}{{\underline{\omega}}}

\newcommand{\comp}[1]{{#1}^{\textsf{c}}}

\newcommand{\eu}[1]{\textsf{e}^{+}_{#1}}
\newcommand{\ed}[1]{\textsf{e}^{-}_{#1}}

\newcommand{\bs}{\quad$\blacksquare$}
\newcommand{\ssucceq}{\succeq\hspace{-2.5pt}\succeq}

\newcommand{\ssucc}{\succ\hspace{-2.5pt}\succ}

\newcommand{\xminus}[2]{\hspace{#2}\raisebox{-4.6pt}{\tikz{\node at (0,0) []{${-}$}; \node at (0,0.15) []{\tiny ${#1}$};}}\hspace{#2}}

\begin{document}

\title{The Gr\"obner fan of the Hilbert scheme}
\author[Y.~Kambe]{Yuta Kambe}
\address{Graduate School of Science and Engineering \\ Saitama University \\ Shimo-Okubo 255\\ Sakura-ku \\ Saitama-shi \\ 338-8570 \\ Japan.}
\email{\href{mailto:y.kambe.021@ms.saitama-u.ac.jp}{y.kambe.021@ms.saitama-u.ac.jp}}
\urladdr{\url{https://sites.google.com/view/yuta-kambe/}}

\author[P.~Lella]{Paolo Lella}
\address{Dipartimento di Matematica\\ Politecnico di Milano\\ 
         Piazza Leonardo da Vinci 32\\ 20133 Milano\\ Italy.}
\email{\href{mailto:paolo.lella@polimi.it}{paolo.lella@polimi.it}}
\urladdr{\url{http://www.paololella.it/}}



\thanks{The first author is supported by JSPS Research Fellowships for Young Scientists (DC2). The second author is supported by MIUR funds FFABR-LELLA-2018 and is member of GNSAGA (INdAM, Italy)}

\begin{abstract}
We give a notion of \lq\lq combinatorial proximity\rq\rq~among strongly stable ideals in a given polynomial ring with a fixed Hilbert polynomial. We show that this notion guarantees \lq\lq geometric proximity\rq\rq~of the corresponding points in the Hilbert scheme. We define a graph whose vertices correspond to strongly stable ideals and whose edges correspond to pairs of adjacent ideals. Every term order induces an orientation of the edges of the graph. This directed graph describes the behavior of the points of the Hilbert scheme under Gr\"obner degenerations with respect to the given term order.

Then, we introduce a polyhedral fan that we call \emph{Gr\"obner fan of the Hilbert scheme}. Each cone of maximal dimension corresponds to a different directed graph induced by a term order. This fan encodes several properties of the Hilbert scheme. We use these tools to present a new proof of the connectedness of the Hilbert scheme. Finally, we improve the technique introduced in the paper \lq\lq Double-generic initial ideal and Hilbert scheme\rq\rq~\cite{BCR-GG} (Bertone, Cioffi, Roggero, Ann.~Mat.~Pura Appl.~(4) \textbf{196}(1), 19--41, 2017) to give a lower bound on the number of irreducible components of the Hilbert scheme.
\end{abstract}

\maketitle

\tableofcontents

\section{Introduction}

The Hilbert scheme $\Hilb{p(t)}{n}$, parametrizing subschemes of $\PP^n$ with Hilbert polynomial $p(t)$, has been intensively studied since its definition and proof of existence by Grothendieck \cite{Grothendieck}. Nevertheless, very few comprehensive properties are known and lots of natural questions are still open. Among the known results, we mention connectedness \cite{HartshorneThesis,PeevaStillman}, the smoothness of the lexicographic point \cite{ReevesStillman} and the existence of bound on the \lq\lq distance\rq\rq~between irreducible components \cite{Reeves}. 

The problem of understanding the topological structure of the Hilbert scheme is usually complicated due to its unpredictable and mysterious behavior. Questions such \lq\lq how many irreducible components are there in $\Hilb{p(t)}{n}$?\rq\rq, \lq\lq how are the irreducible components related?\rq\rq, \lq\lq are the irreducible components rational?\rq\rq~are in most cases without a complete answer. More is known about some particular Hilbert schemes or some special sub-loci. The case of punctual Hilbert schemes has been studied continuously since the 70s (see \cite{IarroOverview} and references therein), and it is still under investigation nowadays \cite{CEVV,Huibregtse,JJ1,JJ2,RamkumarSammartano}. In the case of 1-dimensional subschemes of the projective space $\PP^3$ there is a remarkable variety of results (for instance about ACM curves, see \cite{Ellingsrud,Treger,FloystadRoggero,BCR-GG}).

In this context, a classical approach consists in trying to rephrase a global question in terms of a local question for a few, possibly finite, number of points of $\Hilb{p(t)}{n}$. For instance, under the right conditions, the rationality of an irreducible component can be deduced by the smoothness of a special point lying on it \cite[Corollary 6.10]{LR}, \cite[Theorem 6]{BCR-GG}. An efficient way to accomplish this task is to consider Gr\"obner degenerations to monomial ideals and in particular to generic initial ideals. Indeed, on one hand each irreducible component and each intersection of irreducible components of $\Hilb{p(t)}{n}$ contains at least one point corresponding to a generic initial ideal. On the other hand, generic initial ideals are Borel-fixed ideal, i.e.~invariant under the action of the Borel subgroup of $\textnormal{GL}_\kk(n+1)$ consisting of upper triangular matrices. Furthermore, in characteristic 0, Borel-fixed ideals enjoy additional combinatorial properties. Hence, Borel-fixed ideals are well distributed throughout the Hilbert scheme and have special properties that make them extremely effective.

This paper is strongly influenced by the theory of Gr\"obner strata and marked families (see \cite{LR2} and references therein). Given a Borel-fixed ideal $J$ and a term order $\Omega$, the Gr\"obner stratum $\St{J}{\Omega}$ is the scheme parametrizing the family of ideals with initial ideal $J$ with respect to $\Omega$. The marked scheme $\Mf{J}$ is the scheme parametrizing the family of ideals whose quotient algebras have the set of monomials not contained in $J$ as basis. These two types of families are flat, so that Gr\"obner strata and marked schemes describe subsets of the Hilbert scheme. These families can be used to parametrize open subsets of $\Hilb{p(t)}{n}$ (or of one of its irreducible component) or sub-loci corresponding to schemes with special properties (such as Hilbert function, type of resolution, \ldots).

However, if one is interested in studying the irreducible components of $\Hilb{p(t)}{n}$, the set of Borel-fixed ideals turns out to be redundant, in a sense clarified by the following example. 

\begin{exampleIntro}
Consider the Hilbert scheme $\Hilb{6t-3}{3}$ parametrizing 1-dimensional subschemes of $\PP^3$ of degree $6$ and arithmetic genus $4$. There are 3 irreducible components: 
\begin{itemize}
\item the first component has dimension $48$ and the general element is the union of a plane curve of degree $6$ and 6 isolated points; 
\item the second component has dimension $32$ and the general element is the union of a plane quintic and a line intersecting in one point, and 2 isolated points; 
\item the third component has dimension $24$ and the general element is a complete intersection of a quadric surface and a cubic surface.
\end{itemize}
By the theory of marked families, in order to parametrize an open subset of each irreducible component, we need at most 3 Borel-fixed ideals. In $\Hilb{6t-3}{3}$ there are 31 points corresponding to Borel-fixed ideals to choose from (see Example \ref{ex:6t-3 part 1}), whose algebraic and geometric properties are very diverse. First, such points are not equally distributed along the irreducible components. In fact, most of them lie exclusively on the first irreducible component. Second, there are smooth points, singular points lying on a single component and singular points that are in the intersection of 2 irreducible components and that are smooth if we restrict to any of them. Third, these points have different behavior with respect to Gr\"obner degenerations (see Example \ref{ex:6t-3 part 2}). 
\end{exampleIntro}

Two natural questions arise.
\begin{enumerate}
\item[\it (Q1)] Assume that the topological structure of the Hilbert scheme and the distribution between components of points corresponding to Borel-fixed ideals are known. Which ones are better suited for effective investigation?
\item[\it (Q2)] Suppose that one knows nothing about the Hilbert scheme, but the list of Borel-fixed ideals defining points on it. Is it possible to deduce information about the topological structure of $\Hilb{p(t)}{n}$?
\end{enumerate}

These two problems are discussed in the inspiring paper \lq\lq Double-generic initial ideal and Hilbert scheme\rq\rq~\cite{BCR-GG} by Bertone, Cioffi and Roggero. The double-generic initial ideal is a Borel-fixed ideal associated to an irreducible component of $\Hilb{p(t)}{n}$. Intuitively, it is the generic initial ideal of the ideal describing the generic element of the component. Hence, choosing the double-generic initial ideal among Borel-fixed ideals lying on a given component is a reasonable and natural option to answer \textit{(Q1)}. Still, there are some difficulties. First of all, the double-generic initial ideal is not intrinsically determined by an irreducible component, but it depends on the term order. Secondly, if we do not know a priori the list of Borel-fixed ideals defining points on a given irreducible components, we might not be able to detect the corresponding double-generic initial ideal with respect to a fixed term order (this makes it difficult to answer \textit{(Q2)}).

The definition of the double-generic initial ideal is based on a careful analysis of the action of the linear group on the generators of an ideal defining a point on the Hilbert scheme. Instead of the standard action of $\textnormal{GL}_\kk(n+1)$ on $\kk[x_0,\ldots,x_n]$ used for defining the generic initial ideal (see \cite[Chapter 15]{Eisenbud}), in \cite{BCR-GG} the group $\textnormal{GL}_{\kk}(n+1)$ acts on the elements $f_1 \wedge \cdots \wedge f_q$ of the exterior algebra $\bigwedge^q \kk[x_0,\ldots,x_n]_r$, where $\{f_1,\ldots,f_q\}$ is a basis of the homogeneous piece $I_r$ of an ideal $I$ defining a point on $\Hilb{p(t)}{n}$ for a sufficiently large $r$.

In this paper, we present a different approach based on the study of the combinatorial properties of Borel-fixed ideals. In particular, the combinatorics allow to better understand the behavior of the points of the Hilbert scheme under Gr\"obner degenerations (and thus also the dependence of double-generic initial ideal on the term order). We begin by studying the relative position of points corresponding to Borel-fixed ideals in the Hilbert scheme.

\begin{theoremIntro}[Definition \ref{def:borelAdjacent} and Theorem \ref{thm:mainDef}]
Let $J,J'\subset \kk[x_0,\ldots,x_n]$ be two saturated Borel-fixed ideals defining points on $\Hilb{p(t)}{n}$ and denote by $\mathfrak{J}$ and $\mathfrak{J}'$ the monomial bases of $J_r$ and $J_r'$ for $r$ sufficiently large. If the monomials in the sets $\mathfrak{J}\setminus\mathfrak{J}'$ and $\mathfrak{J}'\setminus \mathfrak{J}$ have the same linear syzygies, then there is a rational curve on $\Hilb{p(t)}{n}$ passing through the points defined by $J$ and $J'$, so that these points lie on a common irreducible component.
\end{theoremIntro}

As a consequence of this result, we introduce the Borel graph of $\Hilb{p(t)}{n}$ (Definition \ref{def:BorelGraph}) whose vertices correspond to Borel-fixed ideals and whose edges correspond to unordered pairs of ideals satisfying the hypothesis of the previous theorem. We underline that the rational curve passing through two Borel-fixed points is in fact the closure of a one-dimensional orbit of the action on $\Hilb{p(t)}{n}$ of the standard torus $T = (\kk^\ast)^{n+1}$ of $\PP^n$ (Remark \ref{rk:T-orbit}). Hence, the Borel graph turns out to be a subgraph of the $T$-graph of $\Hilb{p(t)}{n}$ \cite{AltmannSturmfels,HeringMaclagan} whose vertices correspond to monomial ideals and whose edges correspond to unordered pairs of ideals contained in the closure of a one-dimensional $T$-orbit (Remark \ref{rk:T-graph}).

Any term order induces an orientation of the edges of the Borel graph. We call degeneration graphs the directed graphs supported on the Borel graph induced by a term order. The name is motivated by the fact that this type of graphs encodes the behavior of the points in the neighborhood of a Borel-fixed ideal with respect to Gr\"obner degenerations (Proposition \ref{prop:grobnerStratumClosure}).

Then, we classify all the possible degeneration graphs, by means of a polyhedral fan that we call Gr\"obner fan of the Hilbert scheme (Definition \ref{def:groebnerFan} and Theorem \ref{thm:groebnerFan}). Each cone of maximal dimension corresponds to a different directed degeneration graph where the orientation of the edges is induced by some term order. Cones of lower dimension correspond to mixed graphs, where the orientation of the edges is induced by weight orders on the monomials.

For several degeneration graphs, we are able to construct a minimum spanning tree. This implies that the Borel graph is a connected graph (Corollary \ref{cor:spanning tree}) and gives a new strategy to prove the connectedness of the Hilbert scheme (see proofs of Hartshorne \cite{HartshorneThesis} and Peeva, Stillman \cite{PeevaStillman}).

\begin{theoremIntro}[Theorem \ref{thm:chainConnected}]
The Hilbert scheme is rationally chain connected.
\end{theoremIntro}

In the degeneration graphs having a minimum spanning tree, there is a unique vertex with no incoming edge. Typically, this is not the case. Rather the number of vertices with no incoming edge in a degeneration graph can give interesting information about the topological structure of the Hilbert scheme (answering \textit{(Q2)}). Exploiting again properties of double-generic initial ideals (see \cite[Proposition 9]{BCR-GG}), we can give the following lower bound on the number of irreducible components of the Hilbert scheme.

\begin{theoremIntro}[Proposition \ref{prop:gg} and Conjecture \ref{conj:v1}] The number of irreducible components of $\Hilb{p(t)}{n}$ is at least the maximum number of vertices with no incoming edge in any degeneration graph.
\end{theoremIntro}

In order to obtain the best estimate, one has to examine a finite number of degeneration graphs, one for each cone of maximal dimension of the Gr\"obner fan. For instance, in the case of the Hilbert scheme $\Hilb{6t-3}{3}$, the Gr\"obner fan has 268 cones of maximal dimension and the maximum number of vertices with no incoming edge in a degeneration graph is 3. Hence, in this case our method detects all the irreducible components of the Hilbert scheme and it also suggests three Borel-fixed ideals to consider to parametrize the components via marked families.

\medskip

\paragraph{\bf Organization}  In Section \ref{sec:preliminaries}, we discuss preliminaries about Hilbert schemes and Borel-fixed ideals in characteristic 0. In Section \ref{sec:Borel deformations}, we introduce a notion of combinatorial proximity of two Borel-fixed ideals with the same Hilbert polynomial and we show that it corresponds to geometric proximity on the Hilbert scheme. In Section \ref{sec:Groebner fan}, we classify the behavior of the points of the Hilbert scheme with respect to Gr\"obner degenerations by means of a polyhedral fan. In Section \ref{sec:applications}, we exploit the Gr\"obner fan to prove that the Hilbert scheme is rationally chain connected and to give an efficient method to compute a lower bound on the number of irreducible components of the Hilbert scheme.

\smallskip

\paragraph{\bf Software} We implemented the algorithms for using the tools developed in the paper in the \textit{Macaulay2} package \texttt{GroebnerFanHilbertScheme.m2}. The package is available at the web page \href{http://www.paololella.it/publications/kl/}{\tt www.paololella.it/publications/kl/} with a second file containing the scripts for computing the examples of the paper.

\smallskip

\paragraph{\bf Acknowledgements} We thank the referee for the careful reading of our manuscript and for valuable suggestions.

\section{Preliminaries}\label{sec:preliminaries}

Let $\kk[\xx]:=\kk[x_0,\ldots,x_n]$ be a polynomial ring in $n+1$ with coefficient in an algebraically closed field $\kk$ of characteristic 0. We denote by $\mathbb{T}^n$ the set of monomials of $\kk[\xx]$ and we describe them with the standard multi-index notation; namely, for any $\aaa =(a_0,\ldots,a_n)\in \ZZ_{\geqslant 0}^{n+1}$, $\xx^{\aaa}$ stands for $x_0^{a_0} \cdots x_n^{a_n}$. Whenever the multi-index $\aaa$ is in $ \ZZ^{n+1}$, $\xx^\aaa$ stands for the generalized monomial in $\kk(\xx) := \text{Frac}(\kk[\xx])$. We denote the set of generalized monomial by $\widehat{\mathbb{T}}^n$.

We think of $\kk[\xx]$ as the coordinate ring of the projective space $\PP^n = \Proj \kk[\xx]$. We consider the standard grading on $\kk[\xx]$ and we denote by $\vert \aaa\vert = a_0 + \cdots + a_n$ the total degree of a monomial $\xx^\aaa$. Given a positive integer $m$, we denote by  $\mathbb{T}^n_m$ the set of monomials of degree $m$, by $\kk[\xx]_m$ the homogeneous piece of degree $m$ of $\kk[\xx]$ and by $\kk[\xx]_{\geqslant m}$ the direct sum $\bigoplus_{t \geqslant m} \kk[\xx]_{t}$. Every ideal $I\subset \kk[\xx]$ is always assumed to be homogeneous, $I_m = I \cap \kk[\xx]_m$ denotes the homogeneous piece of degree $m$ and $I_{\geqslant m} = I \cap \kk[\xx]_{\geqslant m}$ denotes the truncated ideal in degree $m$.

\smallskip

For a subscheme $X \subset \PP^n$, we denote by $I_X \subset \kk[\xx]$ the unique saturated ideal such that $X = \Proj \kk[\xx]/I_X$ and by $p_X(t)$ its \emph{Hilbert polynomial}, that is the unique numerical polynomial such that $p_X(t) = \dim_{\kk} (\kk[\xx]/I_X)_t = \dim_{\kk} \kk[\xx]_t/(I_X)_t$ for $t$ large enough. By a little abuse of notation, we refer to the Hilbert polynomial $p_I(t)$ of an ideal $I$ as the Hilbert polynomial of its quotient ring $\kk[\xx]/I$. We refer to the unique numerical polynomial $q_I(t)$ such that $\dim_{\kk} I_{t} = q_I(t), t \gg 0$  as \emph{volume polynomial} of the ideal $I$. By definition, $q_I(t) = \binom{t+n}{n} - p_I(t)$ for $t$ sufficiently large.

Given a Hilbert polynomial $p(t) \in \QQ[t]$, we study the Hilbert scheme $\Hilb{p(t)}{n}$ representing the contravariant Hilbert functor $\underline{\mathbf{Hilb}}_{p(t)}^n: (\text{$\kk$-schemes})^\circ \to (\text{Sets})$. This functor associates to a scheme $Z$ over $\kk$ the set
\[
\underline{\mathbf{Hilb}}_{p(t)}^n(Z) = \left\{\begin{tikzcd}[row sep=tiny,cells={inner ysep=2pt}] Y\arrow[rd] \arrow[r,hook] & \PP^n \times_\kk Z \arrow[d] \\ & Z \end{tikzcd} \ \middle\vert\ \begin{array}{l}Y \to Z \text{~flat morphism whose fibers over}\\ \text{points have Hilbert polynomial~}p(t) \end{array} \right\}
\]
and to a morphism of schemes $f: X \to Z$ the map $\underline{\mathbf{Hilb}}_{p(t)}^n(f): \underline{\mathbf{Hilb}}_{p(t)}^n(Z) \to \underline{\mathbf{Hilb}}_{p(t)}^n(X)$
\[
Y \to Z \in \underline{\mathbf{Hilb}}_{p(t)}^n(Z)\qquad \longmapsto\qquad Y \times_{Z} X \to X \in \underline{\mathbf{Hilb}}_{p(t)}^n(X).
\]
For all schemes $Z$, there is a 1-to-1 correspondence between the set $\underline{\mathbf{Hilb}}_{p(t)}^n(Z)$ and the set of morphisms $\textnormal{Mor}(Z,\Hilb{p(t)}{n})$ from $Z$ to the Hilbert scheme. For a scheme $X \in \underline{\mathbf{Hilb}}_{p(t)}^n(\Spec \kk)$, we denote by $[X] \in \Hilb{p(t)}{n}$ the corresponding $\kk$-rational point (the image of the corresponding morphism $\Spec \kk \to \Hilb{p(t)}{n}$).

The Hilbert functor has been introduced by Grothendieck  \cite{Grothendieck}, who first proved its representability. The Hilbert scheme is classically constructed as a subscheme of a suitable Grassmannian and eventually as subscheme of a projective space through the corresponding Pl\"ucker embedding. We recall briefly the idea of the construction, because it motivates the setting of this paper (for more details see \cite{BayerThesis,IarrobinoKanev,HaimanSturmfels,BLMR}).

Every Hilbert polynomial $p(t)$ has a unique decomposition as finite sum of binomial coefficients
\begin{equation*}
p(t) = \tbinom{t+a_1}{a_1} + \tbinom{t+a_2-1}{a_2} + \cdots + \tbinom{t+{a_i}-i+1}{a_i} + \cdots + \tbinom{t+a_r-r+1}{a_r},\qquad a_1 \geqslant \cdots \geqslant a_r \geqslant 0.
\end{equation*}
The first coefficient $a_1$ equals the degree of $p(t)$, i.e.~the dimension of the schemes parametrized by $\Hilb{p(t)}{n}$, and the number of summands $r$ is called \emph{Gotzmann number} of $p(t)$. Gotzmann's Regularity Theorem \cite{Gotzmann} says that the saturated ideal $I_X$ of a scheme $[X] \in \Hilb{p(t)}{n}$ is $r$-regular, so that we can associated to every scheme $X \subset \PP^n$ with Hilbert polynomial $p(t)$ the vector space $\kk[\xx]_r/(I_X)_r$ of dimension $p(r)$ (or equivalently the vector space $(I_X)_r$ of dimension $q(r)$). This result explains the idea of embedding $\Hilb{p(t)}{n}$ in the Grassmannian $\mathbf{Gr}(p(r),\kk[\xx]_r)$ of $p(r)$-dimensional quotients of the vector space $\kk[\xx]_r$. The closed condition describing the Hilbert scheme as subscheme of the Grassmannian is given by a second crucial result by Gotzmann.  Gotzmann's Persistence Theorem \cite{Gotzmann} states that an ideal $I$, generated by polynomials of degree $r$ and such that $\kk[\xx]_r/I_r$ has dimension $p(r)$, has Hilbert polynomial $p(t)$ if, and only if, the quotient $\kk[\xx]_{r+1}/I_{r+1}$ has dimension $p(r+1)$.

\smallskip

Lots of investigations about Hilbert schemes are conducted with the help of the theory of Gr\"obner bases (and generalizations). In fact, the procedure of associating to any ideal $I\subset \kk[\xx]$ the initial ideal $\text{in}_{\Omega}(I)$ (for some term order $\Omega$) can be described in terms of a flat family over the affine line $\mathbb{A}^1$ (see for instance \cite[Theorem 15.17]{Eisenbud}). The generic fiber is projectively equivalent to $I$, while the special fiber is $\text{in}_{\Omega}(I)$. 

 When working with term orders and initial ideals, we need to fix an order on variables. We use the order $x_0 <  \cdots < x_n$, so that the minimum index $\min \xx^\aaa$ and maximum index $\max \xx^\aaa$ of a variable appearing in a monomial $\xx^\aaa$ correspond to the minimum and maximum variables. As the orders we consider on monomials have to be multiplicative orders, the choice $x_0 < \cdots < x_n$ induces a partial order on the set of monomials of a given degree:
 \[
 x_i > x_j \qquad\Longrightarrow\qquad \xx^\aaa = x_i \cdot \xx^{\underline{c}} > x_j \cdot \xx^{\underline{c}} = \xx^\bbb.
 \] 
 We refer to this order as \emph{Borel order} and we denote it by $\geq_B$. Each graded term order is a refinement of $\geq_B$.
 
 \begin{definition}\label{def elementary move}
For $i<n$ and $j>0$, we define the $i$-th increasing elementary move  and the $j$-th decreasing elementary move as the maps
\begin{equation}\label{eq:elemMoves}
\begin{array}{r c c c}
\eu{i}: & \widehat{\mathbb{T}}^n & \to & \widehat{\mathbb{T}}^n \\
& \xx^\aaa & \mapsto & \frac{x_{i+1}}{x_i} \xx^\aaa
\end{array}
\qquad\text{and}\qquad
\begin{array}{r c c c}
\ed{j}: & \widehat{\mathbb{T}}^n & \to & \widehat{\mathbb{T}}^n \\
& \xx^\aaa & \mapsto & \frac{x_{j-1}}{x_j} \xx^\aaa
\end{array}.
\end{equation}
We say that an elementary move is admissible for a monomial $\xx^\aaa \in \mathbb{T}^n$ if also the image is a monomial in $\mathbb{T}^n$. Compositions $\eu{i} \circ \ed{i+1}$ and $\ed{j}\circ \eu{j-1}$ give the identity $\textsf{id}:  \widehat{\mathbb{T}}^n \to  \widehat{\mathbb{T}}^n$.
\end{definition}
 
We can interpret the Borel order $\geq_B$ as the transitive closure of the relations
\[
 \xx^\aaa >_B \xx^\bbb \qquad \Longleftrightarrow \qquad \xx^\aaa = \eu{i}(\xx^\bbb), \text{~for some $i$},
\]
and use these elementary relations to visualize the order among monomials (see Figure \ref{fig:BorelOrder3VarDeg3}). By definition $\xx^\aaa >_B \xx^\bbb$ means that there is sequence of (admissible) elementary moves $\eu{i_1},\ldots,\eu{i_s}$ such that
\[
\begin{split}
 \xx^\aaa = \eu{i_1}(\xx^{\ccc_1}) &{}=   \eu{i_1}\circ \eu{i_2}(\xx^{\ccc_2}) = \ldots = \eu{i_1}\circ \cdots \circ \eu{i_{s-1}}(\xx^{\ccc_{s-1}}) =  \eu{i_1}\circ \cdots \circ \eu{i_{s}}(\xx^\bbb)\\
&{} \Rightarrow\quad \xx^\aaa = \frac{x_{i_1+1}}{x_{i_1}}\cdot\ldots\cdot\frac{x_{i_s+1}}{x_{i_s}}\xx^\bbb.
 \end{split}
\]
Even though the product in $\kk(\xx)$ is commutative, we notice that if we change the order of application of the elementary moves we may lose the admissibility at each step. Next lemma shows that a composition of moves that is overall admissible for $\xx^\bbb$ can be always decomposed in a composition of moves admissible at each step.

For a monomial $\xx^\aaa \in \mathbb{T}^n$, we denote by $\vert \aaa \vert_i$ the sum $a_i + \cdots + a_n$, i.e.~the degree of the part of $\xx^\aaa$ in $\kk[x_i,\ldots,x_n]$. Obviously, $\vert \aaa \vert_0 = \vert \aaa \vert$.
\begin{lemma}\label{lem:BorelOrder}
Let $\xx^\aaa$ and $\xx^\bbb$ be two monomials in $\mathbb{T}^n$.
\begin{equation}
\xx^\aaa \geq_B \xx^\bbb \qquad\Longleftrightarrow\qquad \vert \aaa \vert_i \geqslant \vert \bbb \vert_i,~\forall\ i=0,\ldots,n.
\end{equation}
\end{lemma}
\begin{proof}
($\Rightarrow$) $\xx^\aaa \geq_B \xx^\bbb$ implies that $\xx^\aaa = (\frac{x_1}{x_0})^{c_1} \cdots (\frac{x_{n}}{x_{n-1}})^{c_{n}} \xx^\bbb = x_0^{-c_1} x_1^{c_1-c_2} \cdots x_{n-1}^{c_{n-1}-c_{n}} x_n^{c_{n}} \xx^\bbb$, so that
\[
\begin{split}
 \vert \aaa \vert_0 &{} = \vert \aaa \vert = \vert \bbb \vert = \vert \bbb \vert_0,\\
\vert \aaa \vert_i & {} = (c_i - c_{i+1}) + (c_{i+1} - c_{i+2}) + \cdots + (c_{n-1}-c_{n}) + c_{n} + \vert \bbb \vert_i = {} \\
&{} = c_{i} + \vert \bbb \vert_i \geqslant \vert \bbb \vert_i, \quad i = 1,\ldots,n.
\end{split}
\]
($\Leftarrow$) We have $\xx^\aaa = \textsf{E} (\xx^\bbb)$, where $\textsf{E}$ is the composition of elementary moves
\[
\underbrace{\eu{n-1} \circ \cdots \circ \eu{n-1}}_{\vert \aaa \vert_n - \vert \bbb \vert_n\text{~times}} \circ  \quad \cdots \quad \circ \underbrace{\eu{1} \circ \cdots \circ \eu{1}}_{\vert \aaa \vert_2 - \vert \bbb \vert_2\text{~times}} \circ \underbrace{\eu{0} \circ \cdots \circ \eu{0}}_{\vert \aaa \vert_1 - \vert \bbb \vert_1\text{~times}} . \qedhere
\]
\end{proof}

\begin{example}
Consider the monomials $x_0^2 x_1 = \xx^{(2,1,0)}$, $x_0 x_2^2 = \xx^{(1,0,2)}$ and $x_1^3 = \xx^{(0,3,0)}$ in $\mathbb{T}^2_3$. We have
\[
\vert (1,0,2) \vert_0 = \vert (2,1,0) \vert_0 = 3,\ \vert (1,0,2) \vert_1 = 2 > 1 =  \vert (2,1,0) \vert_1,\  \vert (1,0,2) \vert_2 = 2 > 0 =  \vert (2,1,0) \vert_2,
\]
 so that $ x_0 x_2^2 >_B x_0^2 x_1$ and $ x_0 x_2^2 = \eu{1} \circ \eu{1} \circ \eu{0} (x_0^2 x_1)$, while $x_0 x_2^2$ and $x_1^3$ are not comparable with respect to $\geq_B$, as
\[
\vert (1,0,2) \vert_0 = \vert (0,3,0) \vert_0 = 3,\ \vert (1,0,2) \vert_1 =  2 < 3 =  \vert (0,3,0) \vert_1,\  \vert (1,0,2) \vert_2 = 2 > 0 =  \vert (0,3,0) \vert_2.
\]
\end{example}

In the context of Hilbert schemes we are particularly interested in \emph{generic initial ideals}, that is initial ideal in generic coordinates.  Galligo \cite{Galligo} proved that generic initial ideals are monomial ideals fixed by the action of the Borel subgroup of upper triangular matrices of the projective linear group and thus called \emph{Borel-fixed} ideals. When the characteristic of the base field is $0$, the notion of Borel-fixed ideal coincides with the notion of strongly stable ideal. This type of ideals is characterized by the following combinatorial property.

\begin{definition}\label{def strongly stable}
An ideal $J \subset \kk[\xx]$ is called \emph{strongly stable} if
\begin{itemize}
\item[(1)] $J$ is a monomial ideal;
\item[(2)] $\xx^{\aaa} \geq_B \xx^\bbb$ and $\xx^\bbb \in J$ imply $\xx^{\aaa} \in J$.
\end{itemize}
\end{definition}

By the definition, the set of monomials of degree $m$ of a strongly stable ideal $J$ is a subset of $\mathbb{T}^n_m$ closed with respect to increasing elementary moves. Such a set is often call \emph{Borel set} of $\mathbb{T}^n_m$. From now on, when considering a strongly stable ideal, we focus on the set of monomials of degree equal to the Gotzmann number of its Hilbert polynomial. This set plays a crucial role throughout the paper, so that we introduce some special notation. 
We write in superscript \lq\lq$\sat$\rq\rq~to denote a \emph{saturated} strongly stable ideal and, given a saturated ideal $J^\sat$, we denote with $J$ (same letter, no superscript) the truncation $J^\sat_{\geqslant r}$, where $r$ is the Gotzmann number of the Hilbert polynomial of $J^\sat$. Furthermore, given a saturated ideal $J^\sat$ or its truncation $J$, we denote with the same letter in fraktur alphabet $\mathfrak{J}$ the set of its monomials of degree $r$, i.e.~$J = (\mathfrak{J})$.

 For any set $\mathfrak{A}$, we denote by $\vert \mathfrak{A} \vert$ its cardinality and for any pair of sets $\mathfrak{A},\mathfrak{B}$, we write $\mathfrak{A}\setminus\mathfrak{B}$ meaning $\mathfrak{A} \setminus (\mathfrak{A}\cap\mathfrak{B})$. For a subset $\mathfrak{A} \subset \mathbb{T}^n_m$, we consider the partition $\mathfrak{A}_0 \sqcup \cdots \sqcup \mathfrak{A}_n$, where $\mathfrak{A}_i = \{ \xx^\aaa \in \mathfrak{A}\ \vert\ \min \xx^\aaa = i\}$, and $\mathfrak{A}_{\geqslant i}$ stands for the set $\mathfrak{A}_i \sqcup \cdots \sqcup \mathfrak{A}_n = \{\xx^\aaa \in \mathfrak{A}\ \vert\ \min \xx^\aaa \geqslant i\}$. Moreover,  we denote by $\comp{\mathfrak{A}}$ the complementary set $\mathbb{T}^n_m\setminus \mathfrak{A}$.

We briefly recall the deep relation between the combinatorics of a strongly stable ideal and its Hilbert polynomial (see \cite{Mall1,BigattiRobbiano,Efficient,AL} for details). From now on, $r$ is for the Gotzmann number of the Hilbert polynomial of any strongly stable ideal $J$ we consider. We denote the volume polynomial of $J$ by $q(t)$. The set $\mathfrak{J}$ is a basis of the vector space $J_r$, i.e.~it consists of $q(r)$ distinct monomials of degree $r$. For any $m > r$, the monomial basis of $J_m$ can be decomposed as follows 
\[
( \mathfrak{J}_n \cdot \kk[\xx]_{m-r}) \sqcup (\mathfrak{J}_{n-1} \cdot \kk[x_0,\ldots,x_{n-1}]_{m-r})  \sqcup \cdots \sqcup (\mathfrak{J}_1 \cdot \kk[x_0,x_1]_{m-r})\sqcup( \mathfrak{J}_0\cdot \kk[x_0]_{m-r}),
\]
where $\mathfrak{J}_i \cdot \kk[x_0,\ldots,x_i]_{m-r}$ stands for the set of monomials $\xx^\aaa \cdot \xx^\ccc$ with $\xx^\ccc \in \kk[x_0,\ldots,x_i]_{m-r}$ and $\xx^\aaa \in \mathfrak{J}_i$. Consequently, one has
\[
q(t) = \sum_{i=0}^n \vert \mathfrak{J}_i\vert \binom{i + t - r}{i} \quad\Longrightarrow\quad p(t) = \binom{n+t}{n} - \sum_{i=0}^n \vert \mathfrak{J}_i\vert \binom{i + t - r}{i}.
\]
 Furthermore, set $\Delta^0 p(t) = p(t)$ and $\Delta^{k} p(t) =\Delta^{k-1} p(t) - \Delta^{k-1} p(t-1)$ for $k>0$, one deduces
 \[
\vert\mathfrak{J}_{\geqslant i} \vert = \sum_{k=i}^n \vert \mathfrak{J}_k \vert = \binom{n+r-i}{n-i} - \Delta^i p(r) \quad\Longrightarrow\quad \vert \mathfrak{J}_i\vert = \binom{n+r-i-1}{n-i-1} - \Delta^i p(r) + \Delta^{i+1}p(r).
 \]
Hence, for any pair of strongly stable ideals $J,J'\subset \kk[\xx]$ with Hilbert polynomial $p(t)$, it holds $\vert \mathfrak{J}_i\vert = \vert \mathfrak{J}_i'\vert$ for all $i=0,\ldots,n$.  This property has been used for designing the algorithm computing the set of saturated strongly stable ideals in $\kk[\xx]$ with a given Hilbert polynomial introduced in \cite{CLMR} and improved in \cite{Efficient,AL}. Another algorithm was known since \cite{Reeves} and has been taken up more recently in \cite{MN}. We denote by $\SI{p(t)}{n}$ the set of strongly stable ideals $J = (\mathfrak{J}) \subset \kk[x_0,\ldots,x_n]$ with Hilbert polynomial $p(t)$.

 \begin{example}\label{ex:borelSets}
 Consider the saturated strongly stable ideal $J^{\sat} = (x_2^2,x_1x_2,x_1^2) \subset \kk[x_0,x_1,x_2]$. Its Hilbert polynomial is $p(t)=3$ with Gotzmann number $3$. In Figure \ref{fig:BorelOrder3VarDeg3}, there is the subset $\mathfrak{J} \subset \mathbb{T}^2_3$.  As $\Delta^i p(t) = 0$, for all $i>0$, we have
 \[
 \vert \mathfrak{J}_0 \vert = \binom{4}{2} - 3 = 3,\quad  \vert \mathfrak{J}_1  \vert = \binom{3}{1} = 3,\quad  \vert \mathfrak{J}_2  \vert = \binom{2}{0} = 1.
 \]
 There is a second saturated strongly stable ideal with Hilbert polynomial $p(t)=3$, the lexicographic ideal $L^\sat = (x_2,x_1^3)$. In this case, we have
 \[
 \mathfrak{L}_0 = \left\{ x_0^2 x_2, x_0 x_1 x_2, x_0 x_2^2\right\}, \quad \mathfrak{L}_1 = \{ x_1 x_2^2, x_1^2 x_2 , x_1^3\},\quad \mathfrak{L}_2 = \{x_2^3\}.
  \]
  \end{example}

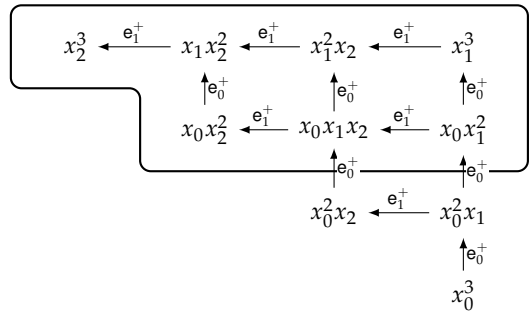
\begin{figure}[!ht]
\begin{center}
\begin{tikzpicture}[xscale=1.7,yscale=-1.1]

\draw [thick,rounded corners] (-0.5,-0.5) -- (3.5,-0.5) -- (3.5,1.5) -- (0.5,1.5) -- (0.5,0.5) -- (-0.5,0.5) -- cycle;

\node (003) at (0,0) [] {\footnotesize$x_2^3$};
\node (012) at (1,0) [] {\footnotesize$x_1x_2^2$};
\node (021) at (2,0) [] {\footnotesize$x_1^2x_2$};
\node (030) at (3,0) [] {\footnotesize$x_1^3$};

\draw [latex-] (003) --node[above,inner sep=1pt]{\tiny $\eu{1}$} (012);
\draw [latex-] (012) --node[above,inner sep=1pt]{\tiny $\eu{1}$} (021);
\draw [latex-] (021) --node[above,inner sep=1pt]{\tiny $\eu{1}$} (030);

\node (102) at (1,1) [] {\footnotesize$x_0 x_2^2$};
\node (111) at (2,1) [] {\footnotesize$x_0 x_1 x_2$};
\node (120) at (3,1) [] {\footnotesize$x_0 x_1^2$};

\draw [latex-] (012) --node[right,inner sep=1pt]{\tiny $\eu{0}$} (102);
\draw [latex-] (021) --node[right,inner sep=1pt]{\tiny $\eu{0}$} (111);
\draw [latex-] (030) --node[right,inner sep=1pt]{\tiny $\eu{0}$} (120);

\draw [latex-] (102) --node[above,inner sep=1pt]{\tiny $\eu{1}$} (111);
\draw [latex-] (111) --node[above,inner sep=1pt]{\tiny $\eu{1}$} (120);

\node (201) at (2,2) [] {\footnotesize$x_0^2  x_2$};
\node (210) at (3,2) [] {\footnotesize$x_0^2 x_1$};

\draw [latex-] (111) --node[right,inner sep=0pt,fill=white,xshift=1pt]{\tiny $\eu{0}$} (201);
\draw [latex-] (120) --node[right,inner sep=0pt,fill=white,xshift=1pt]{\tiny $\eu{0}$} (210);

\draw [latex-] (201) --node[above,inner sep=1pt]{\tiny $\eu{1}$} (210);

\node (300) at (3,3) [] {\footnotesize$x_0^3$};

\draw [latex-] (210) --node[right,inner sep=1pt]{\tiny $\eu{0}$} (300);

\end{tikzpicture}
\caption{The Borel order $\geq_B$ on the set of monomials $\mathbb{T}^2_{3}$ and the Borel set $\mathfrak{J}$ corresponding to the ideal $J^\sat = (x_2^2, x_1x_2, x_1^2) \subset \kk[x_0,x_1,x_2]$.}
\label{fig:BorelOrder3VarDeg3}
\end{center}
\end{figure}

Each component and each intersection of components of the Hilbert scheme contains at least a point corresponding to a scheme $\Proj \kk[\xx]/J$ defined by a strongly stable ideal $J$. For this reason, it has been natural to look for flat families of ideals \lq\lq centered\rq\rq~at a strongly stable ideal to study the Hilbert scheme. In this context, a key notion is that of \emph{marked family} of ideals (see \cite{CR,BCLR,LR} and references therein for a detailed treatment of the topic). Given a strongly stable ideal $J = (\mathfrak{J})$ generated in degree $r$ equal to the Gotzmann number of its Hilbert polynomial, a \emph{monic reduced $J$-marked set} is a set of polynomials of the shape
\begin{equation}\label{eq:markedset}
\left\{ \mathbf{f}_{\aaa} := \xx^\aaa + \sum_{\xx^\bbb \in \comp{\mathfrak{J}}} c_{\aaa,\bbb}\, \xx^\bbb \ \middle\vert\ \xx^\aaa \in \mathfrak{J},\ c_{\aaa,\bbb}\in \kk \right\}.
\end{equation}
Each polynomial $\ff_{\aaa}$ in the collection contains only the monomial $\xx^\aaa$ belonging to $J$. Such monomial has to be monic, it is called \emph{head term} of $\mathbf{f}_{\aaa}$ and it is denoted by $\Ht(\ff_{\aaa})$. 
This set of polynomials resembles a reduced Gr\"obner basis, but we underline  that in general the marking is not given by a term order, i.e.~$\Ht(\ff_{\aaa})$ might not be the leading term with respect to any term order.

Among all the $J$-marked sets, we are interested in those defining ideals sharing properties with the fixed monomial ideal $J$ (as in the case of a Gr\"obner basis and the corresponding initial ideal). A marked set $F$ is called \emph{marked basis} if the monomials of degree $m$ not contained in $J$ form a basis of the vector space $\kk[\xx]_m/(F)_m$ for all $m \geqslant r$. In particular, the ideal defined by a $J$-marked basis has the same Hilbert polynomial of $J$.

\begin{proposition}[{\cite[Theorem 2.11]{LR2}}]\label{prop:liftSyzygies}
Given a strongly stable ideal $J = (\mathfrak{J}) \subset \kk[\xx]$, a $J$-marked set $F$ is a $J$-marked basis if, and only if, all syzygies among monomials in $\mathfrak{J}$ lift to syzygies among polynomials in $F$.
\end{proposition}

Obviously, we can restrict to a basis of the syzygies of $J$ and since we are dealing with strongly stable ideals, it is natural to look at the Eliahou-Kervaire syzygies \cite{EK}. Furthermore, the ideal $J = (\mathfrak{J})$ is generated in degree $r$ and $r$-regular, so that the Eliahou-Kervaire syzygies of $J$ are linear.
Let $\xx^{\bbb} \in \mathfrak{J}$ be a generator of $J$ with $\min \xx^{\bbb} = h$. In the Eliahou-Kervaire resolution of $J$, $\xx^{\bbb}$ appears in syzygies of the type
\begin{equation*}
x_i \cdot \xx^{\bbb} - x_h \cdot \xx^{\aaa} = 0,\qquad i = h+1,\ldots,n.
\end{equation*}
Notice that
\[
 \xx^{\aaa}= \frac{x_i}{x_h}   \xx^{\bbb} =  \frac{x_i}{x_{i-1}} \cdots \frac{x_{h+1}}{x_h}   \xx^{\bbb}\quad\Rightarrow\quad   \xx^{\aaa} >_B  \xx^{\bbb}.
\]

The assumption that the head term is monic makes natural to extend the definition of marked set and marked basis to polynomial rings $A[\xx]$ with coefficient in any $\kk$-algebra $A$. Given a strongly stable ideal $J = (\mathfrak{J})$, we define the covariant \emph{marked family functor} 
\[
\underline{\mathbf{Mf}}_{J}: (\text{$\kk$-Algebras}) \to (\text{Sets}).
\] This functor associates to a $\kk$-algebra $A$ the family of ideals in $A[\xx]$ generated by a $J$-marked basis
\[
\underline{\mathbf{Mf}}_{J}(A) = \left\{ I = (F) \in A[\xx]\ \middle\vert\ F \text{~is a $J$-marked basis}\right\}
\]
and to a morphism of $\kk$-algebras $f: A \to B$ the map $\underline{\mathbf{Mf}}_{J}(f): \underline{\mathbf{Mf}}_{J}(A) \to \underline{\mathbf{Mf}}_{J}(B)$
\[
(I\subset A[\xx]) \in \underline{\mathbf{Mf}}_{J}(A)\qquad \longmapsto\qquad (I\otimes_A B  \subset B[\xx] )  \in \underline{\mathbf{Mf}}_{J}(B).
\]
The functor $\underline{\mathbf{Mf}}_J$ is representable \cite[Theorem 2.6]{LR2} and the representing scheme is called \emph{$J$-marked scheme} and denoted by $\mathbf{Mf}_J$.  Moreover, the inclusion $\underline{\mathbf{Mf}}_J \to \underline{\mathbf{Hilb}}_{p(t)}^n$ given by
\[
(I\subset A[\xx]) \in \underline{\mathbf{Mf}}_{J}(A)\qquad \longmapsto\qquad \Proj A[\xx]/I \to \Spec A \in \underline{\mathbf{Hilb}}_{p(t)}^n(\Spec A),
\]
 where $p(t)$ is the Hilbert polynomial of $J$, realizes the marked family functor as open subfunctor of the Hilbert functor. Hence, $\mathbf{Mf}_J$ turns out to be an open subscheme of the Hilbert scheme $\Hilb{p(t)}{n}$. For an ideal $I \in \underline{\mathbf{Mf}}_{J}(\kk)$, we denote by $[I]$ the corresponding point in $\mathbf{Mf}_J$ or in $\Hilb{p(t)}{n}$.

In order to study the family of ideals having initial ideal $J$ with respect to a given term order $\Omega$, we consider monic reduced $(J,\Omega)$-marked sets, namely sets of polynomials of the shape
\begin{equation}\label{eq:Omegamarkedset}
\left\{ \mathbf{f}_{\aaa} := \xx^\aaa + \sum_{\begin{subarray}{c} \xx^\bbb \in \comp{\mathfrak{J}}\\ \xx^\bbb <_\Omega \xx^\aaa \end{subarray}} c_{\aaa,\bbb}\, \xx^\bbb \ \middle\vert\ \xx^\aaa \in \mathfrak{J},\ c_{\aaa,\bbb}\in \kk \right\}.
\end{equation}
In this case, for each $\ff_{\aaa}$ the head term $\xx^\aaa$ is the leading term $\In{\Omega}(\ff_\aaa)$ with respect to the term order $\Omega$.
If a $(J,\Omega)$-marked set is a marked basis, then it is indeed the reduced Gr\"obner basis of the ideal $I$ with respect to $\Omega$ and $J = \In{\Omega}(I)$. Also in this case, we define a covariant functor $\underline{\mathbf{St}}_{J}^{\Omega}: (\text{$\kk$-Algebras}) \to (\text{Sets})$ \cite[Section 5]{LR2}. This functor is called \emph{Gr\"obner functor} and associates to a $\kk$-algebra $A$ the family of ideals in $A[\xx]$ generated by a $(J,\Omega)$-marked basis
\[
\underline{\mathbf{St}}_{J}^{\Omega}(A) = \big\{ I = (F) \in A[\xx]\ \big\vert\ F \text{~is a $(J,\Omega)$-marked basis}\big\}
\]
and to a morphism of $\kk$-algebras $f: A \to B$ the map $\underline{\mathbf{St}}_{J}^{\Omega}(f):\underline{\mathbf{St}}_{J}^{\Omega}(A) \to \underline{\mathbf{St}}_{J}^{\Omega}(B)$
\[
(I\subset A[\xx]) \in \underline{\mathbf{St}}_{J}^{\Omega}(A)\qquad \longmapsto\qquad (I\otimes_A B  \subset B[\xx] )  \in \underline{\mathbf{St}}_{J}^{\Omega}(B).
\]
The functor $\underline{\mathbf{St}}_J^\Omega$ is a closed subfunctor of $\underline{\mathbf{Mf}}_J$ and it is representable \cite[Theorem 5.3]{LR2}. Hence, in general the representing scheme $\mathbf{St}_J^\Omega$, called \emph{Gr\"obner stratum},  is a locally closed subscheme of the Hilbert scheme $\Hilb{p(t)}{n}$. 

A Gr\"obner stratum can describe an open subset of the Hilbert scheme if the inclusion $\underline{\mathbf{St}}_{J}^{\Omega} \hookrightarrow \underline{\mathbf{Mf}}_J$ is in fact a bijection. This happens with a special class of strongly stable ideals and with particular term orders.
\begin{definition}[{\cite[Definition 6.5]{LR}\cite[Definition 3.7]{CLMR}}]\label{def:hilbSegment}
Let $\Omega$ be a term order. We say that an ideal $J \in \SI{p(t)}{n}$ is the \emph{$\Omega$-hilb-segment ideal} if
\[
  \xx^\aaa >_\Omega \xx^\bbb,\qquad\forall\ \xx^\aaa \in \mathfrak{J},\ \forall\ \xx^\bbb \in \comp{\mathfrak{J}}.
\]
\end{definition}
This definition generalizes the notion of lexsegment ideals. In fact, the unique lexicographic ideal $L \in \SI{p(t)}{n}$ is the $\mathtt{DegLex}$-hilb-segment ideal. If $J\in\SI{p(t)}{n}$ is the $\Omega$-hilb-segment ideal, then the Gr\"obner stratum $\mathbf{St}_J^{\Omega}$ coincides with the marked scheme $\mathbf{Mf}_J$ and it is an open subset of $\Hilb{p(t)}{n}$.

\section{Borel deformations}\label{sec:Borel deformations}

In this section, we investigate the relative position of points of the Hilbert scheme corresponding to strongly stable ideals. In particular, we determine a combinatorial condition for two strongly stable ideals $J,J' \in \SI{p(t)}{n}$ to define points on a common irreducible component of $\Hilb{p(t)}{n}$. In next section, we discuss their behavior with respect to Gr\"obner degenerations. 

\begin{definition}\label{def:borelAdjacent}
We say that two strongly stable ideals $J,J' \subset \kk[\xx]$ with the same Hilbert polynomial are \emph{Borel adjacent} if the following conditions hold:
\begin{enumerate}
\item[(1)] \begin{itemize}
\item[-] $\mathfrak{J} \setminus \mathfrak{J}'$ has a Borel maximum $\xx^\aaa$ (i.e. a maximum with respect to the Borel order);
\item[-] $\mathfrak{J}' \setminus \mathfrak{J}$ has a Borel maximum $\xx^{\aaa'}$;
\end{itemize}
\item[(2)] there is a set $\mathcal{E}_{J,J'}$ made of the identity $\textsf{id}: \widehat{\mathbb{T}}^n \to \widehat{\mathbb{T}}^n$ and compositions of elementary decreasing moves $\ed{i_1}\circ \cdots \circ \ed{i_s}$ with such that
\[
\mathfrak{J} \setminus \mathfrak{J}' = \{ \textsf{E}(\xx^\aaa) \ \vert\ \textsf{E} \in \mathcal{E}_{J,J'} \} \qquad\text{and}\qquad
\mathfrak{J}' \setminus \mathfrak{J} = \{ \textsf{E}(\xx^{\aaa'} ) \ \vert\ \textsf{E} \in \mathcal{E}_{J,J'} \}.
\]
\end{enumerate}
\end{definition}

\begin{example}[Borel adjacent ideals]\label{ex:BorelAdjacent} [BA1 -- Figure \ref{fig:examplesBorelAdjacent}{\sc\subref{fig:examplesBorelAdjacentBA1}}] Consider the ideals $L^\sat = (x_2,x_1^3)$ and $J^\sat = (x_2^2,x_1x_2,x_1^2)$ introduced in Example \ref{ex:borelSets}. We have
\[
\mathfrak{J} \setminus \mathfrak{L} = \{ x_0 x_1^2\},\qquad \mathfrak{L} \setminus \mathfrak{J} = \{ x_0^2 x_2\}\quad\text{and}\quad\mathcal{E}_{J,L} = \{\textsf{id}\}.
\]
Conditions (1) and (2) are obviously satisfied, so that $L$ and $J$ are Borel adjacent.

\smallskip

[BA2 -- Figure \ref{fig:examplesBorelAdjacent}{\sc\subref{fig:examplesBorelAdjacentBA2}}]  In the polynomial ring $\kk[x_0,x_1,x_2]$, consider the lexicographic ideal $L^\sat = (x_2,x_1^5)$ and the ideal $J^\sat = (x_2^2,x_1^2 x_2,x_1^3)$. The sets 
\[
\mathfrak{L}\setminus\mathfrak{J} = \{ x_0^3 x_1 x_2, x_0^4 x_2\}\qquad\text{and}\qquad \mathfrak{J}\setminus\mathfrak{L} = \{x_0 x_1^4, x_0^2 x_1^3 \}
\]
have Borel maxima $x_0^3 x_1 x_2$ and $x_0 x_1^4$. Moreover,
\[
\mathfrak{L}\setminus\mathfrak{J} = \left\{ \textsf{id}(x_0^3 x_1 x_2), \ed{1}(x_0^3 x_1 x_2) \right\} \qquad\text{and}\qquad \mathfrak{J}\setminus\mathfrak{L} = \left\{ \textsf{id}(x_0 x_1^4), \ed{1}(x_0 x_1^4) \right\}
\]
so that $L$ and $J$ are Borel adjacent.

\smallskip

[BA3] Consider the ideals $J^\sat = (x_3^2,x_2x_3,x_2^2)$ and $J'^\sat = (x_3^2, x_2x_3,x_1x_3,x_2^3)$ defining points in the Hilbert scheme $\Hilb{3t+1}{3}$. The Gotzmann number of $p(t) = 3t+1$ is 4. We have
\[
\mathfrak{J} \setminus \mathfrak{J}'=\{x_1^2x_2^2, x_0x_1 x_2^2, x_0^2 x_2^2\} \quad\text{and}\quad\mathfrak{J}' \setminus \mathfrak{J}= \{x_1^3x_3, x_0x_1^2 x_3, x_0^2 x_1 x_3\}.
\]
The Borel maximum of $\mathfrak{J} \setminus \mathfrak{J}'$ is $x_1^2 x_2^2$, the Borel maximum of $\mathfrak{J}' \setminus \mathfrak{J}$ is $x_1^3x_3$. The set of compositions of decreasing elementary moves satisfying condition (2) is $\mathcal{E}_{J,J'} = \{ \textsf{id}, \ed{1},$ $ \ed{1}\circ\ed{1} \}$.
\end{example}

\begin{figure}[!ht]
\begin{center}
\subfloat[][Example \ref{ex:BorelAdjacent} -- BA1]{\label{fig:examplesBorelAdjacentBA1}
\begin{tikzpicture}[xscale=1.1,yscale=-0.7]

\begin{scope}[]
\draw [black,rounded corners,pattern color=black!35,pattern=north east lines] (-0.5,-0.5) -- (3.5,-0.5) -- (3.5,1.5) -- (0.5,1.5) -- (0.5,0.5) -- (-0.5,0.5) -- cycle;
\end{scope}
\begin{scope}[]
\draw [black,rounded corners,pattern color=black!35,pattern=north west lines] (-0.5,-0.5) -- (3.5,-0.5) -- (3.5,0.5) -- (2.5,0.5) -- (2.5,2.5) -- (1.5,2.5) -- (1.5,1.5) -- (0.5,1.5) -- (0.5,0.5)-- (-0.5,0.5) -- cycle;
\end{scope}

\node (003) at (0,0) [inner sep=1pt] {\tiny $x_2^3$};
\node (012) at (1,0) [inner sep=1pt] {\tiny $x_1x_2^2$};
\node (021) at (2,0) [inner sep=1pt] {\tiny $x_1^2x_2$};
\node (030) at (3,0) [inner sep=1pt] {\tiny $x_1^3$};

\draw [latex-] (003) -- (012);
\draw [latex-] (012) -- (021);
\draw [latex-] (021) -- (030);

\node (102) at (1,1) [inner sep=1pt] {\tiny $x_0 x_2^2$};
\node (111) at (2,1) [inner sep=1pt] {\tiny $x_0 x_1 x_2$};
\node (120) at (3,1) [inner sep=1pt] {\tiny $x_0 x_1^2$};

\draw [latex-] (012) -- (102);
\draw [latex-] (021) -- (111);
\draw [latex-] (030) -- (120);

\draw [latex-] (102) --(111);
\draw [latex-] (111) -- (120);

\node (201) at (2,2) [inner sep=1pt] {\tiny $x_0^2  x_2$};
\node (210) at (3,2) [inner sep=1pt] {\tiny $x_0^2 x_1$};

\draw [latex-] (111) -- (201);
\draw [latex-] (120) --(210);

\draw [latex-] (201) -- (210);

\node (300) at (3,3) [inner sep=1pt] {\tiny $x_0^3$};

\node at (3,5) [inner sep=1pt] {\tiny \phantom{ $x_0^5$}};
\draw [latex-] (210) -- (300);
\end{tikzpicture}
}
\hspace{2cm}
\subfloat[][Example \ref{ex:BorelAdjacent} -- BA2]{\label{fig:examplesBorelAdjacentBA2}
\begin{tikzpicture}[xscale=-1.1,yscale=-0.7]
\draw [black,rounded corners,pattern color=black!35,pattern=north west lines] (-0.5,-0.5) -- (5.5,-0.5) -- (5.5,0.5) -- (4.5,0.5) -- (4.5,1.5) -- (3.5,1.5) -- (3.5,2.5) -- (2.5,2.5) -- (2.5,3.5) -- (1.5,3.5) -- (1.5,4.5) -- (0.5,4.5) -- (0.5,0.5) -- (-0.5,0.5) -- cycle;
\draw [black,rounded corners,pattern color=black!35,pattern=north east lines] (-0.5,-0.5) -- (5.5,-0.5) -- (5.5,0.5) -- (4.5,0.5) -- (4.5,1.5) -- (3.5,1.5) -- (3.5,2.5) -- (2.5,2.5) -- (2.5,3.5) -- (1.5,3.5)  -- (1.5,2.5) -- (-0.5,2.5) -- cycle;

 \node (005) at (5,0) [inner sep=1pt] {\tiny $x_2^{5} $};
  \node (014) at (4,0) [inner sep=1pt] {\tiny $x_1x_2^{4} $};
  \draw [-latex] (014) -- (005);
  \node (023) at (3,0) [inner sep=1pt] {\tiny $x_1^{2} x_2^{3} $};
  \draw [-latex] (023) -- (014);
  \node (032) at (2,0) [inner sep=1pt] {\tiny $x_1^{3} x_2^{2} $};
  \draw [-latex] (032) -- (023);
  \node (041) at (1,0) [inner sep=1pt] {\tiny $x_1^{4} x_2$};
  \draw [-latex] (041) -- (032);
  \node (050) at (0,0) [inner sep=1pt] {\tiny $x_1^{5} $};
  \draw [-latex] (050) -- (041);
  \node (104) at (4,1) [inner sep=1pt] {\tiny $x_0x_2^{4} $};
  \draw [-latex] (104) -- (014);
  \node (113) at (3,1) [inner sep=1pt] {\tiny $x_0x_1x_2^{3} $};
  \draw [-latex] (113) -- (023);
  \draw [-latex] (113) -- (104);
  \node (122) at (2,1) [inner sep=1pt] {\tiny $x_0x_1^{2} x_2^{2} $};
  \draw [-latex] (122) -- (032);
  \draw [-latex] (122) -- (113);
  \node (131) at (1,1) [inner sep=1pt] {\tiny $x_0x_1^{3} x_2$};
  \draw [-latex] (131) -- (041);
  \draw [-latex] (131) -- (122);
  \node (140) at (0,1) [inner sep=1pt] {\tiny $x_0x_1^{4} $};
  \draw [-latex] (140) -- (050);
  \draw [-latex] (140) -- (131);
  \node (203) at (3,2) [inner sep=1pt] {\tiny $x_0^{2} x_2^{3} $};
  \draw [-latex] (203) -- (113);
  \node (212) at (2,2) [inner sep=1pt] {\tiny $x_0^{2} x_1x_2^{2} $};
  \draw [-latex] (212) -- (122);
  \draw [-latex] (212) -- (203);
  \node (221) at (1,2) [inner sep=1pt] {\tiny $x_0^{2} x_1^{2} x_2$};
  \draw [-latex] (221) -- (131);
  \draw [-latex] (221) -- (212);
  \node (230) at (0,2) [inner sep=1pt] {\tiny $x_0^{2} x_1^{3} $};
  \draw [-latex] (230) -- (140);
  \draw [-latex] (230) -- (221);
  \node (302) at (2,3) [inner sep=1pt] {\tiny $x_0^{3} x_2^{2} $};
  \draw [-latex] (302) -- (212);
  \node (311) at (1,3) [inner sep=1pt] {\tiny $x_0^{3} x_1x_2$};
  \draw [-latex] (311) -- (221);
  \draw [-latex] (311) -- (302);
  \node (320) at (0,3) [inner sep=1pt] {\tiny $x_0^{3} x_1^{2} $};
  \draw [-latex] (320) -- (230);
  \draw [-latex] (320) -- (311);
  \node (401) at (1,4) [inner sep=1pt] {\tiny $x_0^{4} x_2$};
  \draw [-latex] (401) -- (311);
  \node (410) at (0,4) [inner sep=1pt] {\tiny $x_0^{4} x_1$};
  \draw [-latex] (410) -- (320);
  \draw [-latex] (410) -- (401);
  \node (500) at (0,5) [inner sep=1pt] {\tiny $x_0^{5} $};
  \draw [-latex] (500) -- (410);
\end{tikzpicture}
}
\caption{Examples of Borel adjacent pairs of strongly stable ideals in the polynomial ring $\kk[x_0,x_1,x_2]$.}
\label{fig:examplesBorelAdjacent}
\end{center}
\end{figure}

\begin{example}[non Borel adjacent ideals]\label{ex:nonBorelAdjacent} [nBA1 -- Figure \ref{fig:examplesBorelAdjacent}{\sc\subref{fig:examplesNonBorelAdjacentBA1}}] In the polynomial ring $\kk[x_0,x_1,x_2]$, consider the ideals $J^\sat = (x_2^2,x_1^2 x_2,x_1^6)$ and $J'^\sat = (x_2^3,x_1x_2^2,x_1^3 x_2, x_1^4)$ with Hilbert polynomial $p(t)=8$ (the Gotzmann number is also $8$). We have
\[
\mathfrak{J} \setminus \mathfrak{J}'= \{x_0^6 x_2^2, x_0^5 x_1^2 x_2\}\quad\text{and}\quad\mathfrak{J}' \setminus \mathfrak{J}= \{ x_0^3 x_1^5, x_0^4 x_1^4\}.
\]
The monomials in $\mathfrak{J} \setminus \mathfrak{J}'$ are not comparable with respect to $\geq_B$, since $\vert (6,0,2) \vert_1 = 2 < 3 =  \vert (5,2,1) \vert_1$ and $\vert (6,0,2) \vert_2 = 2 > 1 = \vert (5,2,1) \vert_2$ (Lemma \ref{lem:BorelOrder}). Hence,  $\mathfrak{J} \setminus \mathfrak{J}'$ has two maximal elements and does not satisfy condition (1).

\smallskip

 [nBA2 -- Figure \ref{fig:examplesBorelAdjacent}{\sc\subref{fig:examplesNonBorelAdjacentBA2}}] In the polynomial ring $\kk[x_0,x_1,x_2]$, consider the ideals $J^\sat = (x_2^2,x_1 x_2,x_1^5)$ and $J'^\sat = (x_2^3,x_1x_2^2,$ $x_1^2 x_2, x_1^3)$ with Hilbert polynomial $p(t)=6$ (the Gotzmann number is also $6$). We have
\[
\mathfrak{J} \setminus \mathfrak{J}'= \{x_0^4 x_2^2, x_0^4 x_1 x_2\}\quad\text{and}\quad\mathfrak{J}' \setminus \mathfrak{J}= \{ x_0^2 x_1^4, x_0^3 x_1^3\}.
\]
Condition (1) is satisfied because both $\mathfrak{J} \setminus \mathfrak{J}'$ and $\mathfrak{J}' \setminus \mathfrak{J}$ have the Borel maximum  ($x_0^4 x_2^2$ and $x_0^2 x_1^4$ respectively). Whereas, condition (2) can not be satisfied as
\[
x_0^4 x_1 x_2 = \ed{2}(x_0^4 x_2^2)\qquad\text{and}\qquad x_0^3 x_1^3 = \ed{1}(x_0^2 x_1^4).
\]

\smallskip

[nBA3] Consider the lexicographic ideal $L^\sat = (x_3,x_2^4,x_1x_2^3) \subset \kk[x_0,x_1,x_2,x_3]$ defining a point on the Hilbert scheme $\Hilb{3t+1}{3}$ and the ideal $J^\sat = (x_3^2, x_2 x_3, x_2^2)$ already introduced in Example \ref{ex:BorelAdjacent} [BA3]. We have
\[
\mathfrak{L} \setminus \mathfrak{J} = \{ x_1^3 x_3,  x_0 x_1^2 x_3, x_0^2 x_1 x_3, x_0^3 x_3 \} \qquad\text{and}\qquad \mathfrak{J}\setminus \mathfrak{L} = \{ x_1^2 x_2^2, x_0 x_1 x_2^2, x_0^2 x_2^2, x_0 x_2^3\}.
\]
The set $\mathfrak{L} \setminus \mathfrak{J}$ has Borel maximum $x_1^3 x_3$, while $\mathfrak{J}\setminus \mathfrak{L}$ has two maximal elements: $x_1^2 x_2^2$ and $x_0 x_2^3$. Hence, $L$ and $J$ are not Borel adjacent.
\end{example}

\begin{figure}[!ht]
\begin{center}
\subfloat[][Example \ref{ex:nonBorelAdjacent} -- nBA1]{\label{fig:examplesNonBorelAdjacentBA1}
\begin{tikzpicture}[xscale=-1.2,yscale=-0.7]
\draw [black,rounded corners,pattern color=black!35,pattern=north west lines] (-0.5,-0.5) -- (8.5,-0.5) -- (8.5,0.5) -- (7.5,0.5) -- (7.5,1.5) -- (6.5,1.5) -- (6.5,2.5) -- (5.5,2.5) -- (5.5,3.5) -- (4.5,3.5) -- (4.5,4.5) -- (3.5,4.5) -- (3.5,5.5) -- (2.5,5.5) -- (2.5,6.5) -- (1.5,6.5) -- (1.5,5.5) -- (0.5,5.5) -- (0.5,2.5) -- (-0.5,2.5)-- cycle;
\draw [black,rounded corners,pattern color=black!35,pattern=north east lines] (-0.5,-0.5) -- (8.5,-0.5) -- (8.5,0.5) -- (7.5,0.5) -- (7.5,1.5) -- (6.5,1.5) -- (6.5,2.5) -- (5.5,2.5) -- (5.5,3.5) -- (4.5,3.5) -- (4.5,4.5) -- (3.5,4.5) -- (3.5,5.5) -- (1.5,5.5) -- (1.5,4.5) --  (-0.5,4.5)-- cycle;

\node (008) at (8,0) [inner sep=1pt] {\tiny $x_2^{8} $};
  \node (017) at (7,0) [inner sep=1pt] {\tiny $x_1x_2^{7} $};
  \draw [-latex] (017) -- (008);
  \node (026) at (6,0) [inner sep=1pt] {\tiny $x_1^{2} x_2^{6} $};
  \draw [-latex] (026) -- (017);
  \node (035) at (5,0) [inner sep=1pt] {\tiny $x_1^{3} x_2^{5} $};
  \draw [-latex] (035) -- (026);
  \node (044) at (4,0) [inner sep=1pt] {\tiny $x_1^{4} x_2^{4} $};
  \draw [-latex] (044) -- (035);
  \node (053) at (3,0) [inner sep=1pt] {\tiny $x_1^{5} x_2^{3} $};
  \draw [-latex] (053) -- (044);
  \node (062) at (2,0) [inner sep=1pt] {\tiny $x_1^{6} x_2^{2} $};
  \draw [-latex] (062) -- (053);
  \node (071) at (1,0) [inner sep=1pt] {\tiny $x_1^{7} x_2$};
  \draw [-latex] (071) -- (062);
  \node (080) at (0,0) [inner sep=1pt] {\tiny $x_1^{8} $};
  \draw [-latex] (080) -- (071);
  \node (107) at (7,1) [inner sep=1pt] {\tiny $x_0x_2^{7} $};
  \draw [-latex] (107) -- (017);
  \node (116) at (6,1) [inner sep=1pt] {\tiny $x_0x_1x_2^{6} $};
  \draw [-latex] (116) -- (026);
  \draw [-latex] (116) -- (107);
  \node (125) at (5,1) [inner sep=1pt] {\tiny $x_0x_1^{2} x_2^{5} $};
  \draw [-latex] (125) -- (035);
  \draw [-latex] (125) -- (116);
  \node (134) at (4,1) [inner sep=1pt] {\tiny $x_0x_1^{3} x_2^{4} $};
  \draw [-latex] (134) -- (044);
  \draw [-latex] (134) -- (125);
  \node (143) at (3,1) [inner sep=1pt] {\tiny $x_0x_1^{4} x_2^{3} $};
  \draw [-latex] (143) -- (053);
  \draw [-latex] (143) -- (134);
  \node (152) at (2,1) [inner sep=1pt] {\tiny $x_0x_1^{5} x_2^{2} $};
  \draw [-latex] (152) -- (062);
  \draw [-latex] (152) -- (143);
  \node (161) at (1,1) [inner sep=1pt] {\tiny $x_0x_1^{6} x_2$};
  \draw [-latex] (161) -- (071);
  \draw [-latex] (161) -- (152);
  \node (170) at (0,1) [inner sep=1pt] {\tiny $x_0x_1^{7} $};
  \draw [-latex] (170) -- (080);
  \draw [-latex] (170) -- (161);
  \node (206) at (6,2) [inner sep=1pt] {\tiny $x_0^{2} x_2^{6} $};
  \draw [-latex] (206) -- (116);
  \node (215) at (5,2) [inner sep=1pt] {\tiny $x_0^{2} x_1x_2^{5} $};
  \draw [-latex] (215) -- (125);
  \draw [-latex] (215) -- (206);
  \node (224) at (4,2) [inner sep=1pt] {\tiny $x_0^{2} x_1^{2} x_2^{4} $};
  \draw [-latex] (224) -- (134);
  \draw [-latex] (224) -- (215);
  \node (233) at (3,2) [inner sep=1pt] {\tiny $x_0^{2} x_1^{3} x_2^{3} $};
  \draw [-latex] (233) -- (143);
  \draw [-latex] (233) -- (224);
  \node (242) at (2,2) [inner sep=1pt] {\tiny $x_0^{2} x_1^{4} x_2^{2} $};
  \draw [-latex] (242) -- (152);
  \draw [-latex] (242) -- (233);
  \node (251) at (1,2) [inner sep=1pt] {\tiny $x_0^{2} x_1^{5} x_2$};
  \draw [-latex] (251) -- (161);
  \draw [-latex] (251) -- (242);
  \node (260) at (0,2) [inner sep=1pt] {\tiny $x_0^{2} x_1^{6} $};
  \draw [-latex] (260) -- (170);
  \draw [-latex] (260) -- (251);
  \node (305) at (5,3) [inner sep=1pt] {\tiny $x_0^{3} x_2^{5} $};
  \draw [-latex] (305) -- (215);
  \node (314) at (4,3) [inner sep=1pt] {\tiny $x_0^{3} x_1x_2^{4} $};
  \draw [-latex] (314) -- (224);
  \draw [-latex] (314) -- (305);
  \node (323) at (3,3) [inner sep=1pt] {\tiny $x_0^{3} x_1^{2} x_2^{3} $};
  \draw [-latex] (323) -- (233);
  \draw [-latex] (323) -- (314);
  \node (332) at (2,3) [inner sep=1pt] {\tiny $x_0^{3} x_1^{3} x_2^{2} $};
  \draw [-latex] (332) -- (242);
  \draw [-latex] (332) -- (323);
  \node (341) at (1,3) [inner sep=1pt] {\tiny $x_0^{3} x_1^{4} x_2$};
  \draw [-latex] (341) -- (251);
  \draw [-latex] (341) -- (332);
  \node (350) at (0,3) [inner sep=1pt] {\tiny $x_0^{3} x_1^{5} $};
  \draw [-latex] (350) -- (260);
  \draw [-latex] (350) -- (341);
  \node (404) at (4,4) [inner sep=1pt] {\tiny $x_0^{4} x_2^{4} $};
  \draw [-latex] (404) -- (314);
  \node (413) at (3,4) [inner sep=1pt] {\tiny $x_0^{4} x_1x_2^{3} $};
  \draw [-latex] (413) -- (323);
  \draw [-latex] (413) -- (404);
  \node (422) at (2,4) [inner sep=1pt] {\tiny $x_0^{4} x_1^{2} x_2^{2} $};
  \draw [-latex] (422) -- (332);
  \draw [-latex] (422) -- (413);
  \node (431) at (1,4) [inner sep=1pt] {\tiny $x_0^{4} x_1^{3} x_2$};
  \draw [-latex] (431) -- (341);
  \draw [-latex] (431) -- (422);
  \node (440) at (0,4) [inner sep=1pt] {\tiny $x_0^{4} x_1^{4} $};
  \draw [-latex] (440) -- (350);
  \draw [-latex] (440) -- (431);
  \node (503) at (3,5) [inner sep=1pt] {\tiny $x_0^{5} x_2^{3} $};
  \draw [-latex] (503) -- (413);
  \node (512) at (2,5) [inner sep=1pt] {\tiny $x_0^{5} x_1x_2^{2} $};
  \draw [-latex] (512) -- (422);
  \draw [-latex] (512) -- (503);
  \node (521) at (1,5) [inner sep=1pt] {\tiny $x_0^{5} x_1^{2} x_2$};
  \draw [-latex] (521) -- (431);
  \draw [-latex] (521) -- (512);
  \node (530) at (0,5) [inner sep=1pt] {\tiny $x_0^{5} x_1^{3} $};
  \draw [-latex] (530) -- (440);
  \draw [-latex] (530) -- (521);
  \node (602) at (2,6) [inner sep=1pt] {\tiny $x_0^{6} x_2^{2} $};
  \draw [-latex] (602) -- (512);
  \node (611) at (1,6) [inner sep=1pt] {\tiny $x_0^{6} x_1x_2$};
  \draw [-latex] (611) -- (521);
  \draw [-latex] (611) -- (602);
  \node (620) at (0,6) [inner sep=1pt] {\tiny $x_0^{6} x_1^{2} $};
  \draw [-latex] (620) -- (530);
  \draw [-latex] (620) -- (611);
  \node (701) at (1,7) [inner sep=1pt] {\tiny $x_0^{7} x_2$};
  \draw [-latex] (701) -- (611);
  \node (710) at (0,7) [inner sep=1pt] {\tiny $x_0^{7} x_1$};
  \draw [-latex] (710) -- (620);
  \draw [-latex] (710) -- (701);
  \node (800) at (0,8) [inner sep=1pt] {\tiny $x_0^{8} $};
  \draw [-latex] (800) -- (710);
\end{tikzpicture}
}

\subfloat[][Example \ref{ex:nonBorelAdjacent} -- nBA2]{\label{fig:examplesNonBorelAdjacentBA2}
\begin{tikzpicture}[xscale=-1.2,yscale=-0.7]

\begin{scope}[]
\draw [black,rounded corners,pattern color=black!35,pattern=north east lines] (-0.5,-0.5) -- (6.5,-0.5) -- (6.5,0.5) -- (5.5,0.5) -- (5.5,1.5) -- (4.5,1.5) -- (4.5,2.5) -- (3.5,2.5) -- (3.5,3.5) -- (2.5,3.5) -- (2.5,4.5) -- (0.5,4.5)  --(0.5,1.5) --(-0.5,1.5) -- cycle;
\end{scope}
\begin{scope}[]
\draw [black,rounded corners,pattern color=black!35,pattern=north west lines] (-0.5,-0.5) -- (6.5,-0.5) -- (6.5,0.5) -- (5.5,0.5) -- (5.5,1.5) -- (4.5,1.5) -- (4.5,2.5) -- (3.5,2.5) -- (3.5,3.5) -- (-0.5,3.5) -- cycle; 
\end{scope}

\node (006) at (6,0) [inner sep=1pt] {\tiny $x_2^{6} $};
  \node (015) at (5,0) [inner sep=1pt] {\tiny $x_1x_2^{5} $};
  \draw [-latex] (015) -- (006);
  \node (024) at (4,0) [inner sep=1pt] {\tiny $x_1^{2} x_2^{4} $};
  \draw [-latex] (024) -- (015);
  \node (033) at (3,0) [inner sep=1pt] {\tiny $x_1^{3} x_2^{3} $};
  \draw [-latex] (033) -- (024);
  \node (042) at (2,0) [inner sep=1pt] {\tiny $x_1^{4} x_2^{2} $};
  \draw [-latex] (042) -- (033);
  \node (051) at (1,0) [inner sep=1pt] {\tiny $x_1^{5} x_2$};
  \draw [-latex] (051) -- (042);
  \node (060) at (0,0) [inner sep=1pt] {\tiny $x_1^{6} $};
  \draw [-latex] (060) -- (051);
  \node (105) at (5,1) [inner sep=1pt] {\tiny $x_0x_2^{5} $};
  \draw [-latex] (105) -- (015);
  \node (114) at (4,1) [inner sep=1pt] {\tiny $x_0x_1x_2^{4} $};
  \draw [-latex] (114) -- (024);
  \draw [-latex] (114) -- (105);
  \node (123) at (3,1) [inner sep=1pt] {\tiny $x_0x_1^{2} x_2^{3} $};
  \draw [-latex] (123) -- (033);
  \draw [-latex] (123) -- (114);
  \node (132) at (2,1) [inner sep=1pt] {\tiny $x_0x_1^{3} x_2^{2} $};
  \draw [-latex] (132) -- (042);
  \draw [-latex] (132) -- (123);
  \node (141) at (1,1) [inner sep=1pt] {\tiny $x_0x_1^{4} x_2$};
  \draw [-latex] (141) -- (051);
  \draw [-latex] (141) -- (132);
  \node (150) at (0,1) [inner sep=1pt] {\tiny $x_0x_1^{5} $};
  \draw [-latex] (150) -- (060);
  \draw [-latex] (150) -- (141);
  \node (204) at (4,2) [inner sep=1pt] {\tiny $x_0^{2} x_2^{4} $};
  \draw [-latex] (204) -- (114);
  \node (213) at (3,2) [inner sep=1pt] {\tiny $x_0^{2} x_1x_2^{3} $};
  \draw [-latex] (213) -- (123);
  \draw [-latex] (213) -- (204);
  \node (222) at (2,2) [inner sep=1pt] {\tiny $x_0^{2} x_1^{2} x_2^{2} $};
  \draw [-latex] (222) -- (132);
  \draw [-latex] (222) -- (213);
  \node (231) at (1,2) [inner sep=1pt] {\tiny $x_0^{2} x_1^{3} x_2$};
  \draw [-latex] (231) -- (141);
  \draw [-latex] (231) -- (222);
  \node (240) at (0,2) [inner sep=1pt] {\tiny $x_0^{2} x_1^{4} $};
  \draw [-latex] (240) -- (150);
  \draw [-latex] (240) -- (231);
  \node (303) at (3,3) [inner sep=1pt] {\tiny $x_0^{3} x_2^{3} $};
  \draw [-latex] (303) -- (213);
  \node (312) at (2,3) [inner sep=1pt] {\tiny $x_0^{3} x_1x_2^{2} $};
  \draw [-latex] (312) -- (222);
  \draw [-latex] (312) -- (303);
  \node (321) at (1,3) [inner sep=1pt] {\tiny $x_0^{3} x_1^{2} x_2$};
  \draw [-latex] (321) -- (231);
  \draw [-latex] (321) -- (312);
  \node (330) at (0,3) [inner sep=1pt] {\tiny $x_0^{3} x_1^{3} $};
  \draw [-latex] (330) -- (240);
  \draw [-latex] (330) -- (321);
  \node (402) at (2,4) [inner sep=1pt] {\tiny $x_0^{4} x_2^{2} $};
  \draw [-latex] (402) -- (312);
  \node (411) at (1,4) [inner sep=1pt] {\tiny $x_0^{4} x_1x_2$};
  \draw [-latex] (411) -- (321);
  \draw [-latex] (411) -- (402);
  \node (420) at (0,4) [inner sep=1pt] {\tiny $x_0^{4} x_1^{2} $};
  \draw [-latex] (420) -- (330);
  \draw [-latex] (420) -- (411);
  \node (501) at (1,5) [inner sep=1pt] {\tiny $x_0^{5} x_2$};
  \draw [-latex] (501) -- (411);
  \node (510) at (0,5) [inner sep=1pt] {\tiny $x_0^{5} x_1$};
  \draw [-latex] (510) -- (420);
  \draw [-latex] (510) -- (501);
  \node (600) at (0,6) [inner sep=1pt] {\tiny $x_0^{6} $};
  \draw [-latex] (600) -- (510);
\end{tikzpicture}
}
\caption{Examples of non Borel adjacent pairs of strongly stable ideals in the polynomial ring $\kk[x_0,x_1,x_2]$.}
\label{fig:examplesNonBorelAdjacent}
\end{center}
\end{figure}

After giving some examples of Borel adjacent ideals, we make explicit some properties that are in a sense hidden in the definition.
\begin{remark}\label{rk:propertiesAdjacent}
\begin{enumerate}[\it (i)]
\item\label{rk:propertiesAdjacent_i} Any pair of strongly stable ideals $J,J' \in \SI{p(t)}{n}$ such that $\vert \mathfrak{J}\setminus \mathfrak{J}'\vert = \vert\mathfrak{J}'\setminus\mathfrak{J} \vert = 1$ is Borel adjacent.
 
\item\label{rk:propertiesAdjacent_ii} The fact that $\mathfrak{J}$ and $\mathfrak{J}'$ are closed under the action of increasing elementary moves implies that  also sets $\mathfrak{J} \setminus \mathfrak{J}'$ and $ \mathfrak{J}' \setminus \mathfrak{J}$ are. Indeed, if $\xx^\aaa$ is the Borel maximum of $\mathfrak{J}\setminus\mathfrak{J}'$, $\xx^{\bbb}$ is another monomial in $\mathfrak{J} \setminus \mathfrak{J}'$ and $\xx^\ccc$ is a monomial such that $\xx^{\aaa} \geq_B \xx^{\ccc} \geq_B \xx^{\bbb}$, then $\xx^\ccc \in \mathfrak{J}\setminus\mathfrak{J}'$ (the same holds for $\mathfrak{J}'\setminus \mathfrak{J}$).

\item\label{rk:propertiesAdjacent_iii} The set $\mathcal{E}_{J,J'}$ represents a bijection between the sets $ \mathfrak{J} \setminus \mathfrak{J}'$ and $\mathfrak{J}' \setminus \mathfrak{J}$. One has 
\begin{equation*}
\min \textsf{E}(\xx^\aaa) = \min \textsf{E}(\xx^{\aaa'}),\quad\forall\ \textsf{E} \in \mathcal{E}_{J,J'}\qquad \text{(in particular $\min \xx^\aaa = \min \xx^{\aaa'}$).}
\end{equation*}
In fact, assume $\min \xx^\aaa > \min \xx^{\aaa'}$. By definition of Borel order, we have $\min \textsf{E}(\xx^{\aaa'}) \leqslant \min \xx^{\aaa'} < \min \xx^{\aaa}$ for all $\textsf{E} \in \mathcal{E}_{J,J'}$. But this is not possible because $J$ and $J'$ have the same Hilbert polynomial and $\vert\mathfrak{J}_i\vert =\vert\mathfrak{J}'_i\vert$ for all $i=0,\ldots,n$. This implies that $\vert(\mathfrak{J}\setminus\mathfrak{J}')_i\vert =\vert(\mathfrak{J}'\setminus\mathfrak{J})_i\vert$ for all $i=0,\ldots,n$. 

\item\label{rk:propertiesAdjacent_iv} The monomials $\xx^{\bbb'} \in \mathfrak{J}'\setminus\mathfrak{J}$ are on the \lq\lq outer border\rq\rq~of $\mathfrak{J}$. For any admissible move $\eu{i}$ the monomial $\eu{i}(\xx^{\bbb'})$ either remains in $\mathfrak{J}'\setminus\mathfrak{J}$ or comes into $\mathfrak{J}$. 
The monomials $\xx^{\bbb} \in \mathfrak{J}\setminus\mathfrak{J}'$ are on the \lq\lq inner border\rq\rq~of $\mathfrak{J}$. For any admissible move $\ed{j}$ the monomial $\ed{j}(\xx^{\bbb})$ is either contained in $\mathfrak{J}$ or exits in $\mathfrak{J}'$.

\item\label{rk:propertiesAdjacent_v} The previous remark suggests how to search for Borel adjacent ideals to a given ideal $J$. 
\begin{itemize}
\item[\underline{\it Step 1.}] Determine the set of maximal elements in $\comp{\mathfrak{J}}$ with respect to the Borel order $\geq_B$.

\item[\underline{\it Step 2.}] For each maximal element $\xx^\aaa$ with $\min \xx^\aaa = k$, determine the set of minimal elements in $\mathfrak{J}_{\geqslant k}$ with respect to $\geq_B$.

\item[\underline{\it Step 3.}] For each minimal element $\xx^\bbb \in \mathfrak{J}_{\geqslant k}$ not comparable with $\xx^\aaa$ with respect to $\geq_B$, consider the set $\mathfrak{E} = \{ \xx^\ccc \in \mathfrak{J} \ \vert\ \xx^\ccc \leq_B \xx^\bbb \}$ and the corresponding set of moves $\mathcal{E}$ such that $\mathfrak{E} = \{ \textsf{E}(\xx^\bbb)\ \vert\ \textsf{E} \in \mathcal{E}\}$. $\mathfrak{E}$ describes a \lq\lq inner border\rq\rq~of $\mathfrak{J}$.

\item[\underline{\it Step 4.}] Compute $\mathfrak{F} =  \{ \textsf{E}(\xx^{\aaa})\ \vert\ \textsf{E} \in \mathcal{E}\}$ and check whether $\mathfrak{F}$ describes an \lq\lq outer border\rq\rq~of $\mathfrak{J}$. This means that $(\dagger)$ all the moves $\textsf{E} \in \mathcal{E}$ are admissible for $\xx^{\aaa}$ and $(\ddagger)$ for $\xx^{\ccc} \in \mathfrak{F}$, $\eu{h}(\xx^{\ccc})$ is either in $\mathfrak{F}$ or in $\mathfrak{J}$ for every admissible $\eu{h}$. If this is the case, then $\mathfrak{J} \setminus \mathfrak{E} \cup \mathfrak{F}$ describes the set of generators in degree $r$ of a strongly stable ideal $J'$ with the same Hilbert polynomial of $J$.
\end{itemize}
Notice that in general this procedure does not exhaust the list of all Borel adjacent ideals to $J$. (see for instance Example \ref{ex:BorelAdjacent} [BA2]).
\end{enumerate}
\end{remark}

\bigskip

Next theorem shows that our definition of \lq\lq combinatorial proximity\rq\rq~of strongly stable ideals carries also a \lq\lq geometric proximity\rq\rq~meaning.

\begin{theorem}\label{thm:mainDef}
Let $J,J' \subset \kk[\xx]$ be two Borel adjacent strongly stable ideals. Let $\xx^\aaa$ and $\xx^{\aaa'}$ be the Borel maxima of $\mathfrak{J} \setminus \mathfrak{J}'$ and $\mathfrak{J}'\setminus\mathfrak{J}$. The bi-homogenous ideal $I_{J,J'} \subset \kk[y_0,y_1][\xx]$ generated by the polynomials 
\begin{equation}\label{eq:idealDeformation}
\left(\mathfrak{J} \cap \mathfrak{J}'\right) \cup \left\{ y_0 \mathsf{E}(\xx^\aaa)+y_1 \mathsf{E}(\xx^{\aaa'})\ \middle\vert\ \mathsf{E} \in \mathcal{E}_{J,J'}\right\}
\end{equation}
defines a flat family $\pi : X_{J,J'}  \subset \PP^1 \times_\kk \PP^n \to \mathbb{P}^1$, where $X_{J,J'} = \Proj  \kk[y_0,y_1][\xx]/I_{J,J'}$ and $\pi$ is the restriction to $X_{J,J'}$ of the standard projection $ \PP^1 \times_\kk \PP^n \to \mathbb{P}^1$. Moreover, the fiber $X_{J,J'}\vert_{[1:0]}$ is the scheme $\Proj \left( \kk[\xx]/J \right)$ and the fiber  $X_{J,J'}\vert_{[0:1]}$ is  $\Proj \left(\kk[\xx]/J'\right)$. We call this family \emph{Borel deformation} of $J$ and $J'$.
\end{theorem}
\begin{proof}
The second part of the statement is straightforward from the definition of the ideal $I_{J,J'}$. In fact,
\[
\{ \mathsf{E}(\xx^\aaa)\ \vert\ \mathsf{E} \in \mathcal{E}_{J,J'}\} = \mathfrak{J}\setminus\mathfrak{J}'\quad\text{and}\quad \{ \mathsf{E}(\xx^{\aaa'})\ \vert\ \mathsf{E} \in \mathcal{E}_{J,J'}\} = \mathfrak{J}'\setminus\mathfrak{J}.
\]

Let us consider the standard affine open cover of $\PP^1$ made up of $U_0 = \PP^1 \setminus \{[0:1]\}$ and $U_1 = \PP^1 \setminus \{[1:0]\}$. We prove the flatness of $X_{J,J'}$ over $\PP^1$ by showing that both families $X_{J,J'}\vert_{U_0} \to U_0$ and  $X_{J,J'}\vert_{U_1} \to U_1$ are flat.

Over $U_0$, we can rewrite the ideal $I_{J,J'}$ as the ideal $I_{J,\widehat{J'}} \subset \kk[T][\xx],\ T = y_1/y_0$, generated by the $J$-marked set
\[
\left(\mathfrak{J} \cap \mathfrak{J}'\right)  \cup \left\{ \mathsf{E}(\xx^\aaa)+T\mathsf{E}(\xx^{\aaa'})\ \middle\vert\ \mathsf{E} \in \mathcal{E}_{J,J'}\right\}.
\]
In order to prove that the family is flat, we show that Eliahou-Kervaire syzygies of the ideal $J$ lift to syzygies among the elements of the $J$-marked set (Proposition \ref{prop:liftSyzygies}).

Let $\xx^{\bbb} \in \mathfrak{J}\cap\mathfrak{J}'$. Any monomial $\xx^\ccc$ appearing in an Eliahou-Kervaire syzygy
\[
x_i \xx^\bbb - x_h \xx^\ccc = 0,\qquad i > h = \min \xx^\bbb
\]
is also contained in $\mathfrak{J}\cap\mathfrak{J}'$, because $\xx^\ccc >_B \xx^\bbb$. Then, $\xx^\bbb$ and $\xx^\ccc$ are both generators of $J$ and $I_{J,\widehat{J'}}$ and the syzygy among them trivially lifts.

\smallskip

Now, consider $\xx^\bbb \in \mathfrak{J}\setminus\mathfrak{J'}$. A monomial $\xx^\ccc$ involved in the syzygy $x_i \xx^\bbb - x_h \xx^\ccc = 0,\ i > h = \min \xx^\bbb$ can either belong to $\mathfrak{J}\setminus\mathfrak{J'}$ or to $\mathfrak{J}\cap\mathfrak{J}'$.

\textit{Case 1:} $\xx^\ccc \in \mathfrak{J}\setminus\mathfrak{J'}$. Let $\textsf{E},\widetilde{\mathsf{E}}$ be the elements of $\mathcal{E}_{J,J'}$ such that 
\[
\xx^\bbb = \textsf{E}(\xx^\aaa) = \Ht\big(\textsf{E}(\xx^\aaa) + T \textsf{E}(\xx^{\aaa'})\big)\qquad\text{and}\qquad\xx^\ccc = \widetilde{\textsf{E}}(\xx^\aaa) = \Ht\big(\widetilde{\textsf{E}}(\xx^\aaa) + T \widetilde{\textsf{E}}(\xx^{\aaa'})\big).
\]
From the syzygy $x_i \textsf{E}(\xx^\aaa) - x_h  \widetilde{\textsf{E}}(\xx^\aaa)$, one has
\[
\widetilde{\textsf{E}}(\xx^\aaa)  = \frac{x_i}{x_h} \textsf{E}(\xx^\aaa)= \eu{i-1}\circ\cdots\circ \eu{h} \circ \textsf{E} (\xx^\aaa)\quad \Rightarrow \quad \widetilde{\textsf{E}} = \eu{i-1}\circ\cdots\circ \eu{h} \circ \textsf{E}.
\]
As $\min \textsf{E}(\xx^\aaa) = \min \textsf{E}(\xx^{\aaa'})$, the composition $\eu{i-1}\circ\cdots \eu{h} $ is also admissible for $\textsf{E}(\xx^{\aaa'})$ and
\[
 \eu{i-1}\circ\cdots \eu{h} \circ \textsf{E}(\xx^{\aaa'}) = \widetilde{\textsf{E}}(\xx^{\aaa'}) \quad \Rightarrow\quad \widetilde{\textsf{E}}(\xx^{\aaa'}) = \frac{x_i}{x_h} \textsf{E}(\xx^{\aaa'}),
\]
so that the syzygy between the generators $\xx^\bbb$ and $\xx^\ccc$ of $J$ lifts to the syzygy
\[
x_i \big(\textsf{E}(\xx^\aaa) + T \textsf{E}(\xx^{\aaa'})\big) - x_h \big(\widetilde{\textsf{E}}(\xx^\aaa) + T \widetilde{\textsf{E}}(\xx^{\aaa'})\big) = x_i \textsf{E}(\xx^\aaa) - x_h \widetilde{\textsf{E}}(\xx^\aaa) + T \big( x_i \textsf{E}(\xx^{\aaa'}) - x_h \widetilde{\textsf{E}}(\xx^{\aaa'}) \big) = 0
\]
between the generators $\textsf{E}(\xx^\aaa) + T \textsf{E}(\xx^{\aaa'})$ and $\widetilde{\textsf{E}}(\xx^\aaa) + T \widetilde{\textsf{E}}(\xx^{\aaa'})$ of $I_{J,\widehat{J}'}$.
\smallskip

\textit{Case 2:} $\xx^\ccc \in \mathfrak{J}\cap\mathfrak{J'}$. Let $\textsf{E}$ be the element of $\mathcal{E}_{J,J'}$ such that $\xx^\bbb = \textsf{E}(\xx^\aaa)$. As $\min \textsf{E}(\xx^\aaa) = \min \textsf{E}(\xx^{\aaa'})$, the product $\frac{x_i}{x_h} \textsf{E}(\xx^{{\aaa'}})$ is a standard monomial $\xx^{\ccc'}$ and is contained in $\mathfrak{J}\cap\mathfrak{J}$. Hence, the syzygy $x_i \xx^\bbb - x_h \xx^\ccc = 0$ lifts to the syzygy
\[
x_i\big(\textsf{E}(\xx^\aaa) + T \textsf{E}(\xx^{\aaa'})\big) - x_h \xx^\ccc - (x_h T) \xx^{\ccc'} = 0.
\]

The proof of the flatness of $X_{J,J'}\vert_{U_1} \to U_1$ follows the same argument exchanging the role of $J$ and $J'$.
\end{proof}

\begin{remark}\label{rk:T-orbit}
We notice that the ideal $I_{J,J'}\subset \kk[y_0,y_1][\xx]$ describing a Borel deformation is homogeneous with respect to the non-standard grading
\[
\begin{array}{rccc}
\deg_{\underline{c}}:&\mathbb{T}^n & \longrightarrow & \ZZ^{n+1}/\underline{c}\ZZ \\
& \xx^{\underline{v}} & \longmapsto &  \underline{v}
\end{array}\qquad\text{where~} \underline{c} = \aaa - \aaa'.
\]
In fact, for any element $\textsf{E} \in \mathcal{E}_{J,J'}$, let $\tfrac{\xx^{\underline{p}}}{\xx^{\underline{q}}}$ be the generalized monomial associated to $\textsf{E}$, i.e.~$\textsf{E}(\xx^\aaa) = \tfrac{\xx^{\underline{p}}}{\xx^{\underline{q}}} \xx^{\aaa}$ and $\textsf{E}(\xx^{\aaa'}) = \tfrac{\xx^{\underline{p}}}{\xx^{\underline{q}}} \xx^{\aaa'}$. One has
\[
\deg_{\underline{c}} \textsf{E}(\xx^\aaa) - \deg_{\underline{c}} \textsf{E}(\xx^{\aaa'}) = (\underline{p} + \aaa - \underline{q} ) - (\underline{p}+\aaa'-\underline{q}) = \aaa - \aaa' = \underline{0} \in \ZZ^{n+1}/\underline{c}\ZZ.
\]
Note that by Remark \ref{rk:propertiesAdjacent} and Lemma \ref{lem:BorelOrder}, there exist $i$ and $j$ such that $c_i > 0$ and $c_j < 0$.
Under these assumptions, Hering and Maclagan prove that the ideal $I_{J,J'}$ identifies a one-dimen- sional orbit $\mathcal{O} \subset \Hilb{p(t)}{n}$ of the action induced on the Hilbert scheme by the standard torus $T = (\kk^{\ast})^{n+1}$ of $\PP^n$ such that $\overline{\mathcal{O}} = \mathcal{O} \cup \{[J],[J']\}$  \cite[Proposition 2.5]{HeringMaclagan}. Therefore, every rational curve describing a Borel deformation turns out to be the closure of a $T$-orbit. 

One may wonder whether also the contrary is true, i.e.~whether the closure of a one-dimen- sional $T$-orbit always corresponds to a Borel deformation. This is not the case and we exhibit two types of counterexamples.
\begin{enumerate}[\it (i)]
\item {\em The closure of a one-dimensional $T$-orbit does not contain a pair of strongly stable ideals.} For instance, in the case of the Hilbert scheme $\Hilb{2}{2}$, there are 18 one-dimensional $T$-orbits (see \cite[Section 5.2]{HeringMaclagan}). However, none of them corresponds to a Borel deformation. In fact, chosen an order on the variables there exists a unique strongly stable ideal.
\item {\em The closure of a one-dimensional $T$-orbit contains a pair of strongly stable ideals that are not Borel adjacent.} Consider the strongly stable ideals $J = (x_2^4,x_1 x_2^3,x_1^2 x_2^2,x_1^5 x_2, x_1^6)_{\geqslant 14}$ and $J' = (x_2^3,x_1^3 x_2^2,x_1^4 x_2,x_1^7)_{\geqslant 14}$ in $\kk[x_0,x_1,x_2]$ defining points on $\Hilb{14}{2}$. The ideal generated by 
\[
\left( \mathfrak{J} \cap \mathfrak{J}' \right) \cup \big\{y_0\, x_0^{11} x_2^3 + y_1\, x_0^{10} x_1^2 x_2^2, y_0\, x_0^{9} x_1^4 x_2  + y_1\, x_0^{8} x_1^6 \big\}
\]
is homogeneous with respect to the grading $\mathbb{T}^2 \to \ZZ^3/(1,-2,1)\ZZ$ and identifies a one-dimensional $T$-orbit. Hence, such ideal defines a flat family that describes a rational curve on $\Hilb{p(t)}{n}$ passing through the points $[J]$ and $[J']$. However, $J$ and $J'$ are not Borel adjacent, as 
\[
\mathfrak{J} \setminus \mathfrak{J}' = \big\{x_0^{11} x_2^3, x_0^9 x_1^4 x_2\big\}\qquad\text{and}\qquad\mathfrak{J}'\setminus \mathfrak{J} =\big\{x_0^{10} x_1^2 x_2^2, x_0^8 x_1^6\big\}
\]
do not satisfy condition (1) of Definition \ref{def:borelAdjacent}.\bs
\end{enumerate}
\end{remark}

\begin{corollary}\label{cor:markedSchemeClosure}
Let $J,J' \subset \kk[\xx]$ be two Borel adjacent strongly stable ideals. The points $[J],[J'] \in \Hilb{p(t)}{n}$ are contained in a common irreducible component. Moreover,
 $[J']$ is contained in the closure $\overline{\mathbf{Mf}_J} \subset  \Hilb{p(t)}{n}$ and $[J]$ is contained in the closure $\overline{\mathbf{Mf}_{J'}}\subset  \Hilb{p(t)}{n}$.
\end{corollary}
\begin{proof}
By Theorem \ref{thm:mainDef}, we have $X_{J,J'} \to \PP^1 \in \underline{\mathbf{Hilb}}_{p(t)}^n(\PP^1)$. Hence, there exists a morphism of schemes $\varphi_{I,J}: \PP^1 \to \Hilb{p(t)}{n}$ such that $X_{J,J'}$ is the pullback of the universal family $\mathcal{U}_{p(t)}^n  \to \Hilb{p(t)}{n}$. The image $\varphi_{I,J}(\PP^1)$ is contained in a unique irreducible component of $ \Hilb{p(t)}{n}$ and contains the points $[J] = \varphi_{J,J'}([1:0])$ and $[J'] = \varphi_{J,J'}([0:1])$.

The second part of the statement is a consequence of the observation that $I_{J,\widehat{J}'} \in \underline{\mathbf{Mf}}_J(\mathbb{A}^1)$, where $I_{J,\widehat{J}'}$ is the ideal defining the restriction of the family $X_{J,J'}\vert_{U_0} \to  U_0$, with $U_0 = \PP^1\setminus \{[0:1]\}$. The same holds for $I_{\widehat{J},J'} \in  \underline{\mathbf{Mf}}_{J'}(\mathbb{A}^1)$.
\end{proof}

\begin{example}
Consider the Borel adjacent strongly stable ideals $J^\sat = (x_3^2,x_2x_3,x_2^2)$ and $J'^\sat = (x_3^2, x_2x_3,x_1x_3,x_2^3)$ in $\kk[x_0,x_1,x_2,x_3]$ introduced in Example \ref{ex:BorelAdjacent} [BA3]. The ideal $I_{J,J'}$ in the polynomial ring $\kk[y_0,y_1][x_0,x_1,x_2,x_3]$ defining the Borel deformation $X_{J,J'} \to \PP^1$ described in Theorem \ref{thm:mainDef} is generated by the polynomials
\[
\big(\mathfrak{J} \cap \mathfrak{J}' \big) \cup \left\{ y_0\, x_1^2x_2^2 + y_1\, x_1^3x_3, y_0\, x_0x_1 x_2^2 + y_1\, x_0x_1^2 x_3, y_0\, x_0^2 x_2^2 + y_1\, x_0^2 x_1 x_3 \right\},
\]
where $\mathfrak{J} \cap \mathfrak{J}'$ contains the monomials of degree $4$ of the intersection ideal $J^\sat \cap J'^\sat = (x_3^3,x_2x_3,$ $x_2^3)$. Let us consider the restriction $X_{J,J'}\vert_{\mathcal{U}_0} \to \mathbb{A}^1$, where $\mathcal{U}_0 = \PP^1 \setminus \{[0:1]\}$, and the associated $J$-marked basis 
\[
\big(\mathfrak{J} \cap \mathfrak{J}' \big) \cup \left\{ x_1^2x_2^2 +T x_1^3x_3,  x_0x_1 x_2^2 + T x_0x_1^2 x_3, x_0^2 x_2^2 +T x_0^2 x_1 x_3 \right\}.
\]
The Eliahou-Kervaire syzygies of $J$ that we lift are
\begin{align*}
& x_3 \cdot x_1^2 x_2^2 - x_1 \cdot x_1 x_2^2 x_3 = 0 && \leadsto && x_3 (x_1^2x_2^2 +T x_1^3x_3) - x_1 \cdot x_1 x_2^2 x_3 - T x_1 \cdot x_1^2 x_3^2 = 0,\\
& x_2 \cdot x_1^2 x_2^2 - x_1 \cdot x_1 x_2^3 = 0 && \leadsto && x_2 (x_1^2x_2^2 +T x_1^3x_3) - x_1 \cdot x_1 x_2^3 - T x_1 \cdot x_1^2 x_2 x_3 = 0,\\
& x_3 \cdot x_0x_1 x_2^2 - x_0 \cdot x_1 x_2^2 x_3 = 0 && \leadsto && x_3 ( x_0x_1 x_2^2 + T x_0x_1^2 x_3) - x_0 \cdot x_1 x_2^2 x_3 - Tx_0 \cdot x_1^2 x_3^2 = 0,\\
& x_2 \cdot x_0x_1 x_2^2 - x_0 \cdot x_1 x_2^3 = 0 && \leadsto && x_2 ( x_0x_1 x_2^2 + T x_0x_1^2 x_3) - x_0 \cdot x_1 x_2^3 - Tx_0 \cdot x_1^2 x_2 x_3 = 0,\\
& x_1 \cdot x_0x_1 x_2^2 - x_0 \cdot x_1^2 x_2^2 = 0 && \leadsto && x_1 (x_0x_1 x_2^2 + T x_0x_1^2 x_3 ) - x_0 (x_1^2x_2^2 +T x_1^3x_3) = 0,\\
& x_3 \cdot x_0^2 x_2^2 - x_0 \cdot x_0 x_2^2 x_3 = 0 && \leadsto && x_3 (x_0^2 x_2^2 +T x_0^2 x_1 x_3 ) - x_0 \cdot  x_0 x_2^2 x_3 - Tx_0 \cdot x_0 x_1 x_3^2 = 0,\\
& x_2 \cdot  x_0^2 x_2^2- x_0 \cdot x_0 x_2^3 = 0 && \leadsto && x_2 (x_0^2 x_2^2 +T x_0^2 x_1 x_3 ) - x_0 \cdot   x_0 x_2^3 - Tx_0 \cdot x_0 x_1 x_2 x_3= 0,\\
& x_1 \cdot x_0^2 x_2^2- x_0 \cdot  x_0 x_1 x_2^2= 0 && \leadsto && x_1 ( x_0^2 x_2^2 +T x_0^2 x_1 x_3) - x_0 (x_0x_1 x_2^2 + T x_0x_1^2 x_3) = 0.
\end{align*}
\end{example}

\begin{example}
In the polynomial ring $\kk[x_0,x_1,x_2]$, consider the strongly stable ideals $J^\sat = (x_2^2,x_1 x_2,x_1^5)$ and $J'^\sat = (x_2^3,x_1x_2^2,$ $x_1^2 x_2, x_1^3)$ introduced in Example \ref{ex:BorelAdjacent} [nBA2]. The sets
\[
\mathfrak{J} \setminus \mathfrak{J}'= \{x_0^4 x_2^2, x_0^4 x_1 x_2\}\quad\text{and}\quad\mathfrak{J}' \setminus \mathfrak{J}= \{ x_0^2 x_1^4, x_0^3 x_1^3\}
\]
do not satisfy the requirements of the definition of Borel adjacency. Let us show that Eliahou-Kervaire syzygies can not be lifted. Consider the pairing $x_0^4 x_2^2 \leftrightarrow x_0^2 x_1^4$, $x_0^4 x_1 x_2 \leftrightarrow  x_0^3 x_1^3$ (the other pairing leads to analogous problems) and the associated $J$-marked set
\[
\big(\mathfrak{J} \cap \mathfrak{J}'\big) \cup \{ x_0^4 x_2^2 + T x_0^2 x_1^4, x_0^4 x_1 x_2 + T x_0^3 x_1^3  \}.
\]
The syzygy $x_1 \cdot  x_0^4 x_1 x_2 - x_0 \cdot x_0^3 x_1^2 x_2 = 0$ does not lift. In fact,
\[
x_1 (x_0^4 x_1 x_2 + T x_0^3 x_1^3) - x_0  \cdot x_0^3 x_1^2 x_2= T x_0^3 x_1^4 \neq 0\quad\forall\ T \neq 0
\]
and for $T\neq 0$ the ideal defined by the marked set has Hilbert polynomial $p(t) = 5$.
\end{example}

\begin{remark}\label{rk:veronese}
Consider the projective embedding of the Hilbert scheme as subscheme of the Grassmannian via Pl\"ucker coordinates:
\[
\Hilb{p(t)}{n} \subset \mathbf{Gr}\big(p(r),\kk[\xx]_r\big) \stackrel{\mathcal{P}}{\hookrightarrow} \PP^N,\qquad N = \binom{\tbinom{n+r}{n}}{p(r)}-1.
\]
Furthermore, let $J,J' \in \SI{p(t)}{n}$ be two Borel adjacent ideals and let $\varphi_{J,J'}: \PP^1 \to \Hilb{p(t)}{n}$ the morphism of scheme given by the family $X_{J,J'} \to \PP^1 \in \underline{\mathbf{Hilb}}_{p(t)}^n(\PP^1)$. The composition $\mathcal{P}\circ \varphi_{J,J'}$ is a Veronese embedding of degree $d = \vert \mathfrak{J}\setminus\mathfrak{J'}\vert$, so that the image $\mathcal{P}\circ \varphi_{J,J'}(\PP^1)$ is a rational normal curve lying on $\Hilb{p(t)}{n}$. In fact, Pl\"ucker coordinates of $ \mathbf{Gr}\big(p(r),\kk[\xx]_r\big)$ can be indexed by the set of sets of $q(r)$ monomials of degree $r$. Given a $\kk$-rational point $[X] \in \Hilb{p(t)}{n}$, its Pl\"ucker coordinates are (up to a sign) the $q(r)$-minors of the $q(r)\times \binom{n+r}{r}$ matrix representing a basis of $(I_X)_{r}$. If we consider the set of generators \eqref{eq:idealDeformation} of the ideal $I_{J,J'}$, for any closed point $[y_0:y_1] \in \PP^1$, the $q(r)$-minors are either 0 or a monomial $y_0^{v_0} y_1^{v_1}$ with $v_0 + v_1 = \vert\mathfrak{J}\setminus\mathfrak{J}' \vert = d$. And this corresponds exactly to the Veronese embedding of $\PP^1$ of degree $d$. \bs
\end{remark}

\begin{example}
Consider the Borel adjacent ideals $L = (x_2,x_1^3)$ and $J = (x_2^2,x_1 x_2, x_1^2)$ introduced in Example \ref{ex:BorelAdjacent} [BA1]. The Hilbert scheme $\Hilb{3}{2}$ can be seen as subscheme of the Grassmannian $\mathbf{Gr}(3,\kk[x_0,x_1,x_2]_3)$. Via Pl\"ucker embedding, $\mathbf{Gr}(3,\kk[x_0,x_1,x_2]_3)$ is a subscheme of $\PP^{119}$ (see \cite[Section 7.4]{BLMR} for the equations). The Pl\"ucker coordinates (up to a sign) of the Borel deformation $X_{J,L}$ are the minors of order $7$ of the matrix (where $\centerdot$ stands for $0$)
\begin{center}
\begin{tikzpicture}[yscale=1.35,xscale=1.05]
\begin{scope}[shift={(-2.15,1.14)}] \node at (0,0) [rotate=45] {\tiny $x_2^3$}; \end{scope}
\begin{scope}[shift={(-1.625,1.225)}] \node at (0,0) [rotate=45] {\tiny $x_1x_2^2$}; \end{scope}
\begin{scope}[shift={(-1.15,1.225)}] \node at (0,0) [rotate=45] {\tiny $x_1^2x_2$}; \end{scope}
\begin{scope}[shift={(-0.75,1.14)}] \node at (0,0) [rotate=45] {\tiny $x_1^3$}; \end{scope}
\begin{scope}[shift={(-0.2,1.225)}] \node at (0,0) [rotate=45] {\tiny $x_0x_2^2$}; \end{scope}
\begin{scope}[shift={(0.375,1.29)}] \node at (0,0) [rotate=45] {\tiny $x_0 x_1 x_2$}; \end{scope}
\begin{scope}[shift={(0.8125,1.225)}] \node at (0,0) [rotate=45] {\tiny $x_0x_1^2$}; \end{scope}
\begin{scope}[shift={(1.39,1.225)}] \node at (0,0) [rotate=45] {\tiny $x_0^2x_2$}; \end{scope}
\begin{scope}[shift={(1.9,1.225)}] \node at (0,0) [rotate=45] {\tiny $x_0^2x_1$}; \end{scope}
\begin{scope}[shift={(2.25,1.14)}] \node at (0,0) [rotate=45] {\tiny $x_0^3$}; \end{scope}
\node at (0,0) [] {\scriptsize $\left[
\begin{array}{ cccccccccc }
1 & \centerdot & \centerdot & \centerdot & \centerdot & \centerdot & \centerdot & \centerdot & \centerdot & \centerdot \\
\centerdot & 1 & \centerdot & \centerdot & \centerdot & \centerdot & \centerdot & \centerdot & \centerdot & \centerdot \\
\centerdot & \centerdot & 1 & \centerdot & \centerdot & \centerdot & \centerdot & \centerdot & \centerdot & \centerdot \\
\centerdot & \centerdot & \centerdot & 1 & \centerdot & \centerdot & \centerdot & \centerdot & \centerdot & \centerdot \\
\centerdot & \centerdot & \centerdot & \centerdot & 1 & \centerdot & \centerdot & \centerdot & \centerdot & \centerdot \\
\centerdot & \centerdot & \centerdot & \centerdot & \centerdot & 1 & \centerdot & \centerdot & \centerdot & \centerdot \\
\centerdot & \centerdot & \centerdot & \centerdot & \centerdot & \centerdot & y_0 & y_1 & \centerdot & \centerdot \\
\end{array} \right]$};
\end{tikzpicture}
\end{center}
so that the image of the morphism $\PP^1 \xrightarrow{\mathcal{P}\circ \varphi_{J,L}} \PP^{119}$ is a straight line.
\end{example}

\section{The Gr\"obner fan}\label{sec:Groebner fan}

In this section, we look at the whole set of Borel deformations. In particular, we investigate how Borel deformations are related to Gr\"obner deformations.

\begin{definition}\label{def:BorelGraph}
We call \emph{Borel graph} of the Hilbert scheme $\Hilb{p(t)}{n}$ the undirected graph $\mathscr{G}_{p(t)}^n$ whose 
\begin{itemize}
\item vertices $V(\mathscr{G}_{p(t)}^n)$ correspond to strongly stable ideals in $\SI{p(t)}{n}$;
\item edges $E(\mathscr{G}_{p(t)}^n)$ correspond to unordered pairs $\{J,J'\}$ of Borel adjacent strongly stable ideals. 
\end{itemize}
To describe an edge of $\mathscr{G}_{p(t)}^n$, we write $[J_{\aaa} {-} J'_{\aaa'}]$ in order to add the information that $\xx^\aaa$ and $\xx^{\aaa'}$ are the Borel maxima of $\mathfrak{J}\setminus\mathfrak{J}'$ and $\mathfrak{J}'\setminus\mathfrak{J}$.
\end{definition}

\begin{example}
(1) Consider the Hilbert scheme $\Hilb{5}{2}$ parametrizing $0$-dimensional schemes of degree $5$ in the projective plane $\PP^2$. There are 3 strongly stable ideals in $\SI{5}{2}$ and the Borel graph $\Sk{5}{2}$ is a complete graph $K_3$ (see Figure \ref{fig:skeleton}{\sc\subref{fig:skeletonPoints}}).

\smallskip

(2) Consider the Hilbert scheme $\Hilb{3t+1}{3}$ parametrizing $1$-dimensional schemes of degree $3$ and arithmetic genus $0$ in $\PP^3$. There are 3 strongly stable ideals in $\SI{3t+1}{3}$ and the Borel graph $\Sk{3t+1}{2}$ has two edges (see Figure \ref{fig:skeleton}{\sc\subref{fig:skeletonCurves}}).
 \end{example}

\begin{remark}\label{rk:T-graph}
In \cite{AltmannSturmfels} Altmann and Sturmfels define the $T$-graph of a (multigraded) Hilbert scheme as the undirected graph whose vertices correspond to monomial ideals and whose edges correspond to pairs of ideals contained in the closure of a one-dimensional $T$-orbit of the action of the standard torus $T=(\kk^\ast)^{n+1}$ of $\PP^n$. Remark \ref{rk:T-orbit} implies that the Borel graph $\Sk{p(t)}{n}$ is a proper subgraph of the $T$-graph of $\Hilb{p(t)}{n}$. \bs
\end{remark}

\captionsetup[subfloat]{position=bottom}
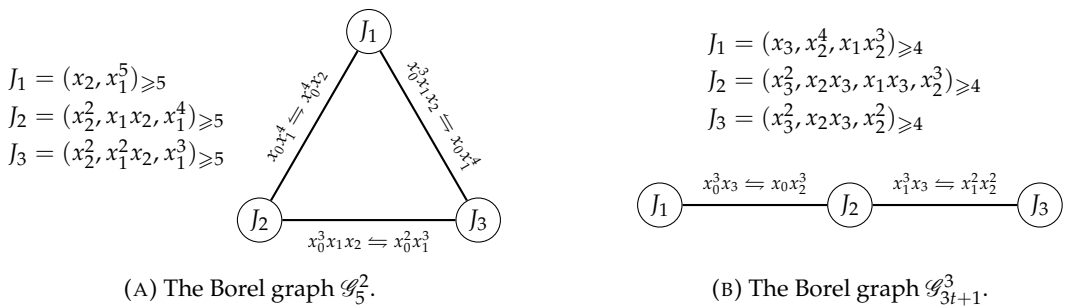
\begin{figure}[!ht]
\begin{center}
\subfloat[][The Borel graph $\Sk{5}{2}$.]{\label{fig:skeletonPoints}
\begin{tikzpicture}
\node at (0,0) [] {\parbox{3cm}{\footnotesize $J_1 = (x_2,x_1^5)_{\geqslant 5}$}};
\node at (0,-0.5) [] {\parbox{3cm}{\footnotesize $J_2 = (x_2^2,x_1x_2,x_1^4)_{\geqslant 5}$}};
\node at (0,-1) [] {\parbox{3cm}{\footnotesize $J_3 = (x_2^2,x_1^2x_2,x_1^3)_{\geqslant 5}$}};

\begin{scope}[shift={(3.25,0)},scale=1.25]
\node (1) at (0,0.5) [circle,draw,inner sep=2pt] {\footnotesize $J_1$};
\node (2) at (-1.15,-1.5) [circle,draw,inner sep=2pt] {\footnotesize $J_2$};
\node (3) at (1.15,-1.5) [circle,draw,inner sep=2pt] {\footnotesize $J_3$};

\draw [thick] (1) --node[rotate=60,above]{\tiny $x_0 x_1^4 \leftrightharpoons x_0^4 x_2$} (2);
\draw [thick] (1) --node[rotate=-60,above]{\tiny $x_0^3 x_1 x_2 \leftrightharpoons x_0 x_1^4$} (3);
\draw [thick] (2) --node[below]{\tiny $x_0^3 x_1 x_2 \leftrightharpoons x_0^2 x_1^3$} (3);
\end{scope}
\end{tikzpicture}
}
\hspace{1.5cm}
\subfloat[][The Borel graph $\Sk{3t+1}{3}$.]{\label{fig:skeletonCurves}
\begin{tikzpicture}

\node at (0,0) [] {\parbox{3.75cm}{\footnotesize $J_1 = (x_3,x_2^4,x_1 x_2^3)_{\geqslant 4}$}};
\node at (0,-0.5) [] {\parbox{3.75cm}{\footnotesize $J_2 = (x_3^2,x_2 x_3, x_1 x_3,x_2^3)_{\geqslant 4}$}};
\node at (0,-1) [] {\parbox{3.75cm}{\footnotesize $J_3 = (x_3^2,x_2x_3,x_2^2)_{\geqslant 4}$}};

\begin{scope}[shift={(-2.5,-2.15)},scale=1.25]
\node (1) at (0,0) [circle,draw,inner sep=2pt] {\footnotesize $J_1$};
\node (2) at (2,0) [circle,draw,inner sep=2pt] {\footnotesize $J_2$};
\node (3) at (4,0) [circle,draw,inner sep=2pt] {\footnotesize $J_3$};
\node at (2,-0.5) [] {};
\draw [thick] (1) --node[above]{\tiny $x_0^3 x_3 \leftrightharpoons x_0 x_2^3$} (2);
\draw [thick] (2) --node[above]{\tiny $x_1^3 x_3 \leftrightharpoons x_1^2 x_2^2$} (3);
\end{scope}
\end{tikzpicture}
}
\caption{The Borel graph of the Hilbert schemes $\Hilb{5}{2}$ and $\Hilb{3t+1}{3}$.}
\label{fig:skeleton}
\end{center}
\end{figure}

In order to investigate properties of an undirected graph (such as connectedness, maximum distance between nodes, \ldots), it is often preferable to assign orientation of the edges and look at it as a directed graph. A natural way to decide the direction of an edge $[J_{\aaa}{-}J'_{\aaa'}]$ is to compare the Borel maxima with a term order on $\mathbb{T}^n$.

Let $J$ and $J'$ be two Borel adjacent ideals, let $\xx^\aaa$ and $\xx^{\aaa'}$ be the Borel maxima of $\mathfrak{J}\setminus\mathfrak{J}'$ and $\mathfrak{J}'\setminus\mathfrak{J}$ and let $\mathcal{E}_{J,J'}$ be the set of compositions of decreasing moves such that
\[
\mathfrak{J}\setminus\mathfrak{J}' = \{ \textsf{E}(\xx^\aaa)\ \vert\ \textsf{E} \in \mathcal{E}_{J,J'}\}\qquad\text{and}\qquad \mathfrak{J}'\setminus\mathfrak{J} = \{ \textsf{E}(\xx^{\aaa'})\ \vert\ \textsf{E} \in \mathcal{E}_{J,J'}\}.
\]
Moreover, consider a term order $\Omega$ and assume that $\xx^\aaa >_{\Omega} \xx^{\aaa'}$. For any element $\textsf{E} \in \mathcal{E}_{J,J'}$, let $\tfrac{\xx^{\underline{p}}}{\xx^{\underline{q}}}$ be the generalized monomial associated to $\textsf{E}$, i.e.~$\textsf{E}(\xx^\aaa) = \tfrac{\xx^{\underline{p}}}{\xx^{\underline{q}}} \xx^{\aaa}$ and $\textsf{E}(\xx^{\aaa'}) = \tfrac{\xx^{\underline{p}}}{\xx^{\underline{q}}} \xx^{\aaa'}$. As $\Omega$ is a multiplicative order, one has
\[
\xx^\aaa >_{\Omega} \xx^{\aaa'}\quad\Rightarrow\quad \textsf{E}(\xx^\aaa) = \tfrac{\xx^{\underline{p}}}{\xx^{\underline{q}}} \xx^{\aaa} >_{\Omega} \tfrac{\xx^{\underline{p}}}{\xx^{\underline{q}}} \xx^{\aaa'} =  \textsf{E}(\xx^{\aaa'}),\qquad\forall\ \textsf{E} \in \mathcal{E}_{J,J'}.
\]
This means that head terms of the polynomials in the $J$-marked basis that generates the ideal $I_{J,\widehat{J'}}$
\[
\left(\mathfrak{J} \cap \mathfrak{J}'\right)  \cup \left\{ \mathsf{E}(\xx^\aaa)+T\mathsf{E}(\xx^{\aaa'})\ \middle\vert\ \mathsf{E} \in \mathcal{E}_{J,J'}\right\}.
\]
are in fact leading terms: $\textsf{E}(\xx^\aaa) = \Ht\big(\mathsf{E}(\xx^\aaa)+T\mathsf{E}(\xx^{\aaa'})\big) = \In{\Omega} \big(\mathsf{E}(\xx^\aaa)+T\mathsf{E}(\xx^{\aaa'})\big),\ \forall\ \textsf{E} \in \mathcal{E}_{J,J'}$. Hence, the $J$-marked basis represents the reduced Gr\"obner basis of $I_{J,\widehat{J'}}$ with respect to $\Omega$ and $J = \In{\Omega}(I_{J,\widehat{J'}})$.

\begin{proposition}\label{prop:grobnerStratumClosure}
Let $\Omega$ be a term order and let $J,J' \subset \kk[\xx]$ be two Borel adjacent strongly stable ideals such that the Borel maximum of $\mathfrak{J}\setminus\mathfrak{J}'$ is greater than the Borel maximum of $\mathfrak{J}'\setminus\mathfrak{J}$ with respect to $\Omega$. The point $[J'] \in \Hilb{p(t)}{n}$ is contained in the closure $\overline{\mathbf{St}_J^\Omega} \subset  \Hilb{p(t)}{n}$.
\end{proposition}
\begin{proof}
We use the same argument of the proof of Corollary \ref{cor:markedSchemeClosure} starting from the observation that  $I_{J,\widehat{J'}} \in \underline{\mathbf{St}}_{J}^\Omega(\mathbb{A}^1)$, where $I_{J,\widehat{J}'}$ is the ideal defining the restriction of the family $X_{J,J'}\vert_{U_0} \to  U_0$, with $U_0 = \PP^1\setminus \{[0:1]\}$. 
\end{proof}

In geometric terms, the proposition says that the ideal $J'$ can be deformed to some ideal $\widetilde{J'}$ such that $\In{\Omega}(\widetilde{J'}) = J$, while there is no deformation $\widetilde{J}$ of $J$ such that $\In{\Omega}(\widetilde{J}) = J'$. From this perspective, we can say that $J'$ is more special or degenerate than $J$ with respect to the term order $\Omega$. For this reason, whenever a point $[J']$ is contained in the closure of a Gr\"obner stratum $\mathbf{St}_J^{\Omega}$, we say that $J'$ is a $\Omega$-degeneration of $J$.

\begin{definition}
Consider a term order $\Omega$. 
We call \emph{$\Omega$-degeneration graph} of the Hilbert scheme $\Hilb{p(t)}{n}$ the directed graph $\mathscr{G}^{n}_{p(t)}(\Omega)$ whose
\begin{itemize}
\item vertices $V\big(\DG{p(t)}{n}{\Omega}\big)$ correspond to strongly stable ideals in $\SI{p(t)}{n}$;
\item edges $E\big(\mathscr{G}_{p(t)}^n(\Omega)\big)$ correspond to ordered pairs $(J,J')$ of Borel adjacent ideals such that $J'$ is a $\Omega$-degeneration of $J$.
\end{itemize}
To describe an edge of $\DG{p(t)}{n}{\Omega}$, we write $[J_{\aaa} {\xrightarrow{\text{\tiny$\Omega$}}} J'_{\aaa'}]$ meaning that $\xx^\aaa$ is the Borel maximum of $\mathfrak{J}\setminus\mathfrak{J}'$, $\xx^{\aaa'}$ is the Borel maximum of $\mathfrak{J}'\setminus\mathfrak{J}$ and $\xx^\aaa >_{\Omega} \xx^{\aaa'}$.
\end{definition}

An immediate consequence of the definition is that every $\Omega$-degeneration graph is a direct acyclic graph (namely, a graph with no oriented cycles). In fact, from the point of view of the generators of the ideals, an edge $[J_{\aaa} {\xrightarrow{\text{\tiny$\Omega$}}} J'_{\aaa'}] \in E\big(\DG{p(t)}{n}{\Omega}\big)$ corresponds to the replacement of some monomials of $\mathfrak{J}$ with smaller monomials with respect to $\Omega$. Consequently, there can not be proper oriented paths in $\DG{p(t)}{n}{\Omega}$ with same initial and final vertex (see Figure \ref{fig:degGraph} for an example).

\begin{figure}[!ht]
\begin{center}
\subfloat[][The degeneration graph $\DG{5}{2}{\mathtt{DegLex}}$.]{\label{fig:degGraphPoints}
\begin{tikzpicture}[scale=1.25,>=angle 60]

\node at (-2.75,0) [] {};
\node at (2.75,0) [] {};
\node (1) at (0,0.5) [circle,draw,inner sep=2pt] {\footnotesize $J_1$};
\node (2) at (-1.15,-1.5) [circle,draw,inner sep=2pt] {\footnotesize $J_2$};
\node (3) at (1.15,-1.5) [circle,draw,inner sep=2pt] {\footnotesize $J_3$};

\draw [thick,->] (1) --node[rotate=60,above]{\tiny $x_0 x_1^4 \leftrightharpoons x_0^4 x_2$} (2);
\draw [thick,->] (1) --node[rotate=-60,above]{\tiny $x_0^3 x_1 x_2 \leftrightharpoons x_0 x_1^4$} (3);
\draw [thick,->] (2) --node[below]{\tiny $x_0^3 x_1 x_2 \leftrightharpoons x_0^2 x_1^3$} (3);
\end{tikzpicture}
}
\hspace{0.5cm}
\subfloat[][The degeneration graph $\DG{3t+1}{3}{\mathtt{RevLex}}$.]{\label{fig:degGraphCurves}
\begin{tikzpicture}[scale=1.25,>=angle 60]
\node at (-0.75,0) [] {};
\node at (4.75,0) [] {};

\node (1) at (0,0) [circle,draw,inner sep=2pt] {\footnotesize $J_1$};
\node (2) at (2,0) [circle,draw,inner sep=2pt] {\footnotesize $J_2$};
\node (3) at (4,0) [circle,draw,inner sep=2pt] {\footnotesize $J_3$};
\node at (2,-0.5) [] {};
\draw [thick,<-] (1) --node[above]{\tiny $x_0^3 x_3 \leftrightharpoons x_0 x_2^3$} (2);
\draw [thick,<-] (2) --node[above]{\tiny $x_1^3 x_3 \leftrightharpoons x_1^2 x_2^2$} (3);
\end{tikzpicture}
}
\caption{The degeneration graphs of $\Hilb{5}{2}$ with respect to the graded lexicogaphic order and of $\Hilb{3t+1}{3}$ with respect to the graded reverse lexicographic order.}
\label{fig:degGraph}
\end{center}
\end{figure}
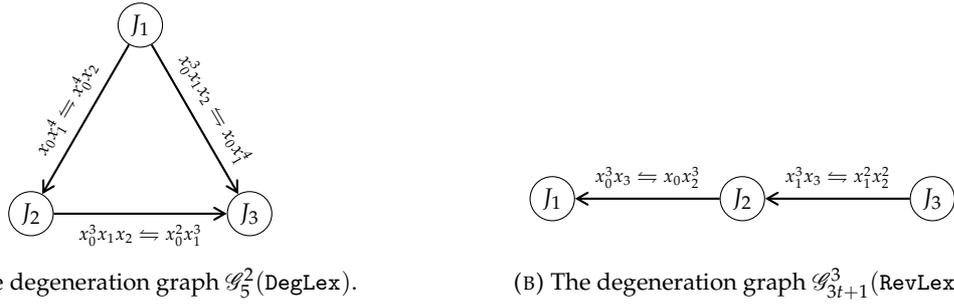

Direct acyclic graphs describe orders of finite sets.
\begin{definition}\label{def:degeneration order}
Let $\Omega$ be a term order and consider the set $\SI{p(t)}{n}$. We denote by $\succeq_\Omega$ the partial order on $\SI{p(t)}{n}$ defined by
\begin{equation*}
J \succeq_\Omega J' \quad \Longleftrightarrow \quad \text{there is a path in $\DG{p(t)}{n}{\Omega}$ with initial vertex $J$ and final vertex $J'$.}
\end{equation*}
\end{definition}

\begin{example}\label{ex:idealsOrder}
Consider the Hilbert schemes $\Hilb{5t-2}{3}$ parametrizing 1-dimensional schemes of degree $5$ and arithmetic genus $3$ in $\PP^3$. The set $\SI{5t-2}{3}$ contains 7 ideals and there are 12 pairs of Borel adjacent ideals, so that the Borel graph $\Sk{5t-2}{3}$ has 7 vertices and 12 edges. In Figure \ref{fig:idealsOrder}, the edges of  $\Sk{5t-2}{3}$ are oriented according to the graded reverse lexicographic order. with respect to the order $\succeq_{\mathtt{RevLex}}$, the set $\SI{5t-2}{3}$ has two maximal elements (ideals $J_6$ and $J_7$ are not comparable) and the minimum (the lexicographic ideal $J_1$).
\end{example}

\begin{figure}[!ht]
\begin{center}
\begin{tikzpicture}

\begin{scope}[shift={(9,2)}]
\node at (0,0) [] {\parbox{4.75cm}{\footnotesize $J_1 = (x_3,x_2^{6},x_1^{3} x_2^{5})_{\geqslant 8}$}};
\node at (0,-0.5) [] {\parbox{4.75cm}{\footnotesize $J_2 = (x_3,x_2^{7},x_1x_2^{6},x_1^{2} x_2^{5})_{\geqslant 8}$}};
\node at (0,-1) [] {\parbox{4.75cm}{\footnotesize $J_3 = (x_3^2, x_2x_3,x_1x_3,x_2^{6},x_1^2 x_2^{5})_{\geqslant 8}$}};
\node at (0,-1.5) [] {\parbox{4.75cm}{\footnotesize $J_4 = (x_3^{2},x_2x_3,x_1^{2}x_3,x_2^{6},x_1x_2^{5})_{\geqslant 8}$}};
\node at (0,-2) [] {\parbox{4.75cm}{\footnotesize $J_5 = (x_3^2, x_2x_3,x_1^{3} x_3,x_2^5)_{\geqslant 8}$}};
\node at (0,-2.5) [] {\parbox{4.75cm}{\footnotesize $J_6 = (x_3^{2},x_2^{2}x_3,x_1x_2x_3,x_1^{2} x_3,x_2^{5})_{\geqslant 8}$}};
\node at (0,-3) [] {\parbox{4.75cm}{\footnotesize $J_7 = (x_3^2,x_2x_3,x_2^{4})_{\geqslant 8}$}};
\end{scope}

\begin{scope}[scale=0.8,shift={(0,0)},>=angle 60]
\node (0) at (-0.5,1) [draw,circle,inner sep=0pt,minimum size=0.6cm] {\footnotesize $J_1$};
\node (1) at (1.5,-2) [draw,circle,inner sep=0pt,minimum size=0.6cm] {\footnotesize $J_2$};
\node (2) at (2,1) [draw,circle,inner sep=0pt,minimum size=0.6cm] {\footnotesize $J_3$};
\node (3) at (3.25,-0.5) [draw,circle,inner sep=0pt,minimum size=0.6cm] {\footnotesize $J_4$};
\node (4) at (3.5,3) [draw,circle,inner sep=0pt,minimum size=0.6cm] {\footnotesize $J_5$};
\node (5) at (5.5,-1) [draw,circle,inner sep=0pt,minimum size=0.6cm] {\footnotesize $J_6$};
\node (6) at (6,2) [draw,circle,inner sep=0pt,minimum size=0.6cm] {\footnotesize $J_7$};

  \draw [->,thick] (1) -- (0);
  \draw [->,thick] (2) -- (1);
  \draw [->,thick] (2) -- (0);
  \draw [->,thick] (3) -- (0);
  \draw [->,thick] (3) -- (2);
  \draw [->,thick] (4) -- (0);
  \draw [->,thick] (4) -- (2);
  \draw [->,thick] (4) -- (3);
  \draw [->,thick] (5) -- (1);
  \draw [->,thick] (5) -- (3);
  \draw [->,thick] (5) -- (4);
  \draw [->,thick] (6) -- (4);

\end{scope}

\end{tikzpicture}
\caption{The $\mathtt{RevLex}$-degeneration graph of the Hilbert scheme $\Hilb{5t-2}{3}$.}
\label{fig:idealsOrder}
\end{center}
\end{figure}
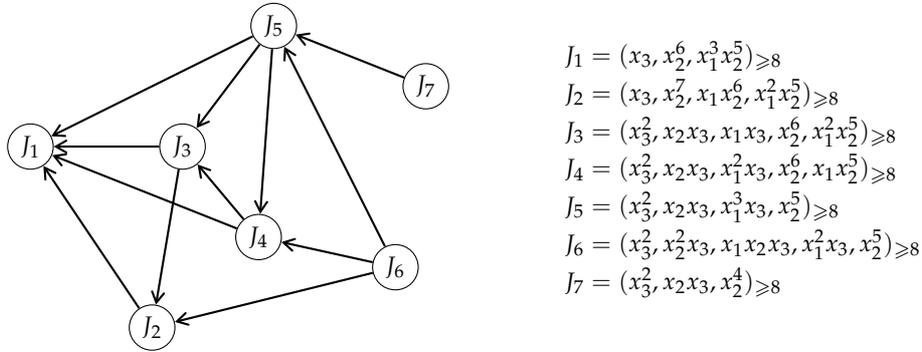
Now, we classify all the possible partial orders on the set $\SI{p(t)}{n}$ defined via degeneration graph. Namely, we classify which directed graphs supported on the the Borel graph of $\Hilb{p(t)}{n}$ can be induced by a term order. We are inspired by the classification of all Gr\"obner bases of a given ideal by means of the Gr\"obner fan (see \cite{MoraRobbiano,SturmfelsGB-CP,FJT}).  We start enlarging the set of monomials orders to weight orders. Given a  vector ${\ooo} \in \RR^{n+1}$, we denote by $\geq_{\ooo}$ the partial order defined by
\[
   \xx^\aaa \geq_{\ooo} \xx^\bbb \qquad \Longleftrightarrow \qquad   \langle \aaa , \ooo \rangle \geqslant \langle \bbb , \ooo\rangle,
\]
where $\langle \cdot, \cdot \rangle$ stands for the standard scalar product.

As $\geq_{\ooo}$ is a partial order on the monomials, it may happen that for an edge $[J_{\aaa} {-} J'_{\aaa'}] \in E(\mathscr{G}_{p(t)}^n)$ of the Borel graph, it holds $\langle \aaa , \ooo \rangle = \langle \aaa' , \ooo\rangle$. This means that a weight order $\geq_{\ooo}$ does not determine the orientation of all edges of $\mathscr{G}_{p(t)}^n$. In such cases, we associate to $\ooo$ a mixed graph.

\begin{definition}
Consider a vector $\ooo\in\RR^{n+1}$.
We call \emph{$\ooo$-degeneration graph} of the Hilbert scheme $\Hilb{p(t)}{n}$ the mixed graph $\mathscr{G}^{n}_{p(t)}(\ooo)$ whose
\begin{itemize}
\item vertices $V\big(\DG{p(t)}{n}{\ooo}\big)$ correspond to strongly stable ideals in $\SI{p(t)}{n}$;
\item undirected edges $E_{\textnormal{u}}\big(\mathscr{G}_{p(t)}^n(\ooo)\big)$ correspond to unordered pairs $\{J,J'\}$ of Borel adjacent ideals such that $\langle \aaa,\ooo\rangle = \langle \aaa',\ooo\rangle$, where $\xx^\aaa$ and $\xx^{\aaa'}$ are the Borel maxima of $\mathfrak{J}\setminus\mathfrak{J}'$ and $\mathfrak{J}'\setminus\mathfrak{J}$; 
\item directed edges $E_{\textnormal{d}}\big(\mathscr{G}_{p(t)}^n(\ooo)\big)$ correspond to ordered pairs $(J,J')$ of Borel adjacent ideals such that $\langle \aaa,\ooo\rangle > \langle \aaa',\ooo\rangle$.
\end{itemize}
To describe an undirected edge of $\mathscr{G}_{p(t)}^n(\ooo)$, we write $[J_{\aaa} \xminus{\ooo}{-2pt} J'_{\bbb}]$, while to describe a directed edge, we write $[J_{\aaa} {\xrightarrow{\text{\tiny$\ooo$}}} J'_{\bbb}]$.
\end{definition}

\begin{remark}
Consider the action of the one-dimensional torus $\kk^\ast$ on $\kk[\xx]$  with weights $\ooo$ defined by
\begin{equation}\label{eq:torusAction}
t \centerdot \xx^\bbb= t^{-\langle \bbb,\ooo\rangle} \xx^\bbb,\qquad\forall\ \xx^\bbb \in \mathbb{T}^n,\ \forall\ t \in \kk^\ast
\end{equation}
and  an undirected edge $[J_{\aaa} \xminus{\ooo}{-2pt} J'_{\aaa'}] \in E_{\textnormal{u}}\big(\mathscr{G}^{n}_{p(t)}(\ooo)\big)$. The ideal $I_{J,J'}$ describing the Borel deformation of $J$ and $J '$ is homogeneous with respect to the $\ZZ$-grading $\xx^\bbb \mapsto \langle \bbb,\ooo\rangle$. Therefore, the rational curve in $\Hilb{p(t)}{n}$ defined by $I_{J,J'}$ is point-wise fixed by the torus action induced by \eqref{eq:torusAction} on the Hilbert scheme.

Instead if we consider a directed edge $[J_{\aaa}{\xrightarrow{\text{\tiny$\ooo$}}} J'_{\aaa'}] \in E_{\textnormal{d}}\big(\mathscr{G}^{n}_{p(t)}(\ooo)\big)$, the ideal $I_{J,J'}$
corresponds to a degeneration using the one-parameter torus action, i.e.
 \[
 J = \lim_{t \to 0} ( t\centerdot I) \qquad \text{and}\qquad  J'=\lim_{t \to \infty}( t\centerdot I)
 \]
 where $I$ is the ideal of the generic fiber of the family $I_{J,J'}$.
 \bs
\end{remark}

\begin{remark}
 Notice that the Borel graph $\Sk{p(t)}{n}$ of $\Hilb{p(t)}{n}$ turns out to coincide with the degeneration graph given by the weight vector $(1,\ldots,1)$. \bs
\end{remark}

\captionsetup[subfloat]{position=bottom,width=0.5\textwidth}
\begin{figure}[!ht]
\begin{center}
\subfloat[][The degeneration graph $\mathscr{G}_5^2\big((0,1,3)\big)$.]{\label{fig:weightDegGraphPoints}
\begin{tikzpicture}[scale=1.25,>=angle 60]

\node at (-2.75,0) [] {};
\node at (2.75,0) [] {};
\node (1) at (0,0.5) [circle,draw,inner sep=2pt] {\footnotesize $J_1$};
\node (2) at (-1.15,-1.5) [circle,draw,inner sep=2pt] {\footnotesize $J_2$};
\node (3) at (1.15,-1.5) [circle,draw,inner sep=2pt] {\footnotesize $J_3$};

\draw [thick,<-] (1) --node[rotate=60,above]{\tiny $x_0 x_1^4 \leftrightharpoons x_0^4 x_2$} (2);
\draw [thick,-] (1) --node[rotate=-60,above]{\tiny $x_0^3 x_1 x_2 \leftrightharpoons x_0 x_1^4$} (3);
\draw [thick,->] (2) --node[below]{\tiny $x_0^3 x_1 x_2 \leftrightharpoons x_0^2 x_1^3$} (3);
\end{tikzpicture}
}
\hspace{0.5cm}
\subfloat[The degeneration graph $\mathscr{G}_{3t+1}^{3}\big((0,1,2,3)\big)$.]{\label{fig:weightDegGraphCurves}
\begin{tikzpicture}[scale=1.25,>=angle 60]
\node at (-0.75,0) [] {};
\node at (4.75,0) [] {};

\node (1) at (0,0) [circle,draw,inner sep=2pt] {\footnotesize $J_1$};
\node (2) at (2,0) [circle,draw,inner sep=2pt] {\footnotesize $J_2$};
\node (3) at (4,0) [circle,draw,inner sep=2pt] {\footnotesize $J_3$};
\node at (2,-0.5) [] {};
\draw [thick,<-] (1) --node[above]{\tiny $x_0^3 x_3 \leftrightharpoons x_0 x_2^3$} (2);
\draw [thick,-] (2) --node[above]{\tiny $x_1^3 x_3 \leftrightharpoons x_1^2 x_2^2$} (3);
\end{tikzpicture}
}
\caption{Examples of degeneration graphs induced by weight orders.}
\label{fig:weightDegGraph}
\end{center}
\end{figure}
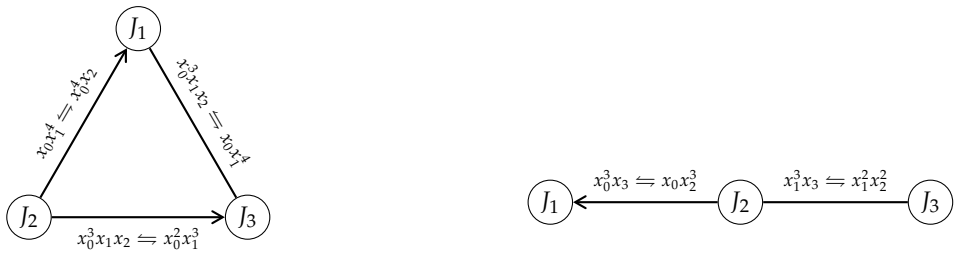

A natural equivalence relation on $\RR^{n+1}$ can be given considering mixed graphs supported on the Borel graph $\mathscr{G}_{p(t)}^n$:
\begin{equation}\label{eq:eqRelation}
\ooo \sim \underline{\sigma} \qquad\Longleftrightarrow\qquad \mathscr{G}_{p(t)}^n(\ooo) = \DG{p(t)}{n}{\underline{\sigma}}. 
\end{equation}

As the number of vertices of the graphs is finite, the number of equivalence classes is finite. A second immediate remark is that equivalence classes are convex cones with vertex in the origin. In fact, consider two vectors $\ooo,\ooo' $ in the same equivalence class $\mathcal{C} \subseteq \RR^{n+1}$. They induce the same orientation of all edges of the Borel graph $\Sk{p(t)}{n}$. Namely, for an edge $[J_{\aaa}{-}J'_{\aaa'}] \in E(\Sk{p(t)}{n})$, we have either $\langle \aaa -\aaa',\ooo \rangle = \langle \aaa - \aaa',\ooo' \rangle = 0$ or $\langle \aaa -\aaa',\ooo \rangle \cdot \langle \aaa - \aaa',\ooo' \rangle > 0$. In the first case, we have an undirected edge and
\[
\langle\aaa - \aaa',T\ooo+(1-T)\ooo' \rangle = 0,\ \forall\ T \qquad \langle\aaa - \aaa',c\ooo\rangle= 0,\ \forall\ c.
\]
In the second case, we have a directed edge. Assuming $\langle \aaa -\aaa',\ooo \rangle> 0$ and $\langle \aaa - \aaa',\ooo' \rangle > 0$, we obtain
\[
\langle\aaa - \aaa',T\ooo+(1-T)\ooo' \rangle > 0,\ \forall\ T \in [0,1]\quad\text{and}\quad \langle\aaa - \aaa',c\ooo\rangle> 0,\ \forall\ c >0.
\]
Hence, vectors $T \ooo + (1-T)\ooo'$ and $c \ooo$ are in $\mathcal{C}$ for every $T \in [0,1]$ and every  $c>0$. 

\begin{lemma}\label{lem:openPolyhedral}
Each equivalence class of vectors is a convex polyhedral cone relatively open, that is open in a suitable affine subspace of $\mathbb{R}^{n+1}$.
\end{lemma}
\begin{proof}
Fix a mixed graph $\mathscr{G}$ supported on the Borel graph of $\Hilb{p(t)}{n}$ and let $\underline{W}=(W_0,\ldots,W_n)$ be the coordinates of $\RR^{n+1}$. For every directed edge $[J_{\aaa} {\to} J'_{\aaa'}] \in E_{\textnormal{d}}(\mathscr{G})$ consider the inequality
\begin{equation}\label{eq:coneIneq}
\langle \aaa - \aaa', \underline{W} \rangle > 0,
\end{equation}
while for every undirected edge $[J_{\aaa} {-} J'_{\aaa'}] \in E_{\textnormal{u}}(\mathscr{G})$ consider the equality
\begin{equation}\label{eq:coneEq}
\langle \aaa-\aaa', \underline{W} \rangle = 0.
\end{equation}
The solutions (if there are) of the system made of inequalities \eqref{eq:coneIneq} and equalities \eqref{eq:coneEq} represent the equivalence class of vectors $\ooo \in \RR^{n+1}$ such that $\mathscr{G}_{p(t)}^n(\ooo) = \mathscr{G}$. The set of solutions is an open subset of the affine subspace of $\RR^{n+1}$ defined by equations \eqref{eq:coneEq}.
\end{proof}

For any term order $\Omega$, all variables $x_i$ are greater than $1$. Since for a weight vector $\ooo$, $x_i >_{\ooo} 1$ if and only if $\omega_i > 0$, we restrict to the positive orthant $\RR^{n+1}_{> 0}$. Notice that there is no loss of information, because each equivalence class $\mathcal{C}$ intersects the positive orthant. In fact, consider vectors $\ooo\in\mathcal{C}$ and $\ooo'=\ooo + c(1,\ldots,1)$. As we are in the homogeneous context, $\ooo$ and $\ooo'$ induces the same weight order. For all $\xx^\aaa,\xx^\bbb$ s.t.~$\vert \aaa \vert = \vert \bbb \vert$,
\[
\begin{split}
\langle  \aaa-\bbb,\ooo'\rangle &{}= \langle \aaa-\bbb, \ooo + c(1,\ldots,1)\rangle = \langle \aaa-\bbb,\ooo \rangle +c \langle\aaa, (1,\ldots,1)\rangle - c \langle\bbb,(1,\ldots,1) \rangle = {}\\
&{} = \langle \aaa-\bbb, \ooo\rangle -c \vert\aaa\vert + c\vert\bbb\vert = \langle  \aaa-\bbb,\ooo\rangle. 
\end{split}
\]
For $c$ sufficiently large, $\underline{\omega}'$ is contained in $\mathcal{C} \cap \RR_{>0}^{n+1}$.

Furthermore, in the definition of Borel-fixed ideals, we need to fix an order among variables and our choice is $x_0 < \cdots < x_n$. Consequently, we are interested in weight vectors $\ooo$ such that $x_0 <_{\ooo} x_1 < _{\ooo} \cdots <_{\ooo} x_n$ (this guarantees that  the weight order $\geq_{\ooo}$ refines the Borel order $\geq_B$). Therefore, we further restrict to the open polyhedral cone
\[
 \mathcal{W} = \left\{ \ooo \in \RR_{> 0}^{n+1}\ \vert\ \omega_{i} < \omega_{i+1},\ i=0,\ldots,n-1\right\}.
\]

We introduce the following notation for the equivalence classes:
\[
\begin{split}
 \text{for a term order~}\Omega,\qquad& \mathcal{C}_{p(t)}^{n}(\Omega) = \left\{ \underline{\sigma} \in \mathcal{W}\ \middle\vert\ \DG{p(t)}{n}{\underline{\sigma}} = \DG{p(t)}{n}{\Omega}\right\},\\ 
\text{for a vector~} \ooo \in \mathcal{W},\qquad& \mathcal{C}_{p(t)}^{n}(\ooo) = \left\{ \underline{\sigma} \in \mathcal{W}\ \middle\vert\ \DG{p(t)}{n}{\underline{\sigma}} = \DG{p(t)}{n}{\ooo}\right\}. 
\end{split}
\]

\begin{definition}\label{def:groebnerFan}
The \emph{Gr\"obner fan} $\GF(\Hilb{p(t)}{n})$ of the Hilbert scheme $\Hilb{p(t)}{n}$ is the set of the closures of all equivalence classes of the relation \eqref{eq:eqRelation} with their proper faces, intersected with the closure $\overline{\mathcal{W}}$ of the cone $\mathcal{W}$.
\end{definition}

\begin{theorem}\label{thm:groebnerFan}
The Gr\"obner fan $\GF(\Hilb{p(t)}{n})$ is a polyhedral fan.
\end{theorem}
\begin{proof}
By Lemma \ref{lem:openPolyhedral}, we know that $\GF(\Hilb{p(t)}{n})$ is a collection of convex polyhedral cones. In order to prove that $\GF(\Hilb{p(t)}{n})$ is a polyhedral fan, we need to show that
\begin{enumerate}[(i)]
\item for every cone $\mathcal{C} \in \GF(\Hilb{p(t)}{n})$, all its faces are contained in $\GF(\Hilb{p(t)}{n})$;
\item for every pair of cones $\mathcal{C}_1,\mathcal{C}_2$, the intersection $\mathcal{C}_1 \cap \mathcal{C}_2$ is a face of $\mathcal{C}_1$ and a face of $\mathcal{C}_2$.
\end{enumerate}
Moreover, as the intersection of a fan with a unique polyhedral cone is still a fan, we show that the set of the closures of all equivalence classes in $\RR^{n+1}$ is a polyhedral fan.

Consider a face $\mathcal{F}$ of the closure $\overline{\mathcal{C}}$ of the equivalence class $\mathcal{C}$. This means that some of the inequalities defining $\overline{\mathcal{C}}$ become equalities when defining $\mathcal{F}$. From the point of view of degeneration graphs, passing from $\mathcal{C}$ to the relative interior of $\mathcal{F}$ means to remove the orientations from edges associated to those inequalities that become equalities. Hence, all the interior points of $\mathcal{F}$ induce the same degeneration graph, i.e.~$\mathcal{F} \subset \overline{\mathcal{C}}'$ for some equivalence class $\mathcal{C}'$. In fact, equality $\mathcal{F} = \overline{\mathcal{C}}'$ holds, because outside of $\mathcal{F}$ the degeneration graph is different. This proves (i).

To prove (ii), consider two cones $\overline{\mathcal{C}}_1$, $\overline{\mathcal{C}}_2$ and the intersection $\mathcal{P} = \overline{\mathcal{C}}_1 \cap \overline{\mathcal{C}}_2$. In the previous paragraph, we showed that every $\ooo'' \in \mathcal{P}$ in contained in a the cone $\overline{\mathcal{C}}''$ that is a face of both $\overline{\mathcal{C}}_1$ and $\overline{\mathcal{C}}_2$. Hence, $\mathcal{P}$ is a finite union of common faces. However, $\mathcal{P}$ is convex and a finite union of cones can only be convex if the union is a singleton. Hence, $\mathcal{P}$ is a common face of $\overline{\mathcal{C}}_1$ and $\overline{\mathcal{C}}_2$.
\end{proof}

If $n=2$, we represent a Gr\"obner fan through the intersection of the fan with the plane $\omega_0 + \omega_1 +\omega_2 = 1$ (each equivalence class intersects such plane).

If $n=3$, we consider the intersection of the fan with the plane $\omega_0 = 0,\  \omega_1 +\omega_2 + \omega_3 = 1$. For each $\ooo \neq (1,\ldots,1) \in \overline{\mathcal{W}}$, the point 
\[
\ooo' = \frac{1}{\omega_3+\omega_2+\omega_1 - 3\omega_0}\left( \ooo - \omega_0 (1,\ldots,1)\right) \in \overline{\mathcal{W}}
\]
lies on the plane and the weight order $\geq_{\ooo'}$ coincides with the weight order $\geq_{\ooo}$. Hence, all equivalence classes are represented except the one corresponding to the Borel graph.

\begin{example}
Consider the Hilbert scheme $\Hilb{5}{2}$ whose Borel graph $\Sk{5}{2}$ is depicted in Figure \ref{fig:skeleton}{\sc\subref{fig:skeletonPoints}}. The orientation of edges depends on the sign of 
\begin{itemize}
\item $[J_1{-}J_2]$ $x_0^4 x_2 \leftrightharpoons x_0 x_1^4$\vspace*{-2pt}
\[
 (4W_0 + W_2) - (W_0 + 4 W_1) = 3W_0 - 4W_1 + W_2\ \gtreqqless\ 0,\vspace*{-2pt}
\]
\item $[J_1{-}J_3]$ $x_0^3 x_1 x_2 \leftrightharpoons x_0 x_1^4$\vspace*{-2pt}
\[
(3W_0 + W_1 + W_2) - (W_0 + 4W_1) = 2W_0 - 3W_1 + W_2\ \gtreqqless\ 0, \vspace*{-2pt}
\]
\item $[J_2{-}J_3]$ $x_0^3 x_1 x_2 \leftrightharpoons x_0^2 x_1^3$\vspace*{-2pt}
\[
(3W_0 + W_1 + W_2) - (2W_0 + 3 W_1) = W_0 - 2W_1 + W_2\ \gtreqqless\ 0.\vspace*{-2pt}
\]
\end{itemize}
It turns out that the fan $\GF(\Hilb{5}{2})$ has 4 cones of maximal dimension and there are $8$ different degeneration graphs (see Figure \ref{fig:gf}).
\end{example}

\begin{figure}[!ht]
\begin{center}
\begin{tikzpicture}[scale=7.5]

              \node at (-.866,-.5) [left] {\tiny $(0,0,1)$};
              \node at (0,-.5) [right] {\tiny $(0,1,1)$};
              \node at (-.288667,-.5) [below,xshift=0.1cm] {\tiny $(0,1,2)$};
              \node at (-.433,-.5) [below,xshift=0.15cm] {\tiny  $(0,1,3)$};
              \node at (-.5196,-.5) [below,xshift=-0.15cm] {\tiny  $(0,1,4)$};
              \node at (0,0) [right] {\tiny $(1,1,1)$};

	      \draw [-] (-.866,-.5) -- (0,-.5) -- (0,0) -- cycle;
              \draw [-] (-.866,-.5) -- (-.5196,-.5) -- (0,0) -- cycle;
              \draw [-] (-.433,-.5) -- (-.5196,-.5) -- (0,0) -- cycle;
              \draw [-] (-.288667,-.5) -- (-.433,-.5) -- (0,0) -- cycle;
              \draw [-] (0,-.5) -- (-.288667,-.5) -- (0,0) -- cycle;

      \end{tikzpicture}
      
      \bigskip
      
\begin{tikzpicture}[scale=3,>=angle 60]
\begin{scope}[shift={(0,0)}]

        \draw [thin,draw=black!50] (-.433,-.5) -- (-.5196,-.5) -- (0,0) -- cycle; 
        \draw [thin,draw=black!50] (-.288667,-.5) -- (-.433,-.5) -- (0,0) -- cycle; 
        \draw [thin,draw=black!50] (0,-.5) -- (-.288667,-.5) -- (0,0) -- cycle; 
        \draw [fill=black!50,draw=black!50] (-.866,-.5) -- (-.5196,-.5) -- (0,0) -- cycle; 
        \draw [thick] (-.5196,-.5)  -- (-.866,-.5) -- (0,0);
        \draw [thick,dashed] (-.5196,-.5)  --  (0,0);
        
        \node (1) at (-0.433,-.65) [circle,draw,inner sep=1pt] {\tiny $J_1$};
        \node (2) at (-0.656,-1.05) [circle,draw,inner sep=1pt] {\tiny $J_2$};
        \node (3) at (-0.203,-1.05) [circle,draw,inner sep=1pt] {\tiny $J_3$};
        \draw [->] (1) -- (2);
        \draw [->] (1) -- (3);
        \draw [->] (2) -- (3);
\end{scope}

\begin{scope}[shift={(1.25,0)}]

        \draw [thin,draw=black!50,fill=black!50] (-.433,-.5) -- (-.5196,-.5) -- (0,0) -- cycle; 
        \draw [thin,draw=black!50] (-.288667,-.5) -- (-.433,-.5) -- (0,0) -- cycle; 
        \draw [thin,draw=black!50] (0,-.5) -- (-.288667,-.5) -- (0,0) -- cycle; 
        \draw [thin,draw=black!50] (-.866,-.5) -- (-.5196,-.5) -- (0,0) -- cycle; 
	\draw[thick] (-.433,-.5) -- (-.5196,-.5);
	\draw[thick,dashed]  (-.433,-.5) -- (0,0);
	\draw[thick,dashed]  (-.5196,-.5) -- (0,0);
        
        \node (1) at (-0.433,-.65) [circle,draw,inner sep=1pt] {\tiny $J_1$};
        \node (2) at (-0.656,-1.05) [circle,draw,inner sep=1pt] {\tiny $J_2$};
        \node (3) at (-0.203,-1.05) [circle,draw,inner sep=1pt] {\tiny $J_3$};
        \draw [<-] (1) -- (2);
        \draw [->] (1) -- (3);
        \draw [->] (2) -- (3);
\end{scope}

\begin{scope}[shift={(2.5,0)}]

        \draw [thin,draw=black!50] (-.433,-.5) -- (-.5196,-.5) -- (0,0) -- cycle; 
        \draw [thin,draw=black!50,fill=black!50] (-.288667,-.5) -- (-.433,-.5) -- (0,0) -- cycle; 
        \draw [thin,draw=black!50] (0,-.5) -- (-.288667,-.5) -- (0,0) -- cycle; 
        \draw [thin,draw=black!50] (-.866,-.5) -- (-.5196,-.5) -- (0,0) -- cycle; 
	\draw[thick] (-.433,-.5) -- (-.288667,-.5);
	\draw[thick,dashed]  (-.433,-.5) -- (0,0);
	\draw[thick,dashed]  (-.288667,-.5) -- (0,0);
        
        \node (1) at (-0.433,-.65) [circle,draw,inner sep=1pt] {\tiny $J_1$};
        \node (2) at (-0.656,-1.05) [circle,draw,inner sep=1pt] {\tiny $J_2$};
        \node (3) at (-0.203,-1.05) [circle,draw,inner sep=1pt] {\tiny $J_3$};
        \draw [<-] (1) -- (2);
        \draw [<-] (1) -- (3);
        \draw [->] (2) -- (3);
\end{scope}

\begin{scope}[shift={(3.75,0)}]

        \draw [thin,draw=black!50] (-.433,-.5) -- (-.5196,-.5) -- (0,0) -- cycle; 
        \draw [thin,draw=black!50] (-.288667,-.5) -- (-.433,-.5) -- (0,0) -- cycle; 
        \draw [thin,draw=black!50,fill=black!50] (0,-.5) -- (-.288667,-.5) -- (0,0) -- cycle; 
        \draw [thin,draw=black!50] (-.866,-.5) -- (-.5196,-.5) -- (0,0) -- cycle; 
	\draw[thick] (-.288667,-.5) -- (0,-.5) -- (0,0);
	\draw[thick,dashed]  (-.288667,-.5) -- (0,0);
        
        \node (1) at (-0.433,-.65) [circle,draw,inner sep=1pt] {\tiny $J_1$};
        \node (2) at (-0.656,-1.05) [circle,draw,inner sep=1pt] {\tiny $J_2$};
        \node (3) at (-0.203,-1.05) [circle,draw,inner sep=1pt] {\tiny $J_3$};
        \draw [<-] (1) -- (2);
        \draw [<-] (1) -- (3);
        \draw [<-] (2) -- (3);
\end{scope}

\end{tikzpicture}

\hspace{0.75cm}

\begin{tikzpicture}[scale=3,>=angle 60]
\begin{scope}[shift={(0,0)}]

        \draw [thin,draw=black!50] (-.433,-.5) -- (-.5196,-.5) -- (0,0) -- cycle; 
        \draw [thin,draw=black!50] (-.288667,-.5) -- (-.433,-.5) -- (0,0) -- cycle; 
        \draw [thin,draw=black!50] (0,-.5) -- (-.288667,-.5) -- (0,0) -- cycle; 
        \draw [draw=black!50] (-.866,-.5) -- (-.5196,-.5) -- (0,0) -- cycle; 
       \draw [thick] (-.5196,-.5)  --  (0,0);
        \draw [thin] (0,0) circle (0.5pt);
        
        \node (1) at (-0.433,-.65) [circle,draw,inner sep=1pt] {\tiny $J_1$};
        \node (2) at (-0.656,-1.05) [circle,draw,inner sep=1pt] {\tiny $J_2$};
        \node (3) at (-0.203,-1.05) [circle,draw,inner sep=1pt] {\tiny $J_3$};
        \draw [-] (1) -- (2);
        \draw [->] (1) -- (3);
        \draw [->] (2) -- (3);
\end{scope}

\begin{scope}[shift={(1.25,0)}]

        \draw [thin,draw=black!50] (-.433,-.5) -- (-.5196,-.5) -- (0,0) -- cycle; 
        \draw [thin,draw=black!50] (-.288667,-.5) -- (-.433,-.5) -- (0,0) -- cycle; 
        \draw [thin,draw=black!50] (0,-.5) -- (-.288667,-.5) -- (0,0) -- cycle; 
        \draw [thin,draw=black!50] (-.866,-.5) -- (-.5196,-.5) -- (0,0) -- cycle; 
	\draw[thick]  (-.433,-.5) -- (0,0);
        \draw [thin] (0,0) circle (0.5pt);
        
        \node (1) at (-0.433,-.65) [circle,draw,inner sep=1pt] {\tiny $J_1$};
        \node (2) at (-0.656,-1.05) [circle,draw,inner sep=1pt] {\tiny $J_2$};
        \node (3) at (-0.203,-1.05) [circle,draw,inner sep=1pt] {\tiny $J_3$};
        \draw [<-] (1) -- (2);
        \draw [-] (1) -- (3);
        \draw [->] (2) -- (3);
\end{scope}

\begin{scope}[shift={(2.5,0)}]

        \draw [thin,draw=black!50] (-.433,-.5) -- (-.5196,-.5) -- (0,0) -- cycle; 
        \draw [thin,draw=black!50] (-.288667,-.5) -- (-.433,-.5) -- (0,0) -- cycle; 
        \draw [thin,draw=black!50] (0,-.5) -- (-.288667,-.5) -- (0,0) -- cycle; 
        \draw [thin,draw=black!50] (-.866,-.5) -- (-.5196,-.5) -- (0,0) -- cycle; 
        	\draw[thick]  (-.288667,-.5) -- (0,0);
        \draw [thin] (0,0) circle (0.5pt);
        
        \node (1) at (-0.433,-.65) [circle,draw,inner sep=1pt] {\tiny $J_1$};
        \node (2) at (-0.656,-1.05) [circle,draw,inner sep=1pt] {\tiny $J_2$};
        \node (3) at (-0.203,-1.05) [circle,draw,inner sep=1pt] {\tiny $J_3$};
        \draw [<-] (1) -- (2);
        \draw [<-] (1) -- (3);
        \draw [-] (2) -- (3);
\end{scope}

      \begin{scope}[shift={(3.75,0)}]

        \draw [thin,draw=black!50] (-.433,-.5) -- (-.5196,-.5) -- (0,0) -- cycle; 
        \draw [thin,draw=black!50] (-.288667,-.5) -- (-.433,-.5) -- (0,0) -- cycle; 
        \draw [thin,draw=black!50] (0,-.5) -- (-.288667,-.5) -- (0,0) -- cycle; 
        \draw [thin,draw=black!50] (-.866,-.5) -- (-.5196,-.5) -- (0,0) -- cycle; 
        \draw [thin,fill] (0,0) circle (0.4pt);
        
        \node (1) at (-0.433,-.65) [circle,draw,inner sep=1pt] {\tiny $J_1$};
        \node (2) at (-0.656,-1.05) [circle,draw,inner sep=1pt] {\tiny $J_2$};
        \node (3) at (-0.203,-1.05) [circle,draw,inner sep=1pt] {\tiny $J_3$};
        \draw [-] (1) -- (2);
        \draw [-] (1) -- (3);
        \draw [-] (2) -- (3);
\end{scope}

      \end{tikzpicture}

\caption{The Gr\"obner fan of the Hilbert scheme $\Hilb{5}{2}$ and the possible degeneration graphs.}
\label{fig:gf}
\end{center}
\end{figure}
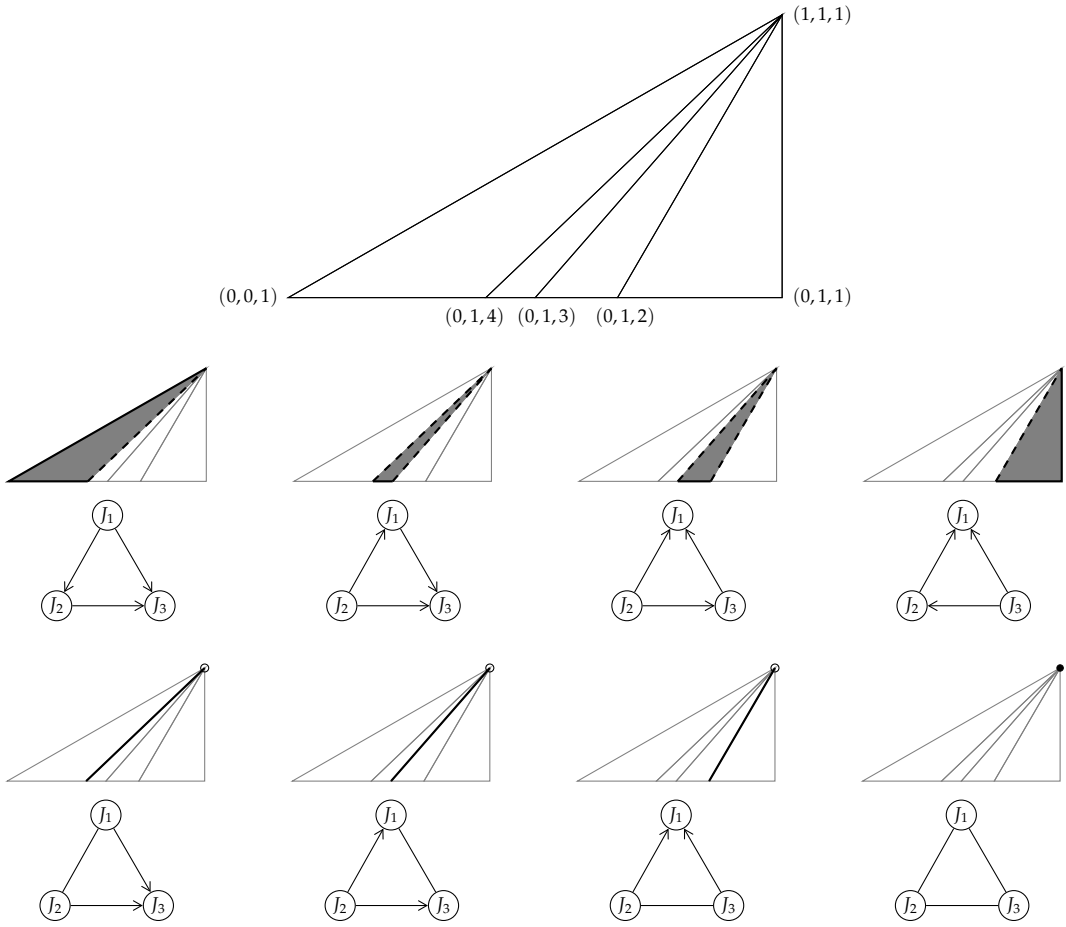

\captionsetup[subfloat]{position=bottom,width=0.7\textwidth}
\begin{table}[!ht]
\begin{center}
\subfloat[][Examples of computation of Gr\"obner fans of Hilbert schemes parametrizing 0-dimensional subschemes in $\PP^2$, $\PP^3$ and $\PP^4$.]{
\begin{tikzpicture}[yscale=-0.8,xscale=1.4]
\node at (1.75,0) [] {\parbox{1cm}{\centering\scriptsize vertices\\[-2pt] of $\Sk{d}{n}$}};
\node at (3.25,0) [] {\parbox{1cm}{\centering\scriptsize  edges\\[-2pt] of $\Sk{d}{n}$}};
\node at (3.25,-0.8) [] {\small Borel graph};
\node at (4.75,0) [] {\parbox{1.5cm}{\centering\scriptsize approximate \\[-2.5pt] cpu time }};

\node at (6.25,0) [] {\parbox{2cm}{\centering\scriptsize maximal cones\\[-2pt] of $\GF(\Hilb{d}{n})$}};
\node at (7.75,0) [] {\parbox{2cm}{\centering\scriptsize extremal rays\\[-2pt] of $\GF(\Hilb{d}{n})$}};
\node at (9.25,0) [] {\parbox{2cm}{\centering\scriptsize approximate \\[-2.5pt] cpu time }};
\node at (7.75,-0.85) [] {\small Gr\"obner fan};

\draw[thick] (1,-0.5) -- (10,-0.5);
\draw[thick] (-0.5,0.5) -- (10,0.5);

\draw[] (-0.5,1.3) -- (10,1.3);
\draw[] (-0.5,2.1) -- (10,2.1);
\draw[] (-0.5,2.9) -- (10,2.9);
\draw[] (-0.5,3.7) -- (10,3.7);
\draw[] (-0.5,4.5) -- (10,4.5);
\draw[] (-0.5,5.3) -- (10,5.3);
\draw[] (-0.5,6.1) -- (10,6.1);
\draw[] (-0.5,6.9) -- (10,6.9);
\draw[thick] (-0.5,7.7) -- (10,7.7);

\draw[thick] (1,-0.5) -- (1,7.7);
\draw[thick] (5.5,-0.5) -- (5.5,7.7);
\draw[thick] (10,-0.5) -- (10,7.7);
\draw[thick] (-0.5,0.5) -- (-0.5,7.7);

\node at (0.25,0.9) [] {\parbox{1.5cm}{\footnotesize $\PP^2$, \hfill $d=5$}};
\node at (1.75,0.9) [] {\small $3$};
\node at (3.25,0.9) [] {\small $3$};
\node at (4.75,0.9) [] {\small $0.009\, s$};
\node at (6.25,0.9) [] {\small $4$};
\node at (7.75,0.9) [] {\small $6$};
\node at (9.25,0.9) [] {\small $0.07\, s$};

\node at (0.25,1.7) [] {\parbox{1.5cm}{\footnotesize $\PP^3$, \hfill$d=5$}};
\node at (1.75,1.7) [] {\small $4$};
\node at (3.25,1.7) [] {\small $5$};
\node at (4.75,1.7) [] {\small $0.02\, s$};
\node at (6.25,1.7) [] {\small $10$};
\node at (7.75,1.7) [] {\small $12$};
\node at (9.25,1.7) [] {\small $0.1\, s$};

\node at (0.25,2.5) [] {\parbox{1.5cm}{\footnotesize $\PP^4$, \hfill$d=5$}};
\node at (1.75,2.5) [] {\small $5$};
\node at (3.25,2.5) [] {\small $6$};
\node at (4.75,2.5) [] {\small $0.06\, s$};
\node at (6.25,2.5) [] {\small $11$};
\node at (7.75,2.5) [] {\small $14$};
\node at (9.25,2.5) [] {\small $0.2\, s$};

\node at (0.25,3.3) [] {\parbox{1.5cm}{\footnotesize $\PP^2$, \hfill$d=8$}};
\node at (1.75,3.3) [] {\small $6$};
\node at (3.25,3.3) [] {\small $10$};
\node at (4.75,3.3) [] {\small $0.04\, s$};
\node at (6.25,3.3) [] {\small $8$};
\node at (7.75,3.3) [] {\small $10$};
\node at (9.25,3.3) [] {\small $0.09\, s$};

\node at (0.25,4.1) [] {\parbox{1.5cm}{\footnotesize $\PP^3$, \hfill$d=8$}};
\node at (1.75,4.1) [] {\small $12$};
\node at (3.25,4.1) [] {\small $31$};
\node at (4.75,4.1) [] {\small $0.4\, s$};
\node at (6.25,4.1) [] {\small $70$};
\node at (7.75,4.1) [] {\small $55$};
\node at (9.25,4.1) [] {\small $1.8\, s$};

\node at (0.25,4.9) [] {\parbox{1.5cm}{\footnotesize $\PP^4$, \hfill$d=8$}};
\node at (1.75,4.9) [] {\small $16$};
\node at (3.25,4.9) [] {\small $45$};
\node at (4.75,4.9) [] {\small $4.3\, s$};
\node at (6.25,4.9) [] {\small $310$};
\node at (7.75,4.9) [] {\small $162$};
\node at (9.25,4.9) [] {\small $15.5\, s$};

\node at (0.25,5.7) [] {\parbox{1.5cm}{\footnotesize $\PP^2$, \hfill$d=11$}};
\node at (1.75,5.7) [] {\small $12$};
\node at (3.25,5.7) [] {\small $33$};
\node at (4.75,5.7) [] {\small $0.2\, s$};
\node at (6.25,5.7) [] {\small $14$};
\node at (7.75,5.7) [] {\small $16$};
\node at (9.25,5.7) [] {\small $0.3\, s$};

\node at (0.25,6.5) [] {\parbox{1.5cm}{\footnotesize $\PP^3$, \hfill$d=11$}};
\node at (1.75,6.5) [] {\small $32$};
\node at (3.25,6.5) [] {\small $134$};
\node at (4.75,6.5) [] {\small $12.1\, s$};
\node at (6.25,6.5) [] {\small $259$};
\node at (7.75,6.5) [] {\small $186$};
\node at (9.25,6.5) [] {\small $28.5\, s$};

\node at (0.25,7.3) [] {\parbox{1.5cm}{\footnotesize $\PP^4$,\hfill $d=11$}};
\node at (1.75,7.3) [] {\small $50$};
\node at (3.25,7.3) [] {\small $235$};
\node at (4.75,7.3) [] {\small $320\, s$};
\node at (6.25,7.3) [] {\small $3678$};
\node at (7.75,7.3) [] {\small $1761$};
\node at (9.25,7.3) [] {\small $1131\, s$};

\end{tikzpicture}
}

\subfloat[][Examples of computation of Gr\"obner fans of Hilbert schemes parametrizing 1-dimensional subschemes in $\PP^3$.]{\label{tab:gfan cpuTime curves}
\begin{tikzpicture}[yscale=-0.8,xscale=1.4]
\node at (1.75,0) [] {\parbox{1cm}{\centering\scriptsize vertices\\[-2pt] of $\Sk{d}{n}$}};
\node at (3.25,0) [] {\parbox{1cm}{\centering\scriptsize  edges\\[-2pt] of $\Sk{d}{n}$}};
\node at (3.25,-0.8) [] {\small Borel graph};
\node at (4.75,0) [] {\parbox{1.5cm}{\centering\scriptsize approximate \\[-2.5pt] cpu time }};

\node at (6.25,0) [] {\parbox{2cm}{\centering\scriptsize maximal cones\\[-2pt] of $\GF(\Hilb{d}{n})$}};
\node at (7.75,0) [] {\parbox{2cm}{\centering\scriptsize extremal rays\\[-2pt] of $\GF(\Hilb{d}{n})$}};
\node at (9.25,0) [] {\parbox{2cm}{\centering\scriptsize approximate \\[-2.5pt] cpu time }};
\node at (7.75,-0.85) [] {\small Gr\"obner fan};

\draw[thick] (1,-0.5) -- (10,-0.5);
\draw[thick] (-0.5,0.5) -- (10,0.5);

\draw[] (-0.5,1.3) -- (10,1.3);
\draw[] (-0.5,2.1) -- (10,2.1);
\draw[] (-0.5,2.9) -- (10,2.9);
\draw[] (-0.5,3.7) -- (10,3.7);
\draw[thick] (-0.5,4.5) -- (10,4.5);

\draw[thick] (1,-0.5) -- (1,4.5);
\draw[thick] (5.5,-0.5) -- (5.5,4.5);
\draw[thick] (10,-0.5) -- (10,4.5);
\draw[thick] (-0.5,0.5) -- (-0.5,4.5);

\node at (0.25,0.9) [] {\parbox{1.85cm}{\footnotesize $p(t) = 3t+1$}};
\node at (1.75,0.9) [] {\small $3$};
\node at (3.25,0.9) [] {\small $2$};
\node at (4.75,0.9) [] {\small $0.008\, s$};
\node at (6.25,0.9) [] {\small $3$};
\node at (7.75,0.9) [] {\small $7$};
\node at (9.25,0.9) [] {\small $0.03\, s$};

\node at (0.25,1.7) [] {\parbox{1.85cm}{\footnotesize $p(t) = 4t$}};
\node at (1.75,1.7) [] {\small $4$};
\node at (3.25,1.7) [] {\small $4$};
\node at (4.75,1.7) [] {\small $0.02\, s$};
\node at (6.25,1.7) [] {\small $5$};
\node at (7.75,1.7) [] {\small $9$};
\node at (9.25,1.7) [] {\small $0.04\, s$};

\node at (0.25,2.5) [] {\parbox{1.85cm}{\footnotesize $p(t) = 5t-2$}};
\node at (1.75,2.5) [] {\small $7$};
\node at (3.25,2.5) [] {\small $12$};
\node at (4.75,2.5) [] {\small $0.1\, s$};
\node at (6.25,2.5) [] {\small $18$};
\node at (7.75,2.5) [] {\small $19$};
\node at (9.25,2.5) [] {\small $0.3\, s$};

\node at (0.25,3.3) [] {\parbox{1.85cm}{\footnotesize $p(t) = 6t-3$}};
\node at (1.75,3.3) [] {\small $31$};
\node at (3.25,3.3) [] {\small $110$};
\node at (4.75,3.3) [] {\small $14\, s$};
\node at (6.25,3.3) [] {\small $268$};
\node at (7.75,3.3) [] {\small $186$};
\node at (9.25,3.3) [] {\small $22\, s$};

\node at (0.25,4.1) [] {\parbox{1.85cm}{\footnotesize $p(t) = 7t-5$}};
\node at (1.75,4.1) [] {\small $112$};
\node at (3.25,4.1) [] {\small $651$};
\node at (4.75,4.1) [] {\small $588\, s$};
\node at (6.25,4.1) [] {\small $1204$};
\node at (7.75,4.1) [] {\small $806$};
\node at (9.25,4.1) [] {\small $542\, s$};

\end{tikzpicture}
}
\caption{Examples of computation of Borel graphs and Gr\"obner fans. The code is implemented with {\it Macaulay2} \cite{M2} in the package \texttt{GroebnerFanHilbertScheme.m2} available at  the web page \href{http://www.paololella.it/publications/kl/}{\tt www.paololella.it/publications/kl/}. 
 The algorithms have been run on a MacBook Pro with an Intel Core i5 dual-core 2.9 GHz processor. }
\label{tab:gfan cpuTime}	
\end{center}
\end{table}

\begin{proposition}\label{prop:maxConesTO}
For every term order $\Omega$, there exists $\ooo \in \mathcal{W}$ such that
\[
\mathcal{C}_{p(t)}^{n}(\Omega) = \mathcal{C}_{p(t)}^{n}(\ooo).
\]
Furthermore, $\mathcal{C}_{p(t)}^{n}(\Omega)$ is an open polyhedral cone of maximal dimension.
\end{proposition}
\begin{proof}
The $\Omega$-degeneration graph $\DG{p(t)}{n}{\Omega}$ is a directed graph. Thus, the equivalence class $\mathcal{C}_{p(t)}^{n}(\Omega)$ is defined only by strict inequalities:
\[
\mathcal{C}_{p(t)}^{n}(\Omega) = \left\{ \underline{\sigma} \in \mathcal{W}\ \middle\vert\ \langle \aaa - \aaa', \underline{\sigma} \rangle > 0,\ \forall\  [J_{\aaa}{\xrightarrow{\text{\tiny$\Omega$}}} J'_{\aaa'}] \in E_{\textnormal{d}}\big(\DG{p(t)}{n}{\Omega}\big)\right\}.
\]
The statement is proved if we can show that $\mathcal{C}_{p(t)}^{n}(\Omega)$ is not empty. In order to prove the claim, we recall that every term order $\Omega$ can be described by means of a rational full rank $(n+1)\times (n+1)$ matrix $M_\Omega$ (see \cite{RobbianoKreuzer1,RobbianoTO}) satisfying the following property:
\[
\xx^\aaa \geq_\Omega \xx^\bbb \qquad \Longleftrightarrow\qquad \begin{array}{l} \xx^\aaa = \xx^\bbb\quad \text{or}\\ \text{first non zero entry of~} M_\Omega(\aaa-\bbb)^T \text{~is positive.}\end{array}
\]
Let $R_0,\ldots,R_n$ be the rows of a matrix $M_\Omega$ representing the term order $\Omega$. We construct an element of $\mathcal{C}_{p(t)}^{n}(\Omega)$ as a linear combination $\lambda_0 R_0 + \cdots + \lambda_n R_n$.
Consider sets $E_i,\ i = 0,\ldots,n$ defined by  
\[
E_i = \left\{ [J_\aaa {\xrightarrow{\text{\tiny$\Omega$}}} J'_{\aaa'}] \in  E_{\textnormal{d}}\big(\DG{p(t)}{n}{\Omega}\big)\ \middle\vert\ \langle  \aaa - \aaa',R_i  \rangle > 0 \text{~and~} \langle  \aaa - \aaa',R_j  \rangle = 0 \text{~for~} j<i \right\},
\] 
and $X_i,\ i = 0,\ldots,n$ defined by
\[
X_i = \left\{ x_k\ \middle\vert\ R_{i,k} > R_{i,k-1}\text{~and~} R_{j,k} = R_{j,k-1}\text{~for~} j < i \right\}.
\]
Sets $\{E_i\}_{i=0,\ldots,n}$ represent a partition of the set of edges of $\DG{p(t)}{n}{\Omega}$ and sets $\{X_i\}_{i=0,\ldots,n}$ represent a partition of the set of variables $\{x_1,\ldots,x_n\}$. The set $X_i$ contains the variables $x_k$ such that the $i$-th row of $M_\Omega$ is the row giving the order relation $x_k >_{\Omega} x_{k-1}$.
Then, let $s = \max \{i\ \vert\ E_i \neq \emptyset\text{~or~} X_i \neq 0\}$ and set $\lambda_s = 1, \lambda_i = 0,\ i > s$. Assuming to have fixed a value for the last $n-i$ coefficients $\lambda_{i+1},\ldots,\lambda_n$ ($i < s$), we choose  
\[
\lambda_i = \begin{cases}
0,& \text{if~} E_i = \emptyset, X_i=\emptyset,\\
\lambda_i' + 1,&  \text{if~}  E_i\neq\emptyset, X_i=\emptyset,\\
\lambda_i'' + 1,&  \text{if~}  E_i=\emptyset,  X_i\neq\emptyset,\\
\max\{\lambda_i' ,\lambda_i''\}+1,&  \text{if~}  E_i\neq\emptyset, X_i\neq\emptyset,
\end{cases}
\]
 where
\[
\begin{split}
\lambda_i' &{}= \max \left\{  -\frac{1}{\langle \aaa - \aaa', R_i \rangle} \sum_{j = i+1}^n \lambda_j \langle \aaa - \aaa', R_j \rangle\  \middle\vert\ [J_\aaa {\xrightarrow{\text{\tiny$\Omega$}}} J'_{\aaa'}] \in E_i \right\},\\
\lambda_i'' &{}= \max \left\{  -\frac{1}{R_{i,k} - R_{i,k-1}} \sum_{j = i+1}^n \lambda_j (R_{j,k}-R_{j,k-1})\ \middle\vert\ x_k \in X_i \right\}.
\end{split}
\]
Then, consider $\ooo = \lambda_0 R_0 + \cdots + \lambda_n R_n$.  For every edge $[J_\aaa {\xrightarrow{\text{\tiny$\Omega$}}} J'_{\aaa'}] \in E_{\textnormal{d}}\big(\DG{p(t)}{n}{\Omega}\big)$, it holds $\langle \aaa - \aaa',\ooo\rangle > 0$. In fact, the edge $[J_\aaa {\xrightarrow{\text{\tiny$\Omega$}}} J'_{\aaa'}]$ belongs to $E_i$ for some $i$, so that
\[
\langle \aaa - \aaa', \ooo \rangle = \Big\langle \aaa - \aaa', \sum_{j=0}^n \lambda_j R_j\Big\rangle = \sum_{j=0}^n \lambda_j  \langle \aaa - \aaa', R_j\rangle =  \sum_{j=i}^n \lambda_j  \langle \aaa - \aaa', R_j\rangle.
\]
By the choice of $\lambda_i$, we have
\[
\lambda_i > \lambda_i' \geqslant- \frac{1}{\langle \aaa - \aaa', R_i \rangle} \sum_{j = i+1}^n \lambda_j \langle \aaa - \aaa', R_j \rangle
\]
and, since $\langle \aaa - \aaa', R_i \rangle > 0$,
\[
\lambda_i \langle \aaa - \aaa', R_i \rangle > - \sum_{j = i+1}^n \lambda_j \langle \aaa - \aaa', R_j \rangle \qquad \Longleftrightarrow \qquad \sum_{j=i}^n \lambda_j  \langle \aaa - \aaa', R_j\rangle > 0.
\]

Moreover, $\ooo$ satisfies inequalities $\omega_k > \omega_{k-1}$ for $k=1,\ldots,n$:
\[
\omega_k - \omega_{k-1} = \sum_{j=0}^n \lambda_j (R_{j,k}-R_{j,k-1}) =  \sum_{j=i}^n \lambda_j (R_{j,k}-R_{j,k-1}),\ \forall\ x_k \in X_i
\]
and by the choice of $\lambda_i$, we have
\[
\lambda_i > \lambda_i'' \geqslant -\frac{1}{R_{i,k} - R_{i,k-1}} \sum_{j = i+1}^n \lambda_j (R_{j,k}-R_{j,k-1}).
\]
If $\omega_0 > 0$, then $\ooo \in \mathcal{C}_{p(t)}^n(\Omega)$, otherwise if $\omega_0 \leqslant 0$, $\ooo + (1-\omega_0)(1,\ldots,1) \in \mathcal{C}_{p(t)}^n(\Omega)$.
\end{proof}

The previous statement can be easily reversed. Given an equivalence class $\mathcal{C}$ of maximal dimension, we can produce a term order $\Omega$ such that $\mathcal{C} = \mathcal{C}_{p(t)}^n(\Omega)$ as follows. Pick $\ooo \in \mathcal{C}$ and define the term order $\Omega$ as follows
\begin{equation}\label{eq:TOfromWO}
\xx^\aaa \geq_{\Omega} \xx^\bbb \qquad\Longleftrightarrow \qquad \begin{array}{l} \langle \aaa , \ooo\rangle > \langle \bbb , \ooo\rangle \quad \text{or}\\  \langle \aaa , \ooo\rangle = \langle \bbb , \ooo\rangle \text{~and~} \xx^\aaa \geq_{\Lambda} \xx^\bbb\end{array}
\end{equation}
where $\Lambda$ is an arbitrary term order used as a \lq\lq tie breaker\rq\rq. 

\begin{remark}
In general, a cone of codimension $k$ (not contained in the boundary of $\mathcal{W}$) corresponds to the closure of an equivalence class $\mathcal{C}_{p(t)}^n(\ooo)$ such that the $\ooo$-degeneration graph has at least $k$ undirected edges.
\bs
\end{remark}

\begin{example}
Consider the Hilbert scheme $\Hilb{3t+1}{3}$ and its $\mathtt{RevLex}$-degeneration graph represented in Figure \ref{fig:degGraph}{\sc\subref{fig:degGraphCurves}}. A matrix describing the graded reverse lexicographic order for the polynomial ring $\kk[x_0,x_1,x_2,x_3]$, with the choice $x_0 < x_1 < x_2 < x_3$, is
\[
M_{\mathtt{RevLex}} = \left[\begin{array}{cccc}
1 & 1 & 1 & 1 \\ -1 & 0 & 0 & 0 \\ 0 & -1 & 0 & 0 \\ 0 & 0  & -1 & 0
\end{array}\right].
\]
Going through the lines of the proof of Proposition \ref{prop:maxConesTO}, we consider partitions
\begin{align*}
&E_0 = \emptyset, && E_1 = \{ [J_2 {\to} J_1]\}, && E_2 = \{[J_3 {\to} J_2]\}, && E_3 = \emptyset, \\
&X_0 = \emptyset, && X_1 = \{x_1\}, && X_2 =\{x_2\} , && X_3 = \{x_3\},
\end{align*}
and we start setting $\lambda_3 = 1$. Then, we have 
\[
\begin{split}
&\lambda_2 = \max\{\lambda_2',\lambda_2''\}+1= \max\{2,1\}+1 = 3,\\
&\lambda_1 = \max\{\lambda_1',\lambda_1''\}+1= \max\left\{\tfrac{3}{2},3\right\}+1 = 4,\\
&\lambda_0 = 0.\\
\end{split}
\]
We obtain $\ooo = 4(-1,0,0,0) + 3 (0,-1,0,0) + (0,0,-1,0) = (-4,-3,-1,0)$ and 
\[
\mathcal{C}_{3t+1}^3(\mathtt{RevLex}) =\mathcal{C}_{3t+1}^3\big((1,2,4,5)\big) = \mathcal{C}_{3t+1}^3\big((0,1,3,4)\big).
\]

In the case of the graded lexicographic order, the edges of the $\mathtt{DegLex}$-degeneration graph are $[J_1 {\to} J_2]$ and $[J_2 {\to} J_3]$, a matrix representing the term order is
\[
M_{\mathtt{DegLex}} = \left[\begin{array}{cccc}
1 & 1 & 1 & 1 \\ 0 & 0 & 0 & 1 \\ 0 & 0 & 1 & 0 \\ 0 & 1  & 0 & 0
\end{array}\right]
\]
and partitions are
\begin{align*}
&E_0 = \emptyset, && E_1 = \{ [J_1 {\to} J_2],[J_2 {\to} J_3]\}, && E_2 = \emptyset, && E_3 = \emptyset, \\
&X_0 = \emptyset, && X_1 = \{x_3\}, && X_2 =\{x_2\} , && X_3 = \{x_1\}.
\end{align*}
We have $\lambda_3 = 1$, $\lambda_2 = \lambda_2''+1 = 2$, $\lambda_1 = \max\{\lambda_1',\lambda_1''\}+1 = \max\{6,2\}+1 = 7$ and $\lambda_0 = 0$, so that
\[
\mathcal{C}_{3t+1}^3(\mathtt{DegLex}) =\mathcal{C}_{3t+1}^3\big((1,2,3,8)\big) = \mathcal{C}_{3t+1}^3\big((0,1,2,7)\big).
\]
\end{example}

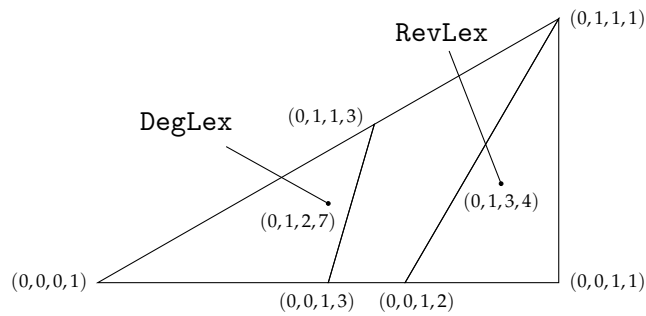
\begin{figure}[!ht]
\begin{center}
\begin{tikzpicture}[scale=7]

	\node at (-.866,-.5) [left] {\tiny $(0,0,0,1)$}; 
        \node at (0,-.5) [right] {\tiny $(0,0,1,1)$}; 
        \node at (-.288667,-.5) [below,xshift=0.15cm] {\tiny $(0,0,1,2)$}; 
        \node at (-.433,-.5) [below,xshift=-0.15cm] {\tiny $(0,0,1,3)$}; 
        \node at (0,0) [right] {\tiny $(0,1,1,1)$}; 
        \node at (-.3464,-.2) [left,yshift=0.1cm] {\tiny $(0,1,1,3)$}; 

        \draw [-] (-.866,-.5) -- (-.433,-.5) -- (-.3464,-.2) -- cycle; 
        \draw [-] (-.288667,-.5) -- (-.433,-.5) --  (-.3464,-.2) --  (0,0) -- cycle; 
        \draw [-] (0,-.5) -- (-.288667,-.5) -- (0,0) -- cycle; 
	
	\draw[fill] (-.433, -.35) circle (0.1 pt);
	\draw[fill]  (-.10825, -.3125) circle (0.1 pt);
	
	\node at (-.433, -.35) [xshift=-0.4cm,yshift=-0.25cm] {\tiny $(0,1,2,7)$};
	\node at (-.10825, -.3125) [yshift=-0.25cm] {\tiny $(0,1,3,4)$};
	
	\node (lex) at (-0.7,-0.2) [] {$\mathtt{DegLex}$};
	\draw (lex) -- (-.433, -.35);
	
	\node (rlex) at (-0.22,-0.025) [] {$\mathtt{RevLex}$};
	\draw (rlex) -- (-.10825, -.3125);
\end{tikzpicture}
\caption{The Gr\"obner fan of the Hilbert scheme $\Hilb{3t+1}{3}$.}
\label{fig:maxConesTO}
\end{center}
\end{figure}

\section{Applications}\label{sec:applications}

In this last section, we use the machinery of Gr\"obner fans to study geometric properties of the Hilbert scheme such as connectedness and irreducibility. We start recalling the definition of a partial order among sets of a fixed number of monomials of a given degree induced by a term order. 

\begin{definition}[{\cite[Definition 6, Proposition 5]{BCR-GG}}]\label{def:ssucc}
Let $\Omega$ be a term order in $\kk[\xx]$ and let $\mathcal{M}_q^r$ be the collection of sets of $q$ monomials of degree $r$. We denote by $\ssucceq_{\Omega}$ the partial order on $\mathcal{M}_{q}^r$ defined as follows: given two sets $\mathfrak{A} = \{ \xx^{\aaa_1},\ldots,\xx^{\aaa_{q}}\},\ \xx^{ \aaa_\ell} >_{\Omega} \xx^{ \aaa_{\ell+1}},\ \ell=1,\ldots,q-1$ and $\mathfrak{B} = \{ \xx^{\bbb_1},\ldots,\xx^{\bbb_{q}}\},\ \xx^{\bbb_\ell} >_{\Omega} \xx^{\bbb_{\ell+1}},\ \ell=1,\ldots,q-1$ in $\mathcal{M}_q^r$ 
\[
\mathfrak{A} \ssucceq_\Omega \mathfrak{B}\qquad\Longleftrightarrow\qquad  \xx^{\aaa_\ell} \geq_\Omega \xx^{\bbb_\ell},\ \forall\ \ell=1,\ldots,q.
\]
We write $\mathfrak{A} \ssucc_\Omega \mathfrak{B}$ if at least one of the inequalities $\xx^{\aaa_\ell} \geq_\Omega \xx^{\bbb_\ell}$ is strict.
\end{definition}

For every ideal $J \in \SI{p(t)}{n}$, the monomial basis $\mathfrak{J}$ of $J_r$ is contained in the set $\mathcal{M}_{q}^r$ with $q = q(r) = \binom{n+r}{n}-p(r)$. Therefore, the order $\ssucceq_\Omega$ induces a partial order on $\SI{p(t)}{n}$:
\begin{equation}\label{eq:ssucceq}
J \ssucceq_\Omega J'\qquad\Longleftrightarrow\qquad \mathfrak{J} \ssucceq_\Omega \mathfrak{J}'.
\end{equation}
Both orders $\succeq_\Omega$ and $\ssucceq_{\Omega}$ are determined by the term order $\Omega$ and they are far from being unrelated. We now explain the relation and we exploit it to deduce properties of the Hilbert scheme. 

\begin{lemma}\label{lem:addMonomials}
Let $\mathfrak{A},\mathfrak{B}$ be two subsets in $\mathcal{M}_q^r$ such that $\mathfrak{A} \ssucceq_\Omega \mathfrak{B}$. For any pair of monomials $\xx^\aaa,\xx^\bbb$ of degree $r$  such that $\xx^\aaa \notin \mathfrak{A}$, $\xx^\bbb \notin \mathfrak{B}$ and $\xx^\aaa \geq_{\Omega} \xx^\bbb$, 
\[
\mathfrak{A} \cup \{\xx^\aaa\} \ssucceq_{\Omega} \mathfrak{B} \cup \{\xx^\bbb\}.
\]
Moreover, the strict inequality $\xx^\aaa >_{\Omega} \xx^\bbb$ guarantees the strict inequality $\mathfrak{A} \cup \{\xx^\aaa\} \ssucc_{\Omega} \mathfrak{B} \cup \{\xx^\bbb\}$.
\end{lemma}
\begin{proof} Let us consider the indices $i$ and $j$ defined by
\[
i = \min \{ \ell\ \vert\ \xx^\aaa >_\Omega \xx^{\aaa_\ell}\} \qquad\text{and}\qquad j = \min \{\ell\ \vert\  \xx^\bbb >_\Omega \xx^{\bbb_\ell}\}.
\]
We can write
\[
\mathfrak{A} \cup \{\xx^\aaa\} = \{\xx^{\aaa'_1},\ldots,\xx^{\aaa'_{q+1}}\},\quad\text{where~} \xx^{\aaa'_\ell} = 
{\small
\begin{cases}
\xx^{\aaa_\ell},& \ell = 1,\ldots,i-1,\\ \xx^\aaa,& \ell = i,\\ \xx^{\aaa_{\ell-1}},& \ell = i+1,\ldots,q+1.
\end{cases}
}
\]
and
\[
\mathfrak{B} \cup \{\xx^\bbb\} = \{\xx^{\bbb'_1},\ldots,\xx^{\bbb'_{q+1}}\},\quad\text{where~} \xx^{\bbb'_\ell} = 
{\small
\begin{cases}
\xx^{\bbb\ell},& \ell = 1,\ldots,j-1,\\ \xx^\bbb,& \ell = j,\\ \xx^{\bbb_{\ell-1}},& \ell = j+1,\ldots,q+1.
\end{cases}
}
\]

If $i < j$, the statement does not really depend on the assumption $\xx^\aaa \geq_{\Omega} \xx^\beta$. In fact, one has
\[
{\small
\begin{cases}
\xx^{\aaa'_\ell} = \xx^{\aaa_\ell} \geq_\Omega \xx^{\bbb_\ell} = \xx^{\bbb'_\ell},& \ell =1 ,\ldots, i-1\\
\xx^{\aaa'_i} = \xx^{\aaa} >_\Omega \xx^{\aaa_i} \geq_\Omega \xx^{\bbb_i} = \xx^{\bbb'_i},& \ell=i, \\
\xx^{\aaa'_\ell} = \xx^{\aaa_{\ell-1}} >_\Omega \xx_{\aaa_\ell}  \geq_\Omega \xx^{\bbb_\ell} =  \xx^{\bbb'_\ell},& \ell = i+1,\ldots,j-1,\\
\xx^{\aaa'_j} = \xx^{\aaa_{j-1}} \geq_\Omega \xx^{\bbb_{j-1}} >_\Omega \xx^\bbb = \xx^{\bbb'_{j}},& \ell=j,\\
\xx^{\aaa'_\ell} = \xx^{\aaa_{\ell-1}} \geq_\Omega \xx^{\bbb_{\ell-1}} = \xx^{\bbb'_\ell},& \ell =j+1 ,\ldots, q+1,
\end{cases}
}
\quad \Rightarrow\quad \mathfrak{A} \cup \{\xx^\aaa\} \ssucc_\Omega \mathfrak{B} \cup \{\xx^\bbb\}.
\]
 \begin{center}
\begin{tikzpicture}[node distance=2mm and 1mm,>=latex]
\node (a1) at (0,0) [inner sep=2pt,minimum width=0.85cm,] {\small $\xx^{\aaa_{1}}$};
\node (a2) at (1.5,0) [inner sep=2pt,minimum width=1.3cm,] {\small $\cdots$};
\draw [->] (a1) -- (a2); 
\node (a3) at (3,0) [inner sep=2pt,minimum width=0.85cm,] {\small $\xx^{\aaa_{i-1}}$};
\draw [->] (a2) -- (a3); 
\node (a4) at (4.25,0) [inner sep=2pt,minimum width=0.85cm,] {\small $\xx^{\aaa}$};
\draw [->] (a3) -- (a4); 
\node (a5) at (5.5,0) [inner sep=2pt,minimum width=0.85cm,] {\small $\xx^{\aaa_{i}}$};
\draw [->] (a4) -- (a5); 
\node (a6) at (7,0) [inner sep=2pt,minimum width=1.3cm,] {\small $\cdots$};
\draw [->] (a5) -- (a6);
 \node (a7) at (8.5,0) [inner sep=2pt,minimum width=0.85cm,] {\small $\xx^{\aaa_{j-2}}$};
\draw [->] (a6) -- (a7); 
\node (a8) at (9.75,0) [inner sep=2pt,minimum width=0.85cm,] {\small $\xx^{\aaa_{j-1}}$};
\draw [->] (a7) -- (a8); 
\node (a9) at (11,0) [inner sep=2pt,minimum width=0.85cm,] {\small $\xx^{\aaa_j}$};
\draw [->] (a8) -- (a9); 
\node (a10) at (12.5,0) [inner sep=2pt,minimum width=1.3cm,] {\small $\cdots$};
\draw [->] (a9) -- (a10);
\node (a11) at (14,0) [inner sep=2pt,minimum width=0.85cm,] {\small $\xx^{\aaa_q}$};
\draw [->] (a10) -- (a11);

\begin{scope}[shift={(0,-1)}]
\node (b1) at (0,0) [inner sep=2pt,minimum width=0.85cm,] {\small $\xx^{\bbb_{1}}$};
\node (b2) at (1.5,0) [inner sep=2pt,minimum width=1.3cm,] {\small $\cdots$};
\draw [->] (b1) -- (b2); 
\node (b3) at (3,0) [inner sep=2pt,minimum width=0.85cm,] {\small $\xx^{\bbb_{i-1}}$};
\draw [->] (b2) -- (b3); 
\node (b4) at (4.25,0) [inner sep=2pt,minimum width=0.85cm,] {\small $\xx^{\bbb_i}$};
\draw [->] (b3) -- (b4); 
\node (b5) at (5.5,0) [inner sep=2pt,minimum width=0.85cm,] {\small $\xx^{\bbb_{i+1}}$};
\draw [->] (b4) -- (b5); 
\node (b6) at (7,0) [inner sep=2pt,minimum width=1.3cm,] {\small $\cdots$};
\draw [->] (b5) -- (b6);
 \node (b7) at (8.5,0) [inner sep=2pt,minimum width=0.85cm,] {\small $\xx^{\bbb_{j-1}}$};
\draw [->] (b6) -- (b7); 
\node (b8) at (9.75,0) [inner sep=2pt,minimum width=0.85cm,] {\small $\xx^{\bbb}$};
\draw [->] (b7) -- (b8); 
\node (b9) at (11,0) [inner sep=2pt,minimum width=0.85cm,] {\small $\xx^{\bbb_j}$};
\draw [->] (b8) -- (b9); 
\node (b10) at (12.5,0) [inner sep=2pt,minimum width=1.3cm,] {\small $\cdots$};
\draw [->] (b9) -- (b10);
\node (b11) at (14,0) [inner sep=2pt,minimum width=0.85cm,] {\small $\xx^{\bbb_q}$};
\draw [->] (b10) -- (b11);
\end{scope}

\draw [latex-,dashed] (b1) -- (a1);
\draw [latex-,dashed] (b3) -- (a3);
\draw [latex-,dashed] (b4) -- (a5);
\draw [latex-,dashed] (b7) -- (a8);
\draw [latex-,dashed] (b9) -- (a9);
\draw [latex-,dashed] (b11) -- (a11);

\draw [,latex-,thick] (b4) -- (a4);
\draw [,latex-,thick] (b5) -- (a5);
\draw [,latex-,thick] (b7) -- (a7);
\draw [,latex-,thick] (b8) -- (a8);

\draw[] (9.375,-1.25)  --  (9.375,-0.8) -- (10.125,-0.2) -- (10.125,-1.25);
\draw[] (3.875,0.25)  --  (3.875,-0.8) -- (4.625,-0.2) -- (4.625,0.25);

\end{tikzpicture}
\end{center}

If $i = j$, then the statement is straightforward and $\xx^\aaa >_{\Omega} \xx^\bbb$ guarantees  $\mathfrak{A} \cup \{\xx^\aaa\}  \ssucc_\Omega \mathfrak{B} \cup \{\xx^\bbb\}$.

Finally, if $i > j$
\[
{\small
\begin{cases}
\xx^{\aaa'_\ell} = \xx^{\aaa_\ell} \geq_\Omega \xx^{\bbb_\ell} = \xx^{\bbb'_\ell},& \ell =1 ,\ldots, j-1\\
\xx^{\aaa'_j} = \xx^{\aaa_j} >_\Omega \xx^{\aaa} \geq_\Omega \xx^{\bbb} = \xx^{\bbb'_j},& \ell=j, \\
\xx^{\aaa'_\ell} = \xx^{\aaa_{\ell}} >_\Omega \xx^{\aaa}  \geq_\Omega \xx^{\bbb} > \xx^{\bbb_{\ell-1}} =  \xx^{\bbb'_\ell},& \ell = j+1,\ldots,i-1,\\
\xx^{\aaa'_i} = \xx^{\aaa} \geq_\Omega \xx^{\bbb} >_\Omega \xx^{\bbb_{i-1}} = \xx^{\bbb'_{i}},& \ell=i,\\
\xx^{\aaa'_\ell} = \xx^{\aaa_{\ell-1}} \geq_\Omega \xx^{\bbb_{\ell-1}} = \xx^{\bbb'_\ell},& \ell =i+1 ,\ldots, q+1,
\end{cases}
}
\quad \Rightarrow\quad \mathfrak{A} \cup \{\xx^\aaa\} \ssucc_\Omega \mathfrak{B} \cup \{\xx^\bbb\}.
\]
\begin{center}
\hfill
\begin{tikzpicture}[node distance=2mm and 1mm,>=latex]
\node (a1) at (0,0) [inner sep=2pt,minimum width=0.85cm,] {\small $\xx^{\aaa_{1}}$};
\node (a2) at (1.5,0) [inner sep=2pt,minimum width=1.3cm,] {\small $\cdots$};
\draw [->] (a1) -- (a2); 
\node (a3) at (3,0) [inner sep=2pt,minimum width=0.85cm,] {\small $\xx^{\aaa_{j-1}}$};
\draw [->] (a2) -- (a3); 
\node (a4) at (4.25,0) [inner sep=2pt,minimum width=0.85cm,] {\small $\xx^{\aaa_{j}}$};
\draw [->] (a3) -- (a4); 
\node (a5) at (5.5,0) [inner sep=2pt,minimum width=0.85cm,] {\small $\xx^{\aaa_{j+1}}$};
\draw [->] (a4) -- (a5); 
\node (a6) at (7,0) [inner sep=2pt,minimum width=1.3cm,] {\small $\cdots$};
\draw [->] (a5) -- (a6);
 \node (a7) at (8.5,0) [inner sep=2pt,minimum width=0.85cm,] {\small $\xx^{\aaa_{i-1}}$};
\draw [->] (a6) -- (a7); 
\node (a8) at (9.75,0) [inner sep=2pt,minimum width=0.85cm,] {\small $\xx^{\aaa}$};
\draw [->] (a7) -- (a8); 
\node (a9) at (11,0) [inner sep=2pt,minimum width=0.85cm,] {\small $\xx^{\aaa_i}$};
\draw [->] (a8) -- (a9); 
\node (a10) at (12.5,0) [inner sep=2pt,minimum width=1.3cm,] {\small $\cdots$};
\draw [->] (a9) -- (a10);
\node (a11) at (14,0) [inner sep=2pt,minimum width=0.85cm,] {\small $\xx^{\aaa_q}$};
\draw [->] (a10) -- (a11);

\begin{scope}[shift={(0,-1)}]
\node (b1) at (0,0) [inner sep=2pt,minimum width=0.85cm,] {\small $\xx^{\bbb_{1}}$};
\node (b2) at (1.5,0) [inner sep=2pt,minimum width=1.3cm,] {\small $\cdots$};
\draw [->] (b1) -- (b2); 
\node (b3) at (3,0) [inner sep=2pt,minimum width=0.85cm,] {\small $\xx^{\bbb_{j-1}}$};
\draw [->] (b2) -- (b3); 
\node (b4) at (4.25,0) [inner sep=2pt,minimum width=0.85cm,] {\small $\xx^{\bbb}$};
\draw [->] (b3) -- (b4); 
\node (b5) at (5.5,0) [inner sep=2pt,minimum width=0.85cm,] {\small $\xx^{\bbb_{j}}$};
\draw [->] (b4) -- (b5); 
\node (b6) at (7,0) [inner sep=2pt,minimum width=1.3cm,] {\small $\cdots$};
\draw [->] (b5) -- (b6);
 \node (b7) at (8.5,0) [inner sep=2pt,minimum width=0.85cm,] {\small $\xx^{\bbb_{i-2}}$};
\draw [->] (b6) -- (b7); 
\node (b8) at (9.75,0) [inner sep=2pt,minimum width=0.85cm,] {\small $\xx^{\bbb_{i-1}}$};
\draw [->] (b7) -- (b8); 
\node (b9) at (11,0) [inner sep=2pt,minimum width=0.85cm,] {\small $\xx^{\bbb_i}$};
\draw [->] (b8) -- (b9); 
\node (b10) at (12.5,0) [inner sep=2pt,minimum width=1.3cm,] {\small $\cdots$};
\draw [->] (b9) -- (b10);
\node (b11) at (14,0) [inner sep=2pt,minimum width=0.85cm,] {\small $\xx^{\bbb_q}$};
\draw [->] (b10) -- (b11);
\end{scope}

\draw [latex-,dashed] (b1) -- (a1);
\draw [latex-,dashed] (b3) -- (a3);
\draw [latex-,dashed] (b9) -- (a9);
\draw [latex-,dashed] (b11) -- (a11);

\draw [thick,latex-,thick] (b4) -- (a4);
\draw [thick,latex-,,thick] (b5) -- (a5);
\draw [thick,latex-,,thick] (b7) -- (a7);
\draw [thick,latex-,,thick] (b8) -- (a8);

\draw[]  (3.875,-1.25)  --  (3.875,-0.2) -- (4.625,-0.8) -- (4.625,-1.25);
\draw[]  (9.375,0.25)  --  (9.375,-0.2) -- (10.125,-0.8) -- (10.125,0.25) ;

\node (b) at (9.75,0) [inner sep=6pt] {};
\node (a) at (4.25,-1) [inner sep=6pt] {};
\draw [->,dashed] (b) to[out=-155,in=25] (a);

\end{tikzpicture} \hfill \qedhere
\end{center}
\end{proof}

\begin{proposition}\label{prop:BorelDef implies order}
The order $\ssucceq_\Omega$ is a refinement of the order $\succeq_{\Omega}$ on $\SI{p(t)}{n}$, i.e.
\begin{equation}
J \succeq_\Omega J' \qquad\Longrightarrow\qquad J \ssucceq_\Omega J'.
\end{equation}
\end{proposition}
\begin{proof}
If $J = J'$, then obviously $\mathfrak{J} = \mathfrak{J}'$. For the transitive property of $\succeq_\Omega$ and $\ssucceq_\Omega$, it suffices to prove the implication for pairs of Borel adjacent ideals ideals $J,J'$ such that $[J_{\aaa}{\xrightarrow{\text{\tiny$\Omega$}}}J'_{\aaa'}]$ is a directed edge of $\DG{p(t)}{n}{\Omega}$. By definition, we have
\[
\mathfrak{J} = (\mathfrak{J} \cap \mathfrak{J}') \cup (\mathfrak{J}\setminus \mathfrak{J}')\quad\text{and}\quad\mathfrak{J}' = (\mathfrak{J} \cap \mathfrak{J}') \cup (\mathfrak{J}'\setminus \mathfrak{J}).
\]
We obtain the thesis, starting from $\mathfrak{J} \cap \mathfrak{J}'$ and applying repeatedly Lemma \ref{lem:addMonomials} on pairs
\[
\textsf{E}(\xx^\aaa) >_{\Omega} \textsf{E}(\xx^{\aaa'}),\qquad\forall\ \mathsf{E} \in \mathcal{E}_{J,J'}. \qedhere
\]
\end{proof}

In the following, maximal elements of $\SI{p(t)}{n}$ with respect to $\ssucceq_\Omega$ and $\succeq_\Omega$ play a crucial role. We introduce the following notation:
\begin{align}
& \max_{\ssucceq_\Omega} \SI{p(t)}{n} = \left\{ J \in \SI{p(t)}{n}\ \middle\vert\ \nexists\ J'\neq J \in \SI{p(t)}{n} \text{~s.t.~} J' \ssucc_{\Omega} J\right\},\\
& \max_{\succeq_\Omega} \SI{p(t)}{n} = \left\{ J \in \SI{p(t)}{n}\ \middle\vert\ \nexists\ J'\neq J \in \SI{p(t)}{n} \text{~s.t.~} J' \succ_{\Omega} J\right\}.
\end{align}
By Proposition \ref{prop:BorelDef implies order}, we have the inclusion $ \max_{\ssucceq_\Omega} \SI{p(t)}{n} \subseteq  \max_{\succeq_\Omega} \SI{p(t)}{n}$. We underline that computing the set $ \max_{\ssucceq_\Omega}  \SI{p(t)}{n}$ from Definition \ref{def:ssucc} is quite involved. Whereas, computing the set $\max_{\succeq_\Omega}  \SI{p(t)}{n}$ is much easier. Indeed, a maximal element with respect to $\succeq_{\Omega}$ corresponds to a vertex in $\DG{p(t)}{n}{\Omega}$ with no incoming edges (in graph theory, one says that the in-degree of the vertex is 0).

\subsection{Connectedness of the Hilbert scheme}

We recall that a strongly stable ideal $J \in \SI{p(t)}{n}$ is called $\Omega$-hilb-segment ideal, for some term order $\Omega$, if $\xx^\aaa >_\Omega \xx^\bbb$ for every $\xx^\aaa \in \mathfrak{J}$ and every $\xx^{\bbb} \in \comp{\mathfrak{J}}$. Moreover, notice that at least one hilb-segment ideal exists  for every Hilbert scheme $\Hilb{p(t)}{n}$. Indeed, the unique lexicographic ideal in $\SI{p(t)}{n}$ is the $\mathtt{DegLex}$-hilb-segment ideal.

\begin{theorem}\label{thm:segment case}
Let $\Omega$ be a term order such that there exists the $\Omega$-hilb-segment ideal $L \in \SI{p(t)}{n}$. Then,
\[
\max_{\ssucceq_{\Omega}} \SI{p(t)}{n} = \max_{\succeq_{\Omega}} \SI{p(t)}{n} = \{L\}.
\]
\end{theorem}
\begin{proof}
Let us start proving that $L$ is the unique maximal element in $ \SI{p(t)}{n}$ with respect to $\ssucceq_{\Omega}$. For any $J \neq L \in \SI{p(t)}{n}$, we have
\[
\mathfrak{L} = (\mathfrak{L} \cap \mathfrak{J}) \cup (\mathfrak{L} \setminus \mathfrak{J})\quad \text{and}\quad \mathfrak{J} = (\mathfrak{L} \cap \mathfrak{J}) \cup (\mathfrak{J} \setminus \mathfrak{L}).
\]
By definition of hilb-segment ideal, every monomial in $ (\mathfrak{L} \setminus \mathfrak{J})\subseteq \mathfrak{L}$ is greater than every monomial in $ (\mathfrak{J} \setminus \mathfrak{L}) \subseteq \comp{\mathfrak{L}}$. Hence, by Lemma \ref{lem:addMonomials} $\mathfrak{L} \ssucc_\Omega \mathfrak{J} \Leftrightarrow L \ssucc_\Omega J$. 

\smallskip

Since $\max_{\ssucceq_{\Omega}} \SI{p(t)}{n} \subseteq \max_{\succeq_{\Omega}} \SI{p(t)}{n}$, in order to prove that $L$ is the unique maximal ideal also respect to $\succeq_\Omega$, we show that for any $J \neq L \in \SI{p(t)}{n}$, there exists a Borel adjacent ideal $I \in \SI{p(t)}{n}$ such that $I \succ_{\Omega} J$. 
First, we describe the procedure to find $I$ and subsequently we prove correctness and termination. We use the idea discussed in Remark \ref{rk:propertiesAdjacent}\textit{(\ref{rk:propertiesAdjacent_v})}. 

\underline{\textit{Step 0.}} Denote by $\mathfrak{A} = \mathfrak{L} \setminus \mathfrak{J}$ and $\mathfrak{B} = \mathfrak{J}\setminus\mathfrak{L}$. We have that every monomial in $\mathfrak{A}$ is greater than every monomial in $\mathfrak{B}$ with respect to $\geq_\Omega$.

\smallskip

\underline{\textit{Step 1.}}
Let $\xx^\aaa = \max_{\geq_\Omega} \mathfrak{A}$ and let $\xx^\bbb = \min_{\geq_\Omega} \mathfrak{B}_k$, where $k = \min \xx^\aaa$. 

\smallskip

\underline{\textit{Step 2.}} Consider the set $\mathfrak{E} = \{ \xx^\ccc \in \mathfrak{B}\ \vert\ \xx^\bbb \geq_B \xx^\ccc\}$ and the associated set of compositions of elementary decreasing moves $\mathcal{E}$ such that $\mathfrak{E} = \{\textsf{E}(\xx^\bbb)\ \vert\ \textsf{E} \in \mathcal{E}\}$.
If $({\dagger})$ every move $\textsf{E}\in\mathcal{E}$ is also admissible for $\xx^\aaa$, i.e.~$\textsf{E}(\xx^\aaa)$ is a monomial, and $({\ddagger})$ for an admissible move $\eu{h}$, the monomial $\eu{h}\big(\textsf{E}(\xx^\aaa)\big)$ is either contained in $\mathfrak{J}$ or is of the type $\widetilde{\textsf{E}}(\xx^\aaa)$ for some $\widetilde{\textsf{E}} \in \mathcal{E}$, then the $\Omega$-degeneration graph has the edge $[I_{\aaa}{\xrightarrow{\text{\tiny$\Omega$}}} J_{\bbb}]$, where $I$ is the ideal generated by 
\[
\mathfrak{I} =\mathfrak{J} \setminus  \{\textsf{E}(\xx^\bbb)\ \vert\ \textsf{E} \in \mathcal{E}\} \cup \{\textsf{E}(\xx^\aaa)\ \vert\ \textsf{E} \in \mathcal{E}\}.
\]

\smallskip

\underline{\textit{Step 3.}} If condition $({\dagger})$ or condition $({\ddagger})$ in \textit{Step 2} is not satisfied, we start again from \textit{Step 1} with
\[
\mathfrak{A}' = \mathfrak{A} \setminus \{ \xx^\ccc \in \mathfrak{A}\ \vert\ \xx^\aaa \geq_B \xx^\ccc  \},\qquad \mathfrak{B}' = \mathfrak{B}.
\]

\textit{Correctness and termination.}  $\bullet$ The monomial $\xx^{\aaa}$ is a maximal element of $\comp{\mathfrak{J}}$ with respect to $\geq_B$. Indeed, $\xx^\aaa \in \mathfrak{A}\subset \mathfrak{L}\setminus\mathfrak{J}  \subset \comp{\mathfrak{J}}$ and for any admissible move $\eu{h}$, the monomial $\eu{h}(\xx^\aaa)$ is contained in $\mathfrak{L}\cap\mathfrak{J}$, as $\mathfrak{L}$ is closed under the action of increasing moves and $\eu{h}(\xx^\aaa) >_B \xx^\aaa$ implies $\eu{h}(\xx^\aaa) >_{\Omega} \xx^\aaa$ and $\eu{h}(\xx^\aaa)$ can not be one of the monomials removed from $\mathfrak{L}\setminus \mathfrak{J}$ in \textit{Step 3}.

\smallskip

$\bullet$ At the beginning, we have $\vert \mathfrak{A}_i \vert = \vert \mathfrak{B}_i \vert$ for all $i=0,\ldots,n$, subsequently $\vert \mathfrak{A}_i \vert \leqslant \vert \mathfrak{B}_i \vert$ for all $i=0,\ldots,n$. Hence, $\xx^\aaa \in \mathfrak{A}_k$ implies $\vert  \mathfrak{B}_k\vert \geqslant \vert  \mathfrak{A}_k \vert > 0$, so that the monomial $\xx^\bbb = \min_{\geq_\Omega} \mathfrak{B}_k$ exists. By definition, $\xx^\bbb$ is a minimal element in $\mathfrak{J} \cap \kk[x_k,\ldots,x_n]$ with respect to $\geq_B$. Since $\xx^\aaa >_\Omega \xx^\bbb$, $\xx^\aaa \in \comp{\mathfrak{J}}$ and $\xx^\bbb\in\mathfrak{J}$ are not comparable with respect to $\geq_B$.

\smallskip

$\bullet$ Let $\mathfrak{F}$ be the set of monomials $\{\textsf{E}(\xx^\aaa)\ \vert\ \textsf{E} \in \mathcal{E}\text{~admissible for~}\xx^\aaa\}$ and assume that $\vert \mathfrak{F} \vert < \vert \mathfrak{E} \vert$ (condition $({\dagger})$ in \textit{Step 2} in not satisfied).
Then, there exists a monomial $\textsf{E}(\xx^\bbb) \in \mathfrak{B}$ that is not paired with a monomial $\xx^{\aaa'} \in \mathfrak{L}\setminus\mathfrak{J}$ by some set $\mathcal{E}$. Hence, $\mathfrak{A}'$ is not empty, as $\mathfrak{F} \cap \mathfrak{A} \subset  \{ \xx^\ccc \in \mathfrak{A}\ \vert\ \xx^\aaa \geq_B \xx^\ccc  \}$.

$\bullet$ Assume that condition $({\ddagger})$ in \textit{Step 2} is not satisfied. Namely, $\vert \mathfrak{F} \vert = \vert \mathfrak{E} \vert$ but there exists a monomial $\textsf{E}(\xx^\aaa) \in \mathfrak{F}$ and an elementary move $\eu{h}$ such that $\eu{h}(\textsf{E}(\xx^\aaa))$ is not contained in $\mathfrak{J}$. If $\mathfrak{F} \subset \mathfrak{A}$, then $\eu{h}(\textsf{E}(\xx^\aaa))$ is contained in $\mathfrak{A}$ and it is not comparable with $\xx^\aaa$ with respect to $\geq_B$ and $\mathfrak{A}'$ is not empty. If there exists $\textsf{E} \in \mathcal{E}$ such that $\textsf{E}(\xx^\aaa) \notin \mathfrak{A}$, then the monomial $\textsf{E}(\xx^\bbb) \in \mathfrak{B}$ is not paired with a monomial $\xx^{\aaa'} \in \mathfrak{L}\setminus\mathfrak{J}$ by some set $\mathcal{E}$ and we apply the same argument as before. 

\smallskip

$\bullet$ Notice that the minimality of $\xx^\bbb$ among monomials in $\mathfrak{B}_k \setminus \mathfrak{A}_k$ implies that $\min \textsf{E}(\xx^\bbb) < k$ for all $\textsf{E} \neq \textsf{id} \in \mathcal{E}$ and that $\min \textsf{E}(\xx^\aaa) =  \min \textsf{E}(\xx^\bbb)$ for all $\textsf{E} \in \mathcal{E}$.

\smallskip 

$\bullet$ Conditions $({\dagger})$ and $({\ddagger})$ required in \textit{Step 2} guarantee that the set $\{\textsf{E}(\xx^\aaa)\ \vert\ \textsf{E} \in \mathcal{E}\}$ is an \lq\lq outer border\rq\rq~of $\mathfrak{J}$ and the set $\{\textsf{E}(\xx^\bbb)\ \vert\ \textsf{E} \in \mathcal{E}\}$ is an \lq\lq inner border\rq\rq~of $\mathfrak{J}$. Furthermore, as $\xx^\aaa \in \comp{\mathfrak{J}}$ and $\xx^\bbb\in\mathfrak{J}$ are not comparable with respect to $\geq_B$, the set $\mathfrak{I} =\mathfrak{J} \setminus  \{\textsf{E}(\xx^\bbb)\ \vert\ \textsf{E} \in \mathcal{E}\} \cup \{\textsf{E}(\xx^\aaa)\ \vert\ \textsf{E} \in \mathcal{E}\}$ is closed under the action of increasing elementary moves and $\vert \mathfrak{I}_i\vert = \vert \mathfrak{J}_i\vert,\ i=0,\ldots,n$. Then, $\mathfrak{I}$ corresponds to an ideal $I \in \SI{p(t)}{n}$ that is Borel adjacent to $J$ and $\xx^\aaa >_{\Omega} \xx^\bbb$ implies that $I \succ_\Omega J$.

\smallskip

$\bullet$ Each time that we encounter a failure in \textit{Step 3}, the non-emptiness of $\mathfrak{A}'$ is guaranteed by monomials with minimum variable strictly lower that $\min \xx^\aaa$. Since $J\neq L$, then $\mathfrak{J}_0 \neq \mathfrak{L}_0$ and applying repeatedly the procedure we eventually obtain $\xx^\aaa \in \mathfrak{A}_0$. In this case, $\xx^\bbb = \min_{\geq_\Omega} \mathfrak{B}_0$, $\mathcal{E} = \{\textsf{id}\}$, conditions $({\dagger})$ and $({\ddagger})$ are satisfied and we finally find  $I$ such that $I \succ_{\Omega} J$.
\end{proof}

\begin{example}\label{ex:edgeSpanningTree} Consider the polynomial ring $\kk[x_0,x_1,x_2,x_3]$ and the Hilbert polynomial $p(t) = 5t-2$. The ideal $L^\sat = (x_3^2,x_2x_3,x_2^4)$ is the $\Omega$-hilb-segment ideal with respect to the term order $\Omega$ described by the matrix
\[
M_{\Omega} = {\small \left[\begin{array}{cccc}
1 & 1 & 1 & 1 \\ 1 & 3 & 17 & 47 \\ 0 & 0 & 1 & 0 \\ 0 & 1 & 0 & 0
\end{array}\right]}.
\]
Let us determine an edge $[I {\xrightarrow{\text{\tiny$\Omega$}}} J]$ of $\DG{5t-2}{3}{\Omega}$ for $J^\sat = (x_3^2,x_2^2 x_3,x_1 x_2 x_3,x_1^2 x_3,x_2^5)$ following the procedure presented in the proof of Theorem \ref{thm:segment case}. The Gotzmann number of $p(t) = 5t-2$ is 8, so we start considering
\[
\begin{split}
\mathfrak{A} &{} = \mathfrak{L}\setminus\mathfrak{J} = \{x_1^4 x_2^4, x_0x_1^3 x_2^4, x_0^2 x_1^2 x_2^4, x_0^3x_1 x_2^4, x_0^4 x_2^4,x_0^6 x_2x_3\},\\
\mathfrak{B} &{} = \mathfrak{J}\setminus\mathfrak{L} = \{x_1^7 x_3, x_0 x_1^6 x_3, x_0^2 x_1^5 x_3, x_0^3 x_1^4 x_3, x_0^4 x_1^3 x_3, x_0^5 x_1^2 x_3\}.\\
\end{split}
\] 
We have $\max_{\geq_\Omega} \mathfrak{A} = x_1^4 x_2^4$, $\min x_1^4 x_2^4 = 1$, $\min_{\geq_\Omega} \mathfrak{B}_1 = x_1^7 x_3$ and
\[
\mathcal{E} = \left\{\textsf{id}, \ed{1}, (\ed{1})^2, (\ed{1})^3, (\ed{1})^4, (\ed{1})^5\right\}
\]
Condition $(\dagger)$ of \textit{Step 2} is not satisfied because the last move in $\mathcal{E}$ is not admissible for $x_1^4 x_2^4$. Hence, we consider
\[
\mathfrak{A}' = \mathfrak{A} \setminus \{x_0^{c_0} x_1^{c_1}x_2^{c_2} x_3^{c_3} \in \mathfrak{A} \ \vert\ x_1^4 x_2^4 \geq_{B}x_0^{c_0} x_1^{c_1}x_2^{c_2} x_3^{c_3} \} = \{x_0^6 x_2x_3\} \qquad\text{and}\qquad\mathfrak{B}' =\mathfrak{B}.
\]
The next pair of monomials to examine is $\max_{\geq_{\Omega}} \mathfrak{A}' = x_0^6 x_2x_3$ and $\min_{\geq_\Omega} \mathfrak{B}'_0 = x_0^5 x_1^2 x_3$. As $\min x_0^6 x_2x_3 = \min x_0^5 x_1^2 x_3 = 0$, conditions $(\dagger)$ and $(\ddagger)$ are surely satisfied and the saturation of the ideal generated by $\mathfrak{I} = \mathfrak{J} \setminus\{x_0^5 x_1^2 x_3 \}\cup\{x_0^6 x_2x_3\}$ is $I^\sat = (x_3^2,x_2 x_3,x_1^3 x_3, x_2^5)$.
\end{example}

\begin{corollary}\label{cor:spanning tree}
The Borel graph $\mathscr{G}_{p(t)}^n$ of $\Hilb{p(t)}{n}$ is connected.
\end{corollary}
\begin{proof}
Choose a term order $\Omega$ such that $\SI{p(t)}{n}$ contains the $\Omega$-hilb-segment ideal (we recall that each $\SI{p(t)}{n}$ contains at least the $\mathtt{DegLex}$-hilb-segment ideal). Then, we consider the subgraph $\mathscr{T}_{p(t)}^n(\Omega)$ of the $\Omega$-degeneration graph $\DG{p(t)}{n}{\Omega}$ with the same set of vertices and whose edges $E_{\textnormal{d}}\big(\mathscr{T}_{p(t)}^n(\Omega)\big) \subseteq E_{\textnormal{d}}\big(\DG{p(t)}{n}{\Omega}\big)$ correspond to pairs of Borel adjacent ideals determined with the procedure introduced in the proof of Theorem \ref{thm:segment case}. The graph $\mathscr{T}_{p(t)}^n(\Omega)$ turns out to be a minimum spanning tree of $\DG{p(t)}{n}{\Omega}$, because it  is a directed graph and each vertex has exactly one incoming edge, except the one corresponding to the $\Omega$-hilb-segment ideal that is the root of the tree.  
The connectedness of $\mathscr{G}_{p(t)}^n$ follows from the connectedness of $\mathscr{T}_{p(t)}^n(\Omega)$.
\end{proof}

\begin{example}
The set $\SI{5t-2}{3}$ introduced in Example \ref{ex:idealsOrder} contains 7 ideals. There is no $\mathtt{RevLex}$-hilb-segment ideal, as the $\mathtt{RevLex}$-degeneration graph has two vertices with no incoming edges, namely $\max_{\succeq_{\mathtt{RevLex}}} \SI{5t-2}{3}= \{J_6,J_7\}$ (see Figure \ref{fig:idealsOrder}).
The ideals $J_1$, $J_3$, $J_4$, $J_5$ and $J_7$ are hilb-segment ideals with respect to term orders $\Omega_i$ described by the matrices
\[
M_{\Omega_i} = {\small \left[\begin{array}{cccc}
1 & 1 & 1 & 1 \\ \omega_{i,0} & \omega_{i,1} & \omega_{i,2} & \omega_{i,3} \\ 0 & 0 & 1 & 0 \\ 0 & 1 & 0 & 0
\end{array}\right]},\qquad i=1,3,4,5,7
\]
with $\ooo_1 = (1,2,4,19)$, $\ooo_3 = (1,4,9,44)$, $\ooo_4 = (1, 4, 12,53)$, $\ooo_5 = (1,3,11,45)$ and $\ooo_7 = (1,3,17,47)$. 
In Figure \ref{fig:spanningTree}, there is the Borel graph $\Sk{5t-2}{3}$ and the spanning trees computed with the procedure given in the proof of Theorem \ref{thm:segment case} varying the hilb-segment ideal.
\end{example}

\captionsetup[subfloat]{position=bottom,width=0pt}
\begin{figure}[!ht]
\begin{center}
\subfloat[][$\Sk{5t-2}{3}$.]{\label{fig:spanningTreeBorelGraph}
\begin{tikzpicture}[scale=0.5,>=angle 60]
\node (0) at (-0.5,1) [draw,circle,inner sep=1.5pt] {\scriptsize $J_1$};
\node (1) at (1.5,-2) [draw,circle,inner sep=1.5pt] {\scriptsize $J_2$};
\node (2) at (2,1) [draw,circle,inner sep=1.5pt] {\scriptsize $J_3$};
\node (3) at (3.25,-0.5) [draw,circle,inner sep=1.5pt] {\scriptsize $J_4$};
\node (4) at (3.5,3) [draw,circle,inner sep=1.5pt] {\scriptsize $J_5$};
\node (5) at (5.5,-1) [draw,circle,inner sep=1.5pt] {\scriptsize $J_6$};
\node (6) at (6,2) [draw,circle,inner sep=1.5pt] {\scriptsize $J_7$};

  \draw [-,thick] (1) -- (0);
  \draw [-,thick] (2) -- (1);
  \draw [-,thick] (2) -- (0);
  \draw [-,thick] (3) -- (0);
  \draw [-,thick] (3) -- (2);
  \draw [-,thick] (4) -- (0);
  \draw [-,thick] (4) -- (2);
  \draw [-,thick] (4) -- (3);
  \draw [-,thick] (5) -- (1);
  \draw [-,thick] (5) -- (3);
  \draw [-,thick] (5) -- (4);
  \draw [-,thick] (6) -- (4);

\end{tikzpicture}
}
\hspace{1cm}
\subfloat[][$\mathscr{T}_{5t-2}^{3}(\Omega_1)$.]{\label{fig:spanningTreeBorelGraph1}
\begin{tikzpicture}[scale=0.5,>=angle 60]

\node (0) at (-0.5,1) [draw,rectangle,inner sep=2.5pt] {\scriptsize $J_1$};
\node (1) at (1.5,-2) [draw,circle,inner sep=1.5pt] {\scriptsize $J_2$};
\node (2) at (2,1) [draw,circle,inner sep=1.5pt] {\scriptsize $J_3$};
\node (3) at (3.25,-0.5) [draw,circle,inner sep=1.5pt] {\scriptsize $J_4$};
\node (4) at (3.5,3) [draw,circle,inner sep=1.5pt] {\scriptsize $J_5$};
\node (5) at (5.5,-1) [draw,circle,inner sep=1.5pt] {\scriptsize $J_6$};
\node (6) at (6,2) [draw,circle,inner sep=1.5pt] {\scriptsize $J_7$};

  \draw [<-,thick] (1) -- (0);
  \draw [-,dotted] (2) -- (1);
  \draw [<-,thick] (2) -- (0);
  \draw [-,dotted] (3) -- (0);
  \draw [<-,thick] (3) -- (2);
  \draw [-,dotted] (4) -- (0);
  \draw [-,dotted] (4) -- (2);
  \draw [<-,thick] (4) -- (3);
  \draw [-,dotted] (5) -- (1);
  \draw [<-,thick] (5) -- (3);
  \draw [-,dotted] (5) -- (4);
  \draw [<-,thick] (6) -- (4);
  
\end{tikzpicture}
}
\hspace{1cm}
\subfloat[][$\mathscr{T}_{5t-2}^{3}(\Omega_3)$.]{\label{fig:spanningTreeBorelGraph3}
\begin{tikzpicture}[scale=0.5,>=angle 60]

\node (0) at (-0.5,1) [draw,circle,inner sep=1.5pt] {\scriptsize $J_1$};
\node (1) at (1.5,-2) [draw,circle,inner sep=1.5pt] {\scriptsize $J_2$};
\node (2) at (2,1) [draw,rectangle,inner sep=2.5pt] {\scriptsize $J_3$};
\node (3) at (3.25,-0.5) [draw,circle,inner sep=1.5pt] {\scriptsize $J_4$};
\node (4) at (3.5,3) [draw,circle,inner sep=1.5pt] {\scriptsize $J_5$};
\node (5) at (5.5,-1) [draw,circle,inner sep=1.5pt] {\scriptsize $J_6$};
\node (6) at (6,2) [draw,circle,inner sep=1.5pt] {\scriptsize $J_7$};

  \draw [-,dotted] (1) -- (0);
  \draw [->,thick] (2) -- (1);
  \draw [->,thick] (2) -- (0);
  \draw [-,dotted] (3) -- (0);
  \draw [<-,thick] (3) -- (2);
  \draw [-,dotted] (4) -- (0);
  \draw [-,dotted] (4) -- (2);
  \draw [<-,thick] (4) -- (3);
  \draw [-,dotted] (5) -- (1);
  \draw [<-,thick] (5) -- (3);
  \draw [-,dotted] (5) -- (4);
  \draw [<-,thick] (6) -- (4);

\end{tikzpicture}
}

\subfloat[][$\mathscr{T}_{5t-2}^{3}(\Omega_4)$.]{\label{fig:spanningTreeBorelGraph4}
\begin{tikzpicture}[scale=0.5,>=angle 60]
\node (0) at (-0.5,1) [draw,circle,inner sep=1.5pt] {\scriptsize $J_1$};
\node (1) at (1.5,-2) [draw,circle,inner sep=1.5pt] {\scriptsize $J_2$};
\node (2) at (2,1) [draw,circle,inner sep=1.5pt] {\scriptsize $J_3$};
\node (3) at (3.25,-0.5) [draw,rectangle,inner sep=2.5pt] {\scriptsize $J_4$};
\node (4) at (3.5,3) [draw,circle,inner sep=1.5pt] {\scriptsize $J_5$};
\node (5) at (5.5,-1) [draw,circle,inner sep=1.5pt] {\scriptsize $J_6$};
\node (6) at (6,2) [draw,circle,inner sep=1.5pt] {\scriptsize $J_7$};

  \draw [-,dotted] (1) -- (0);
  \draw [->,thick] (2) -- (1);
  \draw [->,thick] (2) -- (0);
  \draw [-,dotted] (3) -- (0);
  \draw [->,thick] (3) -- (2);
  \draw [-,dotted] (4) -- (0);
  \draw [-,dotted] (4) -- (2);
  \draw [<-,thick] (4) -- (3);
  \draw [-,dotted] (5) -- (1);
  \draw [<-,thick] (5) -- (3);
  \draw [-,dotted] (5) -- (4);
  \draw [<-,thick] (6) -- (4);

\end{tikzpicture}
}
\hspace{1cm}
\subfloat[][$\mathscr{T}_{5t-2}^{3}(\Omega_5)$.]{\label{fig:spanningTreeBorelGraph5}
\begin{tikzpicture}[scale=0.5,>=angle 60]

\node (0) at (-0.5,1) [draw,circle,inner sep=1.5pt] {\scriptsize $J_1$};
\node (1) at (1.5,-2) [draw,circle,inner sep=1.5pt] {\scriptsize $J_2$};
\node (2) at (2,1) [draw,circle,inner sep=1.5pt] {\scriptsize $J_3$};
\node (3) at (3.25,-0.5) [draw,circle,inner sep=1.5pt] {\scriptsize $J_4$};
\node (4) at (3.5,3) [draw,rectangle,inner sep=2.5pt] {\scriptsize $J_5$};
\node (5) at (5.5,-1) [draw,circle,inner sep=1.5pt] {\scriptsize $J_6$};
\node (6) at (6,2) [draw,circle,inner sep=1.5pt] {\scriptsize $J_7$};

  \draw [-,dotted] (1) -- (0);
  \draw [->,thick] (2) -- (1);
  \draw [->,thick] (2) -- (0);
  \draw [-,dotted] (3) -- (0);
  \draw [->,thick] (3) -- (2);
  \draw [-,dotted] (4) -- (0);
  \draw [-,dotted] (4) -- (2);
  \draw [->,thick] (4) -- (3);
  \draw [-,dotted] (5) -- (1);
  \draw [-,dotted] (5) -- (3);
  \draw [<-,thick] (5) -- (4);
  \draw [<-,thick] (6) -- (4);
  
\end{tikzpicture}
}
\hspace{1cm}
\subfloat[][$\mathscr{T}_{5t-2}^{3}(\Omega_7)$.]{\label{fig:spanningTreeBorelGraph7}
\begin{tikzpicture}[scale=0.5,>=angle 60]

\node (0) at (-0.5,1) [draw,circle,inner sep=1.5pt] {\scriptsize $J_1$};
\node (1) at (1.5,-2) [draw,circle,inner sep=1.5pt] {\scriptsize $J_2$};
\node (2) at (2,1) [draw,circle,inner sep=1.5pt] {\scriptsize $J_3$};
\node (3) at (3.25,-0.5) [draw,circle,inner sep=1.5pt] {\scriptsize $J_4$};
\node (4) at (3.5,3) [draw,circle,inner sep=1.5pt] {\scriptsize $J_5$};
\node (5) at (5.5,-1) [draw,circle,inner sep=1.5pt] {\scriptsize $J_6$};
\node (6) at (6,2) [draw,rectangle,inner sep=2.5pt] {\scriptsize $J_7$};

  \draw [-,dotted] (1) -- (0);
  \draw [->,thick] (2) -- (1);
  \draw [->,thick] (2) -- (0);
  \draw [-,dotted] (3) -- (0);
  \draw [->,thick] (3) -- (2);
  \draw [-,dotted] (4) -- (0);
  \draw [-,dotted] (4) -- (2);
  \draw [->,thick] (4) -- (3);
  \draw [-,dotted] (5) -- (1);
  \draw [-,dotted] (5) -- (3);
  \draw [<-,thick] (5) -- (4);
  \draw [->,thick] (6) -- (4);

\end{tikzpicture}
}
\caption{The spanning trees of $\Sk{5t-2}{3}$ determined using Theorem \ref{thm:segment case} and varying the hilb-segment ideal.}
\label{fig:spanningTree}
\end{center}
\end{figure}
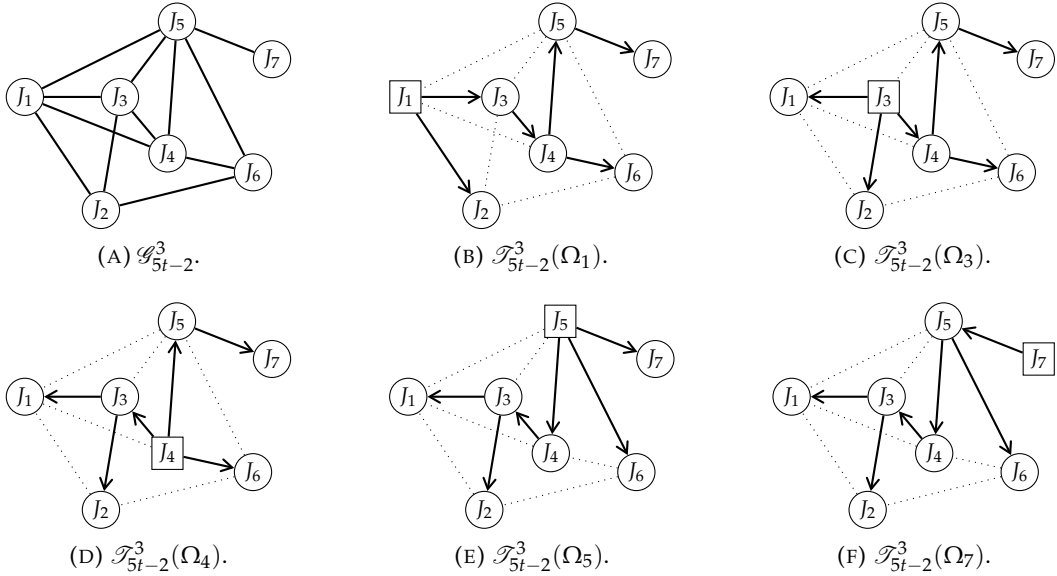

\begin{theorem}\label{thm:chainConnected}
The Hilbert scheme $\Hilb{p(t)}{n}$ is rationally chain connected.
\end{theorem}
\begin{proof}
We need to show that for any pair of closed points $[X],[Y] \in \Hilb{p(t)}{n}$ there exists a sequence of rational curves $C_0,\ldots,C_s$ such that $[X] \in C_0$, $[Y] \in C_s$ and $C_i \cap C_{i-1} \neq \emptyset,\ \forall\ i=1,\ldots,s$. It is equivalent to prove that there exists a sequence of rational curves $C_0,\ldots,C_s$ connecting any point of $\Hilb{p(t)}{n}$ with a fixed point. Hence, choose a point $[L] \in \Hilb{p(t)}{n}$ corresponding to the $\Omega$-hilb-segment ideal for some term order $\Omega$ (e.g.~the lexicogaphic ideal).

Given a point $[X] \in \Hilb{p(t)}{n}$, denote by $I_X \subset \kk[\xx]$ the saturated ideal defining $X$. If $I_X$ is not a strongly stable ideal, consider the generic initial ideal $J$ of $I_X$ with respect to an arbitrary term order. It is well-known that there exists a flat family of ideals parametrized by the affine line $\mathbb{A}^1 = \Spec \kk[T]$ such that the fiber over the point $T=1$ is $I_X$ and the fiber over the point $T=0$ is $J$. Let $\phi_X: \mathbb{A}^1 \to \Hilb{p(t)}{n}$ the associated morphism. As $\phi_X$ is non-constant, the closure $C_X = \overline{\phi_X(\mathbb{A}^1)}$ of the image of $\phi_X$ is a rational curve contained in $ \Hilb{p(t)}{n}$ \cite[Proposition 9.8]{AG}.

It remains to show that there is a sequence of rational curves connecting a point $[J] \in \Hilb{p(t)}{n}$ to $[L]$ for all $J \in \SI{p(t)}{n}$.
Let $\mathscr{T}_{p(t)}^{n}(\Omega)$ be the spanning tree of the Borel graph $\Sk{p(t)}{n}$ constructed in Corollary \ref{cor:spanning tree} and consider the list of edges 
$[(L{=}J_0) {\xrightarrow{\text{\tiny$\Omega$}}} J_1],[J_1 {\xrightarrow{\text{\tiny$\Omega$}}} J_2],\ldots,[J_{s-1} {\xrightarrow{\text{\tiny$\Omega$}}} (J_s{=}J)]$ that form the unique path in $\mathscr{T}_{p(t)}^{n}(\Omega)$ going from the root of the tree $L$ to the vertex $J$. For every edge $[J_i{\xrightarrow{\text{\tiny$\Omega$}}}J_{i+1}],\ i = 0,\ldots,s-1$, consider the flat family $X_{J_{i},J_{i+1}} \to \PP^1$ described in Theorem \ref{thm:mainDef} such that the fiber over $[1:0]$ is the scheme $\Proj \kk[\xx]/J_{i}$ and the fiber over $[0:1]$ is $\Proj \kk[\xx]/J_{i+1}$. The associated morphism $\varphi_{J_{i},J_{i+1}} :\PP^1 \to \Hilb{p(t)}{n}$ is non-constant and the image $C_i = \varphi_{J_{i},J_{i+1}}(\PP^1)$ is a rational curve contained in $\Hilb{p(t)}{n}$ (see also Remark \ref{rk:veronese}). The point $[J_i]$ is contained in the intersection $C_{i-1} \cap C_{i}$ for $i = 1,\ldots,s-1$, so that the sequence of curves $C_0,\ldots,C_s$ gives the chain connecting $[J]$ to $[L]$.
\end{proof}

\begin{remark}
The connectedness of $\Hilb{p(t)}{n}$ has been proved first by Hartshorne \cite{HartshorneThesis} and afterwards by Peeva and Stillman \cite{PeevaStillman}. 
Common ideas of all proofs are
\begin{enumerate}[1.]
\item for any point of $\Hilb{p(t)}{n}$ consider a specialization to a point defined by a strongly stable ideal;
\item determine a sequence of deformations/specializations to move from a strongly stable ideal to another with the goal of getting closer at each step to a fixed strongly stable ideal.
\end{enumerate}
Hartshorne's proof make use of polarization to define the deformation/specialization procedure and a hard part of his argument is to show that applying repeatedly his procedure one reaches the lexicographic ideal. Peeva and Stillman propose a replacement criterion of generators of strongly stable ideals driven by the graded lexicographic order. Hence, in their proof is obvious that at each step the new ideal is closer to the lexicographic ideal than the starting ideal. The main point of their proof is to show that each replacement involves strongly stable ideals and can be realized by means of a deformation/specialization step.

The idea of our proof is very similar to the one of Peeva and Stillman. We now try to point out the main differences.
\begin{itemize}
\item Our replacement criterion of generators is much more flexible because it is not driven by a given term order, but it is based on the combinatorial properties of strongly stable ideals. Term orders help at a later time to move around in the whole set of strongly stable ideals.
\item In our proof the lexicographic ideal can be replaced by any other hilb-segment ideal. In general, for a given Hilbert polynomial there are lots of hilb-segment ideals and some of them can be better suited than the lexicographic ideal to study the Hilbert scheme. For instance, in the case of Hilbert scheme of points, the saturated lexicographic ideal describe a smooth point in the irreducible component of general points, but it has the Hilbert function of aligned points. Whereas the hilb-segment ideal with respect to the graded reverse lexicographic order describe a point in the irreducible component of general points that can be singular, but it has the Hilbert function of general points.
\item In terms of Gr\"obner deformations, we can say that our replacement criterion corresponds to a binomial ideal with exactly two possible initial ideals that are both strongly stable.  Peeva and Stillman replacement criterion corresponds to a binomial ideal with two possible initial ideals: one is always strongly stable, but the other may not be. Hence, they may need an additional Gr\"obner degeneration to a generic initial ideal to restore the strong stability property. \bs
\end{itemize}
\end{remark}

\subsection{Punctual Hilbert schemes}

The case of constant Hilbert polynomials is quite special and allows to prove stronger properties about partial orders $\succeq_\Omega$ and degeneration graphs.

\begin{lemma}\label{lem:maximumConstant}
For each term order $\Omega$, the set of ideals $\SI{d}{n}$ contains the $\Omega$-hilb segment ideal. Hence, there is always the maximum in $\SI{d}{n}$ with respect to $\succeq_{\Omega}$ and $\ssucceq_{\Omega}$.
\end{lemma}
\begin{proof}
Given a term order $\Omega$, let $\mathfrak{L}$ be the set of all monomials in $\kk[\xx]$ of degree $d$ except the $d$ smallest monomials with respect to $\Omega$. The set $\mathfrak{L}$ is closed under the action of increasing elementary moves. In fact, consider two monomials $\xx^\aaa,\xx^\bbb$ of degree $d$. If $\xx^\aaa \in \mathfrak{L}$ and $\xx^\bbb \geq_B \xx^\aaa$, then $\xx^\bbb \geq_\Omega \xx^\aaa$ and $\xx^\bbb \in \mathfrak{L}$. Furthermore, all monomials in $\comp{\mathfrak{L}}$ have minimum equal to 0. Indeed, $\comp{\mathfrak{L}}$ is closed under the action of decreasing moves and there does not exist any such set containing a monomial with minimum greater than 0. This implies that the ideal $(\mathfrak{L})$ is strongly stable and has Hilbert polynomial $p(t) = d$ \cite[Theorem 3.13]{CLMR}.

The second part of the statement follows from the first part applying Theorem \ref{thm:segment case}. 
\end{proof}

Notice that Lemma \ref{lem:maximumConstant} does not say that all ideals in $\SI{d}{n}$ are hilb-segment ideals. See \cite[Proposition 3.16]{CLMR} for the simplest examples of strongly-stable ideals that are not hilb-segment ideals in the case of constant Hilbert polynomials.

We now give some information about the maximum distance between vertices of the Borel graph. We recall that the distance between two vertices of a graph is the number of edges in the shortest path connecting them.

\begin{proposition}\label{prop:distanceVertices}
The distance between vertices $J,J'$ of the Borel graph $\Sk{d}{n}$ is at most $\vert \mathfrak{J} \setminus \mathfrak{J}' \vert$.
\end{proposition}
\begin{proof}
We prove the statement exhibiting a path of length $\vert \mathfrak{J} \setminus \mathfrak{J}' \vert$ between $J$ and $J'$.  Notice that for every ideal $J\in\SI{d}{n}$, the minimum of monomials in $\comp{\mathfrak{J}}$ is 0. Hence, $\mathfrak{J} \setminus \mathfrak{J}' \subset \mathfrak{J}_0$ and $\mathfrak{J}' \setminus \mathfrak{J} \subset \mathfrak{J}'_0$. 

We proceed by induction on $k = \vert \mathfrak{J} \setminus \mathfrak{J}' \vert$. If $k=1$, $J$ and $J'$ are Borel adjacent ideals (see Remark \ref{rk:propertiesAdjacent}\textit{(\ref{rk:propertiesAdjacent_i})}) and the Borel graph $\Sk{d}{n}$ contains the edge $[J{-}J']$. Now, assume that the statement is true for pairs of ideal $J,J'$ such that $\vert \mathfrak{J} \setminus\mathfrak{J}'\vert \leqslant k-1$.

Given $J$ and $J'$ with $\vert \mathfrak{J} \setminus\mathfrak{J}'\vert = k$, consider a maximal element with respect to $\geq_B$ in $\mathfrak{J} \setminus \mathfrak{J}'$ and $\xx^\bbb$ a minimal element with respect to $\geq_B$ in $\mathfrak{J}' \setminus \mathfrak{J}$. The monomial $\xx^\aaa$ is a maximal element of $\comp{\mathfrak{J}'}$ and $\xx^\bbb$ is a minimal element of $\mathfrak{J}'$. Then, consider the set $\mathfrak{J}'' = \mathfrak{J}'\setminus\{\xx^\bbb\} \cup \{\xx^\aaa\}$. It is closed under the action of increasing elementary moves and all monomials in $\comp{\mathfrak{J}''}$ have minimum variable equal to $0$. The ideal $J''$ generated by $\mathfrak{J}''$ is in $\SI{d}{n}$. Sets $\mathfrak{J}'$ and $\mathfrak{J}''$ differ in one element, so $J'$ and $J''$ are Borel adjacent, i.e.~$[J'{-}J] \in E(\Sk{d}{n})$.

Finally, notice that
\[
\mathfrak{J}\setminus\mathfrak{J}'' = \mathfrak{J} \setminus (\mathfrak{J}'\setminus\{\xx^\bbb\} \cup \{\xx^\aaa\}) = (\mathfrak{J} \setminus \mathfrak{J}') \setminus \{\xx^\aaa\}\quad\Rightarrow\quad \vert \mathfrak{J}\setminus\mathfrak{J}'' \vert = k-1.
\] 
By the inductive assumption, there is a path of length $k-1$ from $J$ to $J''$, so that the distance between $J$ and $J'$ is at most $k$.
\end{proof}

\begin{corollary}\label{cor:maxDistance}
The distance between any two vertices of the Borel graph $\Sk{d}{n}$ is at most 
\[
d - \min \left\{ s \ \middle\vert\ \binom{n+s-1}{n}\geqslant d\right\}. 
\]
\end{corollary}
\begin{proof}
By Proposition \ref{prop:distanceVertices} the bound on the maximum distance between vertices is given by the pair of ideals with the smallest intersection. A saturated strongly stable ideal $J^\sat$ with constant Hilbert polynomial has a power $x_1^j$ among the generators (and $j$ is the highest degree of a generator). This implies that $\comp{\mathfrak{J}}$ surely contains $\{x_0^{d-j+1}x_1^{j-1},\ldots,$ $x_0^{d-1}x_1,x_0^d\}$ plus other $d-j$ monomials with maximum variable greater that 1.

Now, consider $J,J' \in \SI{d}{n}$ and let $x_1^j$ and $x_1^{j'}$ the powers of $x_1$ appearing among the generators of $J^\sat$ and $J'^\sat$. Assuming $j < j'$, we have that 
\[
\{x_0^{d-j+1}x_1^{j-1},\ldots,x_0^{d-1}x_1,x_0^d\} \subseteq \comp{\mathfrak{J}} \cap \comp{\mathfrak{J}'}
\]
and
\[
 \vert \mathfrak{J} \setminus \mathfrak{J}' \vert =  \vert \comp{\mathfrak{J}} \setminus \comp{\mathfrak{J}'} \vert = \vert\comp{\mathfrak{J}} \vert -  \vert \comp{\mathfrak{J}} \cap \comp{\mathfrak{J}'} \vert \geqslant d - j.
\]

The saturated lexicographic ideal $L^\sat \in \SI{d}{n}$ is generated by $(x_n,\ldots,x_2,x_1^d)$ and $\comp{\mathfrak{L}} = \{x_0x_1^{d-1},$ $\ldots,x_0^d\}$. Consequently, $\comp{\mathfrak{J}} \cap \comp{\mathfrak{L}} = \{x_0^{d-j+1}x_1^{j-1},\ldots,x_0^{d-1}x_1,x_0^d\}$ and $ \vert \mathfrak{J} \setminus \mathfrak{L} \vert  = d-j$.
In order to maximize $\vert \mathfrak{J} \setminus \mathfrak{L} \vert$, we minimize $j$ looking for an ideal $J$ such that $\comp{\mathfrak{J}}$ contains the $d$ monomials with the highest power of the last variable $x_0$.  We can do this considering the $\mathtt{RevLex}$-hilb-segment ideal. Indeed, with respect to the graded reverse lexicographic order, the monomials
\begin{equation}\label{eq:lastRevLex}
\{x_0^d\} \cup x_0^{d-1}\cdot\kk[x_1,\ldots,x_n]_1 \cup \cdots \cup x_0^{d-j+1}\cdot\kk[x_1,\ldots,x_n]_{j-1}
\end{equation}
form the largest set of monomials with a power of $x_0$ greater than $d-j$. Then, we take the minimum $j$ such that the number of monomials in \eqref{eq:lastRevLex}
\[
1 + n + \cdots + \binom{n-1 + j-1}{n-1} = \binom{n+j-1}{n}
\]
is at least $d$. Such $j$ is the minimum for which the saturation of an ideal $J\in\SI{d}{n}$ has $x_1^j$ among its generators.
\end{proof}

\begin{example}
{(1)} Consider the Hilbert scheme $\Hilb{8}{3}$ that parametrizes subschemes of $\PP^3$ with Hilbert polynomial $p(t)=8$. The set $\SI{8}{3}$ contains 12 strongly stable ideals, 10 of which are hilb-segment ideals with respect to suitable term orders. The Gr\"obner fan $\GF(\Hilb{8}{3})$ has 55 extremal rays and 70 cones of maximal dimension. Lemma \ref{lem:maximumConstant} implies that there are several degeneration graphs that share the same maximum ideal. In Figure \ref{fig:8pointsP3}{\sc\subref{fig:8pointsP3gf}}, there is the Gr\"obner fan of $\Hilb{8}{3}$ with the maximal cones corresponding to term orders $\Omega$ inducing the same maximum $\max_{\succeq} \SI{8}{3}$ grouped together.

 {(2)} The ideal in $\SI{8}{3}$ with the lowest power of $x_1$ among the generators of its saturation is the $\mathtt{RevLex}$-hilb-segment ideal
\[
J_{12} = (x_3^{2},x_2x_3, x_2^3, x_1 x_2^2, x_1^2 x_3, x_1^2 x_2 , x_1^{3})_{\geqslant 8}.
\]
By Corollary \ref{cor:maxDistance}, the  distance between vertices of the Borel graph $\Sk{8}{3}$ is at most 5. In Figure \ref{fig:8pointsP3}{\sc\subref{fig:8pointsP3bg}}, the path from the $\mathtt{RevLex}$-hilb-segment ideal $J_{12}$ to the $\mathtt{DegLex}$-hilb-segment ideal $J_1$ constructed in the proof of Proposition \ref{prop:distanceVertices} is drawn with a thick line.

In fact, the distance between $J_1$ and $J_{12}$ is 3 (and 3 is the maximum distance between vertices of $\Sk{8}{3}$). There are two shortcuts: $[J_{5}{-}J_{12}]$ instead of $[J_5{-}J_{10}],[J_{10}{-}J_{12}]$ and $[J_1{-}J_3]$ instead of $[J_1{-}J_2],[J_2{-}J_3]$.
\end{example}

\captionsetup[subfloat]{position=bottom,width=0.8\textwidth}
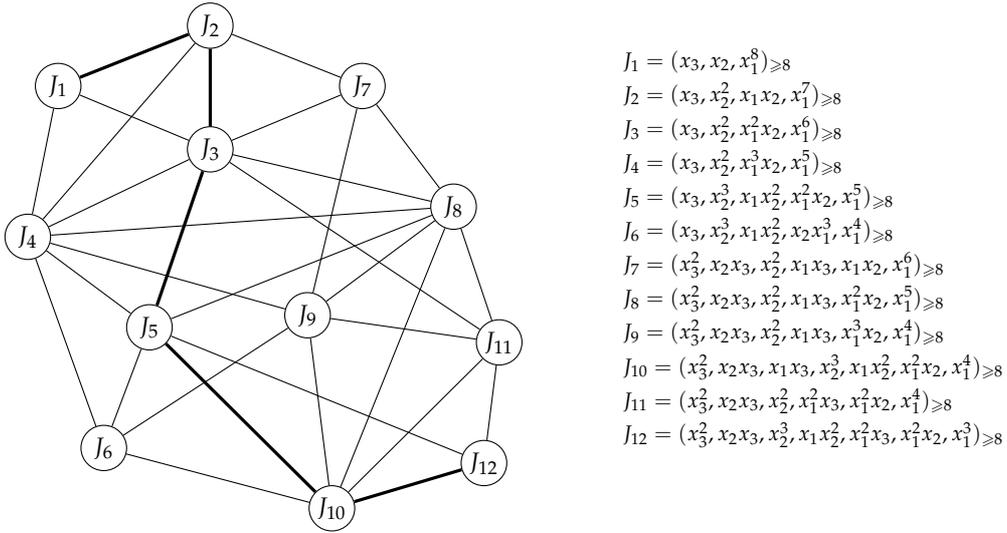
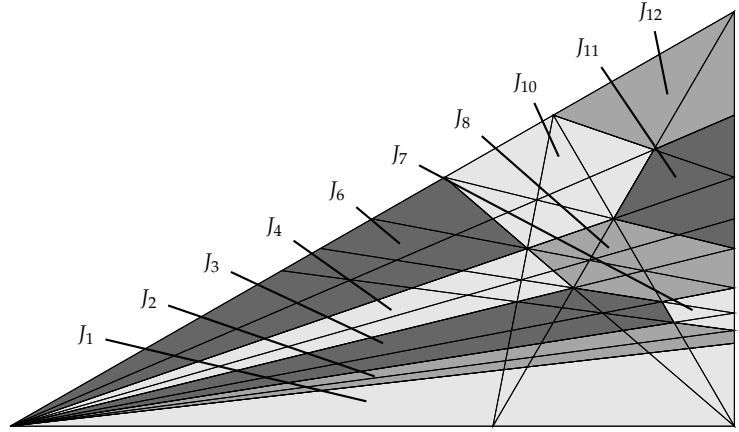
\begin{figure}[!ht]
\begin{center}
\subfloat[][The Borel graph $\Sk{8}{3}$. The path connecting $J_{1}$ and $J_{12}$ constructed in the proof of Proposition \ref{prop:distanceVertices} is highlighted with a thick line.]{\label{fig:8pointsP3bg}
\begin{tikzpicture}
\begin{scope}[scale=0.8]
\node (0) at (0.5,1) [draw,circle,inner sep=0pt,minimum size=0.6cm] {\footnotesize $J_{1}$};
\node (1) at (3,2) [draw,circle,inner sep=0pt,minimum size=0.6cm] {\footnotesize $J_{2}$};
\node (2) at (3,-0.05) [draw,circle,inner sep=0pt,minimum size=0.6cm] {\footnotesize $J_{3}$};
\node (3) at (0,-1.5) [draw,circle,inner sep=0pt,minimum size=0.6cm] {\footnotesize $J_{4}$};
\node (4) at (2,-3) [draw,circle,inner sep=0pt,minimum size=0.6cm] {\footnotesize $J_{5}$};
\node (5) at (1.25,-5) [draw,circle,inner sep=0pt,minimum size=0.6cm] {\footnotesize $J_{6}$};
\node (6) at (5.5,1) [draw,circle,inner sep=0pt,minimum size=0.6cm] {\footnotesize $J_{7}$};
\node (7) at (7,-1) [draw,circle,inner sep=0pt,minimum size=0.6cm] {\footnotesize $J_{8}$};
\node (8) at (4.6,-2.8) [draw,circle,inner sep=0pt,minimum size=0.6cm] {\footnotesize $J_{9}$};
\node (9) at (5,-6) [draw,circle,inner sep=0pt,minimum size=0.6cm] {\footnotesize $J_{10}$};
\node (10) at (7.75,-3.25) [draw,circle,inner sep=0pt,minimum size=0.6cm] {\footnotesize $J_{11}$};
\node (11) at (7.5,-5.25) [draw,circle,inner sep=0pt,minimum size=0.6cm] {\footnotesize $J_{12}$};

 \draw [very thick] (0) -- (1);
  \draw [] (0) -- (2);
  \draw [] (0) -- (3);
  \draw [very thick] (1) -- (2);
  \draw [] (1) -- (3);
  \draw [] (1) -- (6);
  \draw [] (2) -- (3);
  \draw [] (2) -- (6);
  \draw [very thick] (2) -- (4);
  \draw [] (2) -- (7);
  \draw [] (2) -- (10);
  \draw [] (3) -- (4);
  \draw [] (3) -- (7);
  \draw [] (3) -- (5);
  \draw [] (3) -- (8);
  \draw [] (6) -- (7);
  \draw [] (6) -- (8);
  \draw [] (4) -- (7);
  \draw [] (4) -- (5);
  \draw [very thick] (4) -- (9);
  \draw [] (4) -- (11);
  \draw [] (7) -- (10);
  \draw [] (7) -- (8);
  \draw [] (7) -- (9);
  \draw [] (10) -- (11);
  \draw [] (5) -- (8);
  \draw [] (5) -- (9);
  \draw [] (8) -- (10);
  \draw [] (8) -- (9);
  \draw [] (9) -- (10);
  \draw [very thick] (9) -- (11);
\end{scope}

\begin{scope}[scale=0.9,shift={(11.5,1.25)}]
\node  at (0,0) [] {\parbox{5.05cm}{\scriptsize $J_{1} = (x_3,x_2,x_1^{8})_{\geqslant 8}$}};
\node  at (0,-0.5) [] {\parbox{5.05cm}{\scriptsize $J_{2} = (x_3,x_2^2,x_1x_2,x_1^{7})_{\geqslant 8}$}};
\node  at (0,-1) [] {\parbox{5.05cm}{\scriptsize $J_{3} = (x_3,x_2^{2},x_1^{2}x_2,x_1^{6})_{\geqslant 8}$}};
\node  at (0,-1.5) [] {\parbox{5.05cm}{\scriptsize $J_{4} = (x_3,x_2^{2},x_1^{3}x_2,x_1^{5})_{\geqslant 8}$}};
\node  at (0,-2) [] {\parbox{5.05cm}{\scriptsize $J_{5} = (x_3,x_2^3, x_1 x_2^2,x_1^{2}x_2,x_1^{5})_{\geqslant 8}$}};
\node  at (0,-2.5) [] {\parbox{5.05cm}{\scriptsize $J_{6} = (x_3,x_2^{3},x_1x_2^{2},x_2x_1^{3},x_1^{4})_{\geqslant 8}$}};
\node  at (0,-3) [] {\parbox{5.05cm}{\scriptsize $J_{7} = (x_3^2, x_2 x_3, x_2^2, x_1x_3,x_1x_2,x_1^{6})_{\geqslant 8}$}};
\node  at (0,-3.5) [] {\parbox{5.05cm}{\scriptsize $J_{8} = (x_3^{2} , x_2 x_3, x_2^{2},x_1x_3,x_1^{2}x_2,x_1^{5})_{\geqslant 8}$}};
\node  at (0,-4) [] {\parbox{5.05cm}{\scriptsize $J_{9} = (x_3^{2} , x_2 x_3, x_2^{2},x_1x_3,x_1^{3}x_2,x_1^{4})_{\geqslant 8}$}};
\node  at (0,-4.5) [] {\parbox{5.05cm}{\scriptsize $J_{10} = (x_3^{2},x_2x_3,x_1x_3,x_2^{3}, x_1x_2^{2},x_1^{2}x_2,x_1^{4})_{\geqslant 8}$}};
\node  at (0,-5) [] {\parbox{5.05cm}{\scriptsize $J_{11} = (x_3^{2},x_2x_3,x_2^{2}, x_1^{2}x_3, x_1^{2} x_2,x_1^{4})_{\geqslant 8}$}};
\node  at (0,-5.5) [] {\parbox{5.05cm}{\scriptsize $J_{12} = (x_3^{2},x_2x_3, x_2^3, x_1 x_2^2, x_1^2 x_3, x_1^2 x_2 , x_1^{3})_{\geqslant 8}$}};
\end{scope}
\end{tikzpicture}
}

\subfloat[][The Gr\"obner fan $\GF(\Hilb{8}{3})$. Adjacent polygons colored with the same shade of gray corresponds to term orders $\Omega$ whose degeneration graphs have the same maximum in $\SI{8}{3}$ with respect to $\succeq_\Omega$.]{\label{fig:8pointsP3gf}
\begin{tikzpicture}[scale=11]

		\draw [ultra thin,fill=black!10] (0,-.4) -- (0,-.5) --  (-.866,-.5)-- cycle;
		\draw [ultra thin,fill=black!35] (0,-.4) --  (0,-.384615)  -- (-.0721667,-.375) -- (-.866,-.5)--cycle;
		\draw [ultra thin,fill=black!60] (-.192444,-.333333) -- (-.0866,-.35)  --(-.0721667,-.375)-- (-.866,-.5)--cycle;
		\draw [ultra thin,fill=black!10] (-.192444,-.333333) -- (-.247429,-.285714) -- (-.866,-.5)--cycle;
		\draw [ultra thin,fill=black!60] (-.3464,-.2) -- (-.247429,-.285714) -- (-.866,-.5)--cycle;
		\draw [ultra thin,fill=black!10]  (0,-.333333) --  (-.0866,-.35) -- (-.0721667,-.375) --  (0,-.384615) --cycle;
		\draw [ultra thin,fill=black!35]   (-.247429,-.285714) --  (-.192444,-.333333)  -- (-.0866,-.35)  --(0,-.333333)-- (0,-.285714)--  (-.144333,-.25)-- cycle;
		\draw [ultra thin,fill=black!10]   (-.3464,-.2) -- (-.2165,-.125)  -- (-.0962222,-.166667)-- (-.144333,-.25) -- (-.247429,-.285714)  -- cycle;
		\draw [ultra thin,fill=black!60]    (-.144333,-.25)  -- (0,-.285714) -- (0,-.125) -- (-.0962222,-.166667) -- cycle;
		\draw [ultra thin,fill=black!35]   (0,0) -- (-.2165,-.125) -- (-.0962222,-.166667) -- (0,-.125) --  cycle;

	\node (12) at (-0.1,-0.0) [] {\scriptsize $J_{12}$};
	\draw[thick] (12) -- (-0.08,-0.1);
	\node (11) at (-0.175,-0.0433) [] {\scriptsize $J_{11}$};
	\draw[thick] (11) -- (-0.07,-0.2);
	\node (10) at (-0.25,-0.0866) [] {\scriptsize $J_{10}$};
	\draw[thick] (10) -- (-0.21,-0.175);
	\node (8) at (-0.325,-0.13) [] {\scriptsize $J_{8}$};
	\draw[thick] (8) -- (-0.15,-0.285);
	\node (7) at (-0.400,-0.1733) [] {\scriptsize $J_{7}$};
	\draw[thick] (7) -- (-0.05,-0.36);
	\node (6) at (-0.475,-0.2166) [] {\scriptsize $J_{6}$};
	\draw[thick] (6) -- (-0.4,-0.28);
	\node (4) at (-0.55,-0.26) [] {\scriptsize $J_{4}$};
	\draw[thick] (4) -- (-0.41,-0.36);
	\node (3) at (-0.625,-0.3033) [] {\scriptsize $J_{3}$};
	\draw[thick] (3) -- (-0.42,-0.4);
	\node (2) at (-0.7,-0.3466) [] {\scriptsize $J_{2}$};
	\draw[thick] (2) -- (-0.43,-0.44);
	\node (1) at (-0.775,-0.39) [] {\scriptsize $J_{1}$};
	\draw[thick] (1) -- (-0.44,-0.47);
		
              \draw [-,ultra thin] (-.866,-.5) -- (-.288667,-.5) -- (-.275545,-.431818) -- cycle; 
              \draw [-,ultra thin] (-.866,-.5) -- (-.273474,-.421053) -- (-.275545,-.431818) -- cycle; 
              \draw [-,ultra thin] (-.288667,-.5) -- (-.247429,-.428571) -- (-.275545,-.431818) -- cycle; 
              \draw [-,ultra thin] (-.866,-.5) -- (-.270625,-.40625) -- (-.273474,-.421053) -- cycle; 
              \draw [-,ultra thin] (-.240556,-.416667) -- (-.247429,-.428571) -- (-.275545,-.431818) -- (-.273474,-.421053) -- cycle; 
              \draw [-,ultra thin] (0,-.5)  -- (-.101882,-.411765)  -- (-.247429,-.428571) -- (-.288667,-.5) -- cycle; 
              \draw [-,ultra thin] (-.866,-.5) -- (-.266462,-.384615) -- (-.270625,-.40625) -- cycle; 
              \draw [-,ultra thin] (-.230933,-.4) -- (-.240556,-.416667) -- (-.273474,-.421053) -- (-.270625,-.40625) -- cycle; 
              \draw [-,ultra thin] (-.115467,-.4)  -- (-.101882,-.411765) -- (-.247429,-.428571)-- (-.240556,-.416667) -- cycle; 
              \draw [-,ultra thin] (0,-.5) -- (-.054125,-.40625) -- (-.101882,-.411765) -- cycle; 
              \draw [-,ultra thin] (-.866,-.5) -- (-.2598,-.35) -- (-.266462,-.384615) -- cycle; 
              \draw [-,ultra thin] (-.2165,-.375)  -- (-.230933,-.4) -- (-.270625,-.40625)-- (-.266462,-.384615) -- cycle; 
              \draw [-,ultra thin] (-.133231,-.384615)-- (-.115467,-.4) -- (-.240556,-.416667) -- (-.230933,-.4)  -- cycle; 
              \draw [-,ultra thin] (-.0618571,-.392857) -- (-.054125,-.40625) -- (-.101882,-.411765)-- (-.115467,-.4)  -- cycle; 
              \draw [-,ultra thin] (0,-.5) -- (0,-.4) -- (-.054125,-.40625) -- cycle; 
              \draw [-,ultra thin] (-.866,-.5) -- (-.2598,-.35) -- (-.319053,-.342105) -- cycle; 
              \draw [-,ultra thin] (-.2598,-.35)   -- (-.20619,-.357143) -- (-.2165,-.375) -- (-.266462,-.384615)-- cycle; 
              \draw [-,ultra thin] (-.157455,-.363636) -- (-.133231,-.384615) -- (-.230933,-.4)  -- (-.2165,-.375) -- cycle; 
              \draw [-,ultra thin] (-.0721667,-.375)  -- (-.0618571,-.392857) -- (-.115467,-.4)-- (-.133231,-.384615) -- cycle; 
              \draw [-,ultra thin] (0,-.384615)  -- (0,-.4) -- (-.054125,-.40625)-- (-.0618571,-.392857) -- cycle; 
              \draw [-,ultra thin] (-.2598,-.35) -- (-.254706,-.323529) -- (-.319053,-.342105) -- cycle; 
              \draw [-,ultra thin] (-.866,-.5) -- (-.384889,-.333333) -- (-.319053,-.342105) -- cycle; 
              \draw [-,ultra thin] (-.157455,-.363636) -- (-.2165,-.375) -- (-.20619,-.357143) -- cycle; 
              \draw [-,ultra thin] (-.192444,-.333333) -- (-.2598,-.35) -- (-.20619,-.357143) -- cycle; 
              \draw [-,ultra thin] (-.157455,-.363636) -- (-.0721667,-.375) -- (-.133231,-.384615) -- cycle; 
              \draw [-,ultra thin] (-.0721667,-.375) -- (0,-.384615) -- (-.0618571,-.392857) -- cycle; 
              \draw [-,ultra thin] (-.32475,-.3125)  -- (-.254706,-.323529) -- (-.319053,-.342105) -- (-.384889,-.333333)-- cycle; 
              \draw [-,ultra thin] (-.192444,-.333333) -- (-.2598,-.35) -- (-.254706,-.323529) -- cycle; 
              \draw [-,ultra thin] (-.866,-.5) -- (-.384889,-.333333) -- (-.458471,-.323529) -- cycle; 
              \draw [-,ultra thin] (-.192444,-.333333) -- (-.157455,-.363636) -- (-.20619,-.357143) -- cycle; 
              \draw [-,ultra thin] (-.0866,-.35) -- (-.157455,-.363636) -- (-.0721667,-.375) -- cycle; 
              \draw [-,ultra thin] (0,-.363636) -- (-.0721667,-.375) -- (0,-.384615) -- cycle; 
              \draw [-,ultra thin] (-.247429,-.285714) -- (-.32475,-.3125) -- (-.254706,-.323529) -- cycle; 
              \draw [-,ultra thin] (-.32475,-.3125) -- (-.404133,-.3) -- (-.458471,-.323529)-- (-.384889,-.333333)  -- cycle; 
              \draw [-,ultra thin] (-.192444,-.333333) -- (-.2165,-.3125) -- (-.254706,-.323529) -- cycle; 
              \draw [-,ultra thin] (-.866,-.5) -- (-.54125,-.3125) -- (-.458471,-.323529) -- cycle; 
              \draw [-,ultra thin] (-.192444,-.333333) -- (-.0866,-.35) -- (-.157455,-.363636) -- cycle; 
              \draw [-,ultra thin] (-.0866,-.35) -- (0,-.363636) -- (-.0721667,-.375) -- cycle; 
              \draw [-,ultra thin] (-.247429,-.285714) -- (-.2165,-.3125) -- (-.254706,-.323529) -- cycle; 
              \draw [-,ultra thin] (-.247429,-.285714)  -- (-.333077,-.269231) -- (-.404133,-.3) -- (-.32475,-.3125) -- cycle; 
              \draw [-,ultra thin] (-.494857,-.285714)  -- (-.404133,-.3) -- (-.458471,-.323529) -- (-.54125,-.3125) -- cycle; 
              \draw [-,ultra thin] (-.192444,-.333333) -- (-.1732,-.3) -- (-.2165,-.3125) -- cycle; 
              \draw [-,ultra thin] (-.10825,-.3125) -- (-.192444,-.333333) -- (-.0866,-.35) -- cycle; 
              \draw [-,ultra thin] (0,-.333333) -- (-.0866,-.35) -- (0,-.363636) -- cycle; 
              \draw [-,ultra thin] (-.247429,-.285714) -- (-.1732,-.3) -- (-.2165,-.3125) -- cycle; 
              \draw [-,ultra thin] (-.247429,-.285714) -- (-.288667,-.25) -- (-.333077,-.269231) -- cycle; 
              \draw [-,ultra thin] (-.433,-.25)  -- (-.333077,-.269231) -- (-.404133,-.3) -- (-.494857,-.285714) -- cycle; 
              \draw [-,ultra thin] (-.10825,-.3125) -- (-.192444,-.333333) -- (-.1732,-.3) -- cycle; 
              \draw [-,ultra thin] (-.10825,-.3125) -- (0,-.333333) -- (-.0866,-.35) -- cycle; 
              \draw [-,ultra thin] (-.144333,-.25) -- (-.247429,-.285714) -- (-.1732,-.3) -- cycle; 
              \draw [-,ultra thin] (-.247429,-.285714) -- (-.236182,-.227273) -- (-.288667,-.25) -- cycle; 
              \draw [-,ultra thin] (-.3464,-.2) -- (-.288667,-.25) -- (-.333077,-.269231) -- (-.433,-.25) -- cycle; 
              \draw [-,ultra thin] (-.10825,-.3125) -- (-.123714,-.285714) -- (-.1732,-.3) -- cycle; 
              \draw [-,ultra thin] (0,-.285714) -- (-.10825,-.3125) -- (0,-.333333) -- cycle; 
              \draw [-,ultra thin] (-.144333,-.25) -- (-.123714,-.285714) -- (-.1732,-.3) -- cycle; 
              \draw [-,ultra thin] (-.144333,-.25) -- (-.247429,-.285714) -- (-.236182,-.227273) -- cycle; 
              \draw [-,ultra thin] (-.3464,-.2) -- (-.236182,-.227273) -- (-.288667,-.25) -- cycle; 
              \draw [-,ultra thin] (0,-.285714)  -- (-.0666154,-.269231) -- (-.123714,-.285714) -- (-.10825,-.3125) -- cycle; 
              \draw [-,ultra thin] (-.144333,-.25) -- (-.0666154,-.269231) -- (-.123714,-.285714) -- cycle; 
              \draw [-,ultra thin] (-.144333,-.25) -- (-.1732,-.2) -- (-.236182,-.227273) -- cycle; 
              \draw [-,ultra thin] (-.2165,-.125) -- (-.3464,-.2) -- (-.236182,-.227273) -- cycle; 
              \draw [-,ultra thin] (0,-.285714) -- (0,-.25) -- (-.0666154,-.269231) -- cycle; 
              \draw [-,ultra thin] (0,-.2)  -- (0,-.25) -- (-.0666154,-.269231)-- (-.144333,-.25) -- cycle; 
              \draw [-,ultra thin] (-.144333,-.25) -- (-.0962222,-.166667) -- (-.1732,-.2) -- cycle; 
              \draw [-,ultra thin] (-.2165,-.125) -- (-.1732,-.2) -- (-.236182,-.227273) -- cycle; 
              \draw [-,ultra thin] (0,-.2) -- (-.144333,-.25) -- (-.0962222,-.166667) -- cycle; 
              \draw [-,ultra thin] (-.2165,-.125) -- (-.0962222,-.166667) -- (-.1732,-.2) -- cycle; 
              \draw [-,ultra thin] (0,-.2) -- (0,-.125) -- (-.0962222,-.166667) -- cycle; 
              \draw [-,ultra thin] (0,0) -- (-.2165,-.125) -- (-.0962222,-.166667) -- cycle; 
              \draw [-,ultra thin] (0,0) -- (0,-.125) -- (-.0962222,-.166667) -- cycle; 

      \end{tikzpicture}
}
\caption{The Borel graph and the Gr\"obner fan of the Hilbert scheme $\Hilb{8}{3}$.}
\label{fig:8pointsP3}
\end{center}
\end{figure}

\subsection{Irreducibility of the Hilbert scheme}

We recall a nice results from \cite{BCR-GG} that explains how to use maximal strongly stable ideals with respect to $\ssucceq_\Omega$ to study the irreducibility of $\Hilb{p(t)}{n}$. For any term order $\Omega$, we denote by $\mathfrak{m}_{p(t)}^n(\Omega)$ the number of ideals in $\max_{\ssucceq_\Omega} \SI{p(t)}{n}$. 

\begin{proposition}[{\cite[Proposition 9]{BCR-GG}}]\label{prop:gg}
Let $\Omega$ be a term order. The Hilbert scheme  $\Hilb{p(t)}{n}$ has at least $\mathfrak{m}_{p(t)}^n(\Omega)$ irreducible components.
\end{proposition}

To make Proposition \ref{prop:gg} meaningful and effective, one has to look for the term order $\Omega$ that gives the best lower bound on the number of irreducible components of $\Hilb{p(t)}{n}$. From a computational point of view, finding such $\Omega$ from the statement seems as difficult as finding a needle in the haystack. In this context, the problem becomes treatable. 

Proposition \ref{prop:BorelDef implies order} and the inclusion $\max_{\ssucceq_\Omega}\SI{p(t)}{n} \subseteq \max_{\succeq_\Omega}\SI{p(t)}{n}$ suggest to examine maximal cones of the Gr\"obner fan $\GF(\Hilb{p(t)}{n})$. For each cone $\mathcal{C}\in\GF(\Hilb{p(t)}{n})$ of maximal dimension, we want to determine
\[
\mathfrak{m}_{p(t)}^{n}(\mathcal{C}) := \max\left\{ \mathfrak{m}_{p(t)}^n(\Omega)\ \middle\vert\ \Omega\text{~term order s.t.~} \mathcal{C} = \overline{\mathcal{C}_{p(t)}^n(\Omega)}\right\}.
\]
First, we consider an interior point $\ooo \in \mathcal{C}$, i.e.~$\mathcal{C} = \overline{\mathcal{C}_{p(t)}^n(\ooo)}$. By Proposition \ref{prop:maxConesTO}, the $\ooo$-degeneration graph is a directed graph. Second, we compute the set $\textsf{M}$ of vertices with no incoming edge in $\DG{p(t)}{n}{\ooo}$. We have $\mathfrak{m}_{p(t)}^{n}(\mathcal{C}) \leqslant \vert \textsf{M} \vert$. Third, we look for the largest subset $\textsf{M}'\subsetneq \textsf{M}$ such that there exists a term order $\Omega$ such that $\mathcal{C}_{p(t)}^n(\Omega) = \mathcal{C}_{p(t)}^n(\ooo)$ and the ideals in $\textsf{M}'$ are not comparable with respect to $\ssucceq_\Omega$. Hence, $\mathfrak{m}_{p(t)}^{n}(\mathcal{C}) = \vert \textsf{M}' \vert$. Finally, we compute the maximum of $\mathfrak{m}_{p(t)}^{n}(\mathcal{C})$ for $\mathcal{C}$ varying among cones of maximal dimension of $\GF(\Hilb{p(t)}{n})$.

\begin{lemma}\label{lem:lowerBound}
The Hilbert scheme $\Hilb{p(t)}{n}$ has at least $\mathfrak{m}_{p(t)}^n$ irreducible components, where
\begin{equation}\label{eq:lowerBound}
\mathfrak{m}_{p(t)}^n := \max \left\{ \mathfrak{m}_{p(t)}^n(\mathcal{C})\ \middle\vert\  \mathcal{C} \in \GF(\Hilb{p(t)}{n}) \text{~of maximal dimension}\right\}.
\end{equation}
\end{lemma}

\begin{example}\label{ex:6t-3 part 1}
Consider the Hilbert scheme $\Hilb{6t-3}{3}$ parametrizing 1-dimensional subschemes of $\PP^3$ of degree $6$ and arithmetic genus $4$. The Borel graph $\Sk{6t-3}{3}$ has 31 vertices and 110 edges (see Figure \ref{fig:6t-3}{\sc\subref{fig:6t-3 bg}}) and the Gr\"obner fan $\GF(\Hilb{6t-3}{3})$ has 268 cones of maximal dimension and 186 extremal rays (see Figure \ref{fig:6t-3}{\sc\subref{fig:6t-3 gf}}). For every cone, 
$\mathfrak{m}_{6t-3}^3(\mathcal{C})$ coincides with the cardinality of vertices with no incoming edges in the $\ooo$-degeneration graph for some $\ooo$ in the interior of $\mathcal{C}$. We have $251$ cones with $\mathfrak{m}_{6t-3}^3(\mathcal{C}) = 1$, $13$ cones with $\mathfrak{m}_{6t-3}^3(\mathcal{C}) = 2$ and 4 cones with $\mathfrak{m}_{7t-5}^3(\mathcal{C}) = 3$. Therefore, $\mathfrak{m}_{6t-3}^{3} = 3$ and we can affirm that the Hilbert scheme $\Hilb{6t-3}{3}$ has at least 3 irreducible components.
 \end{example}
 
We computed a lot of examples and we always found that $\mathfrak{m}_{p(t)}^n(\mathcal{C})$ coincides with the number of vertices with no incoming edge of the degeneration graph $\DG{p(t)}{n}{\ooo}$, where $\ooo$ is any vector in the interior of $\mathcal{C}$. Hence, we propose the following conjecture.

\begin{conjecture}\label{conj:v1}
For every cone of maximal dimension $\mathcal{C} \in \GF(\Hilb{p(t)}{n})$, there exists a term order $\Omega$ such that $\mathcal{C} = \overline{\mathcal{C}_{p(t)}^n(\Omega)}$ and
\[
\mathfrak{m}_{p(t)}^n(\mathcal{C}) = \mathfrak{m}_{p(t)}^n(\Omega) = \left\vert \max_{\ssucceq_\Omega} \SI{p(t)}{n} \right\vert = \left\vert \max_{\succeq_\Omega} \SI{p(t)}{n} \right\vert. 
\]
\end{conjecture}

 If the conjecture were true, we could compute $\mathfrak{m}_{p(t)}^n(\mathcal{C})$ looking at the $\ooo$-degeneration graph for a single vector $\ooo$ in the interior of $\mathcal{C}$ and compute $\mathfrak{m}_{p(t)}^n$ considering a finite number of degeneration graph (one for each maximal cone of the Gr\"obner fan).
However, computing the Gr\"obner fan $\GF(\Hilb{p(t)}{n})$ can become computationally demanding (see Table \ref{tab:gfan cpuTime}) and the Gr\"obner fan may have a huge number of maximal cones, making the naif procedure ineffective. Moreover, in the previous section we saw that there are directed degeneration graphs with a unique maximal element. Hence, the corresponding maximal cones can be not considered a priori. We now focus on the search for weight vectors $\ooo\in \mathcal{W}$ inducing a direct $\ooo$-degeneration graph with more than one vertex with no incoming edges.

\begin{definition}
Let $J \subset \kk[\xx]$ be a strongly stable ideal in $\SI{p(t)}{n}$. We call \emph{maximality cone} of $J$ (M-cone of $J$ for short) the set
\[
\MC(J) :=  \left\{ \ooo \in \mathcal{W} \ \middle\vert\ J \text{~has no incoming edges in~} \mathscr{G}_{p(t)}^n(\ooo) \right\} 
\] 
and we call \emph{segment cone} of $J$ (S-cone of $J$ for short) the set
\[
\SC(J) :=  \left\{ \ooo \in \mathcal{W} \ \middle\vert\  \langle \aaa , \ooo\rangle > \langle \bbb , \ooo\rangle,\ \forall\ \xx^\aaa \in \mathfrak{J},\ \forall\ \xx^\bbb \in \comp{\mathfrak{J}} \right\}.
\] 
\end{definition}

Both sets are either empty or open polyhedral cones of maximal dimension. The maximality cone of $J$ is the set of solutions of the system of inequalities
\[
\begin{cases}
\omega_0 > 0\\
\omega_i > \omega_{i-1},& i=1,\ldots,n,\\
\langle \aaa , \ooo\rangle > \langle \aaa' , \ooo\rangle,& \forall\ [J_{\aaa}{-}J'_{\aaa'}] \in E(\mathscr{G}_{p(t)}^n),
\end{cases}
\]
and it is equal to interior of the union of cones of $\GF(\Hilb{p(t)}{n})$ corresponding to degeneration graphs in which the vertex $\mathfrak{J}$ has no incoming edges.

The segment cone is polyhedral by definition. Notice that it is not necessary to consider all inequalities $\langle \aaa,\ooo\rangle > \langle \bbb , \ooo\rangle,\ \forall\ \xx^\aaa \in \mathfrak{J},\ \forall\ \xx^\bbb \in \comp{\mathfrak{J}}$, but it suffices to restrict to those corresponding to $\xx^\aaa \in \mathfrak{J}$ minimal and $\xx^\bbb \in \comp{\mathfrak{J}}$ maximal with respect to the Borel order $\geq_B$. In general, an S-cone can have any type of relation with the cones of maximal dimension of $\GF(\Hilb{p(t)}{n})$ (see Example \ref{ex:6t-3 part 2}).

The segment cone of an ideal is contained in the maximality cone and we are interested in ideals for which the inclusion is proper.
\begin{definition}\label{def:irregular}
We say that an ideal $J \in \SI{p(t)}{n}$ is \emph{regular} if $\MC(J) = \SC(J)$ and \emph{irregular} if $\MC(J)\setminus \SC(J) \neq \emptyset$, i.e.~$\SC(J) \subsetneq \MC(J)$. 
\end{definition}

In light of Definition \ref{def:irregular}, in order to determine $\mathfrak{m}_{p(t)}^n$ we can consider irregular ideals in $\SI{p(t)}{n}$ and look for subsets $\{J_1,\ldots,J_s\}$ such that
\[
\MC(J_1,\ldots,J_s) :=  \bigcap_{i = 1}^s \big( \MC(J_i)\setminus \SC(J_i) \big) =  \bigcap_{i = 1}^s \MC(J_i)\neq \emptyset.
\]
For each subset with this property, we have to check that there exists a term order $\Omega$  such that ideals $J_1,\ldots,J_s$ are maximal elements for $\succeq_\Omega$ and not comparable with respect to $\ssucceq_\Omega$. The cardinality of the largest set of ideals with this property is $\mathfrak{m}_{p(t)}^{n}$. Conjecture \ref{conj:v1} can be restated as follows.

\begin{conjecture}\label{conj:v2}
For every set of ideals $J_1,\ldots,J_s \in \SI{p(t)}{n}$ such that $\MC(J_1,\ldots,J_2) \neq \emptyset$, there exists a term order $\Omega$ such that $ \{J_1,\ldots,J_s\} = \max_{\ssucceq_{\Omega}} \SI{p(t)}{n}$.
\end{conjecture}

\begin{example}[continues Example \ref{ex:6t-3 part 1}]\label{ex:6t-3 part 2}
Among the $31$ elements of $\SI{6t-3}{3}$, we have
\begin{itemize}
\item 8 regular ideals (5 of them have empty M-cone, 3 of them are hilb-segment ideals with the M-cone and S-cone coinciding);
\item 23 irregular ideals (10 of which are hilb-segment ideals).
\end{itemize}
There are 59 cones of maximal dimension of $\GF(\Hilb{6t-3}{3})$ whose intersection with at least one segment cone is a cone of maximal dimension (see Figure \ref{fig:6t-3}{\sc\subref{fig:6t-3 gf}}). For these cones, $\mathfrak{m}_{6t-3}^3(\mathcal{C})$ is surely 1.
 
 In all degeneration graphs corresponding to the four cones with $\mathfrak{m}_{6t-3}^3(\mathcal{C}) = 3$, the vertices with no incoming edge correspond to ideals
 \[
 \begin{split}
&J_{26}\footnotemark = (x_3^3, x_2 x_3^2, x_2^2 x_3, x_1 x_3^2, x_1^2 x_2 x_3,x_1^3 x_3,{x}_{2}^{6})_{\geqslant 12},\\
&J_{30} = ({x}_{3}^{3},x_2{x}_{3}^{2},x_2^2 x_3,x_1 x_3^2, x_1 x_2 x_3,{x}_{2}^{5})_{\geqslant 12},\\
&J_{31} = ({x}_{3}^{2},x_2^2 {x}_{3},{x}_{2}^{4})_{\geqslant 12}.
 \end{split}
 \] 
\footnotetext{The index labeling an ideal is the position of the ideal in the list of strongly stable ideals in $\kk[x_0,x_1,x_2,x_3]$ with Hilbert polynomial $p(t)=6t-3$  produced by the algorithm implemented in the \textit{Macaulay2} package \texttt{StronglyStableIdeals.m2} \cite{AL}.} The interior of union of these four cones of $\GF(\Hilb{6t-3}{3})$ is equal to the intersection of maximality cones 
 \[
 \MC(J_{26},J_{30},J_{31}) = \MC(J_{26})\cap \MC(J_{30})\cap\MC(J_{31})
 \]
 (see Figure \ref{fig:6t-3}{\sc\subref{fig:6t-3 gf}}).
\end{example}

We now give a necessary condition for two irregular ideals $J$ and $J'$ to have non-empty intersection $\MC(J) \cap \MC(J')$. For a strongly stable ideal $J$, we denote by $J^\sat_{x_0}$ the saturation of the ideal $J + (x_0)$ in $\kk[x_1,\ldots,x_n]$. The ideal $J + (x_0)$ describes the hyperplane section of the scheme $\Proj \kk[\xx]/J$ with the hyperplane defined by the equation $x_0 = 0$. Notice that if two ideals $J$ and $J'$ have the same hyperplane section, then $\mathfrak{J}_{\geqslant 1} = \mathfrak{J}'_{\geqslant 1}$, so that $\mathfrak{J} \setminus \mathfrak{J}' \subset \mathfrak{J}_0$ and $\mathfrak{J}' \setminus \mathfrak{J} \subset \mathfrak{J}'_0$.

We denote by $\HS{p(t)}{n}$ the set of strongly stable ideals $H \subset \kk[x_1,\ldots,x_n]$ describing a possible hyperplane section of an ideal in $\SI{p(t)}{n}$. The set $\HS{p(t)}{n}$ is a subset of $\SI{\Delta p(t)}{n-1}$ and an ideal $H \subset \kk[x_1,\ldots,x_n]$ in $\SI{\Delta p(t)}{n-1}$ belongs to $\HS{p(t)}{n}$ if the Hilbert polynomial of $H\cdot \kk[\xx]$ is equal to $p(t)-h$ with $h \geqslant 0$ (see \cite{CLMR,Efficient,AL} for more details). 

For all $H^\sat \in \HS{p(t)}{n}$, we denote by $\SI{p(t)}{n,H}$ the subset 
\[
\SI{p(t)}{n,H} := \left\{ J \in \SI{p(t)}{n}\ \middle\vert\ J^\sat_{x_0} = H^\sat \right\}
\]
and by $\DG{p(t)}{n,H}{\Omega}$ the subgraph of $\DG{p(t)}{n}{\Omega}$ containing only vertices in $\SI{p(t)}{n,H}$ and edges among them.

\begin{proposition}
For any term order $\Omega$ and for any ideal $H \in \HS{p(t)}{n}$, the set of ideals $\SI{p(t)}{n,H}$ has maximum with respect to both $\ssucceq_{\Omega}$ and $\succeq_{\Omega}$.
\end{proposition}
\begin{proof}
First, we prove that there is the maximum with respect to $\ssucceq_{\Omega}$ by constructing it. Consider the ideal $H' = H\cdot \kk[\xx]$ and the set $\mathfrak{H}'$ of its monomials of degree $r$. By the assumption, we have that $h = \vert\mathfrak{H}' \vert - q(r) \geqslant 0$. If $h=0$, then $H'$ is the unique ideal in $\SI{p(t)}{n,H}$ and it is also maximal with respect to $\ssucceq_{\Omega}$. If $h>0$, the homogeneous piece of degree $r$ of ideals in $\SI{p(t)}{n,H}$ can be obtained from $\mathfrak{H}'$ removing $h$ monomials $\{\xx^{\aaa_1},\ldots,\xx^{\aaa_h}\}$ with minimum 0 such that $\mathfrak{H}'\setminus \{\xx^{\aaa_1},\ldots,\xx^{\aaa_h}\}$ remains closed under increasing Borel elementary moves.

Let us call $L$ the ideal whose set of monomials of degree $r$ is $\mathfrak{H}'\setminus\{\xx^{\aaa_1},\ldots,\xx^{\aaa_h}\}$, where the monomials we remove are the $h$ smallest monomials with respect to $\Omega$ in $\mathfrak{H}'$ with minimum 0. By construction $L$ is strongly stable and contained in $\SI{p(t)}{n,H}$. For any other ideal $J \in \SI{p(t)}{n,H}$, we have 
\[
\mathfrak{J} = \mathfrak{H}'\setminus\{\xx^{\bbb_1},\ldots,\xx^{\bbb_h}\}\quad\text{and}\quad\mathfrak{L} \cap \mathfrak{J} = \mathfrak{H}' \setminus (\{\xx^{\aaa_1},\ldots,\xx^{\aaa_h}\} \cup \{\xx^{\bbb_1},\ldots,\xx^{\bbb_h}\} ) 
\]
that imply
\[
\begin{split}
\mathfrak{L} &{}= (\mathfrak{L} \cap \mathfrak{J}) \cup \left(\{\xx^{\bbb_1},\ldots,\xx^{\bbb_h}\} \setminus \{\xx^{\aaa_1},\ldots,\xx^{\aaa_h}\}\right),\\
\mathfrak{J} &{}= (\mathfrak{L} \cap \mathfrak{J}) \cup \left(\{\xx^{\aaa_1},\ldots,\xx^{\aaa_h}\} \setminus \{\xx^{\bbb_1},\ldots,\xx^{\bbb_h}\}\right).
\end{split}
\]
By construction, all elements in $\{\xx^{\bbb_1},\ldots,\xx^{\bbb_h}\} \setminus \{\xx^{\aaa_1},\ldots,\xx^{\aaa_h}\}$ are greater with respect to $\geq_\Omega$ than all monomials in $\{\xx^{\aaa_1},\ldots,\xx^{\aaa_h}\} \setminus \{\xx^{\bbb_1},\ldots,\xx^{\bbb_h}\}$ . By Lemma \ref{lem:addMonomials}, $L \ssucceq_{\Omega} J$, for all $J \in \SI{p(t)}{n,H}$.

\smallskip

In order to prove that $L$ is also the maximum with respect to $\succeq_{\Omega}$, we repeat the argument used in Theorem \ref{thm:segment case}. Applying the procedure introduced in the proof of the aforementioned theorem, we encounter $\xx^\aaa = \max_{\geq_\Omega} (\mathfrak{L}\setminus\mathfrak{J})$ and $\xx^\bbb = \min_{\geq_\Omega} (\mathfrak{J}\setminus\mathfrak{L})$ with minimum 0, that satisfy conditions $(\dagger)$ and $(\ddagger)$. The ideal $I$ generated by $\mathfrak{I} = \mathfrak{J} \setminus\{\xx^\bbb\}\cup \{\xx^\aaa\}$ is Borel-adjacent to $J$, it is contained in $\SI{p(t)}{n,H}$ as $\mathfrak{I}_{\geqslant 1} = \mathfrak{J}_{\geqslant 1}$, and $\xx^\aaa >_{\Omega} \xx^\bbb$ implies that $[I{\xrightarrow{\text{\tiny$\Omega$}}}J]$ is an edge of $\DG{p(t)}{n,H}{\Omega}$, so that $I \succ_\Omega J$.
\end{proof}

\begin{corollary}\label{cor:necessCond}
Let $J$ and $J'$ be two irregular ideals in $\SI{p(t)}{n}$. If $\MC(J) \cap \MC(J') \neq \emptyset$, then $J^\sat_{x_0} \neq J'^\sat_{x_0}$.
\end{corollary}

\begin{example}[continues Example \ref{ex:6t-3 part 2}]
The hyperplane sections of the maximal elements are
\[
(J_{26})_{x_0}^\sat  = (x_3,x_2^6),\qquad (J_{30})_{x_0}^\sat  = (x_3^2,x_2x_3,x_2^5)\qquad\text{and}\qquad (J_{31})_{x_0}^\sat=  (x_3^2,x_2^2x_3,x_2^4).
\]
There is no other ideal in $\HS{6t-3}{3}$, so 3 is the maximum number of components of $\Hilb{6t-3}{3}$ that can be detected with this method. 
\end{example}

\begin{remark}
In the case of constant Hilbert polynomials, Lemma \ref{lem:maximumConstant} implies that all ideals in $\SI{d}{n}$ are regular. Namely, either $J \in \SI{d}{n}$ is a hilb-segment ideal, so that $\SC(J) \neq \emptyset$ and $\MC(J) = \SC(J)$, or $\MC(J) = \emptyset$. Hence, we have that $\mathfrak{m}_{d}^n = 1$ for all $n$ and all $d$. This is confirmed by Corollary \ref{cor:necessCond}, as all ideals in $\SI{d}{n}$ share the same (empty) hyperplane section, namely $\HS{d}{n} = \{(1)\}$.

Consequently, Proposition \ref{lem:lowerBound} (or \cite[Proposition 9]{BCR-GG}) can not be used to prove the non-irreducibility of punctual Hilbert schemes. This was already stated in \cite[Section 7.1]{BCR-GG}. but the argument was based on a previous results by Reeves saying that the set  $\SI{d}{n}$ consists of ideals defining points which all lie on a single irreducible component of $\Hilb{d}{n}$ \cite[Theorem 6]{Reeves}. We point out that our proof does not rely on that result. \bs
\end{remark}

\begin{example}[{Cf.~\cite[Example 8]{BCR-GG}}] Consider the Hilbert scheme $\Hilb{7t-5}{3}$ parametrizing 1-dimensional subschemes of $\PP^3$ of degree $7$ and arithmetic genus $6$. The computation of the Gr\"obner fan $\GF(\Hilb{7t-5}{3})$ is quite involved and the number of cones of maximal dimension is large (see Table \ref{tab:gfan cpuTime}{\sc\subref{tab:gfan cpuTime curves}}). Hence, we try to compute $\mathfrak{m}_{7t-5}^3$ looking for sets of irregular ideals with non-empty intersection of their maximality cones.

By Corollary \ref{cor:necessCond}, we know that we have to consider irregular ideals with different plane section. The set $\HS{7t-5}{3}$ is made of 4 ideals:
\[
H^\sat_1 = (x_3,x_2^7),\quad H^\sat_2 = (x_3^2,x_2x_3,x_2^6), \quad H^\sat_3 = (x_3^2,x_2^2 x_3,x_2^5),\quad H^\sat_4 = (x_3^2,x_2^3x_3,x_2^4). 
\]
The set $\SI{7t-5}{3,H_4}$ contains only the ideal $J_{112}\footnotemark[1] = (x_3^2,x_2^3x_3,x_2^4)_{\geqslant 16}$. Such ideal is irregular and its M-cone $\MC(J_{112})$ has extremal rays spanned by $(1,1,0,0)$, $(2,1,0,0)$, $(1,1,1,0)$ and $(1,1,1,1)$.

\footnotetext{The index labeling an ideal is the position of the ideal in the list of strongly stable ideals in $\kk[x_0,x_1,x_2,x_3]$ with Hilbert polynomial $p(t)=7t-5$  produced by the algorithm implemented in the \textit{Macaulay2} package \texttt{StronglyStableIdeals.m2} \cite{AL}.}

The set $\SI{7t-5}{3,H_3}$ contains 3 elements: $J_{109},J_{110},J_{111}\footnotemark[1]$.  The ideal $J_{110}$ can be discarded, because it is Borel adjacent to $J_{112}$. The other 2 ideals have distance 2 from $J_{112}$. The ideal $J_{109}$ is irregular but the maximality cone $\MC(J_{109})$ intersects $\MC(J_{112})$ only along the ray spanned by $(1,1,1,1)$. The ideal $J_{111} = (x_3^3,x_2 x_3^2,x_2^2 x_3, x_1 x_3^2,x_2^5)_{\geqslant 16}$ is irregular and its M-cone contains $\MC(J_{112})$, so that $\MC(J_{111},J_{112}) = \MC(J_{112})$.

The set $\SI{7t-5}{3,H_2}$ contains 14 elements (from $J_{95}$ to $J_{108}$\footnotemark[1]). Among the irregular ideals, there is only one ideal whose M-cone intersects $\MC(J_{111},J_{112})$ in a cone of maximal dimension. The maximality cone of the ideal $J_{108} = (x_3^3,x_2 x_3^2,x_2^2 x_3,x_1^2 x_3^2,$ $ x_1^2 x_2 x_3, x_2^6)_{\geqslant 16}$ has extremal rays spanned by $(1,1,0,0)$, $(1,1,1,0)$, $(3,2,1,0)$ and $(1,1,1,1)$. This cone is contained in $\MC(J_{111},J_{112})$, so that $\MC(J_{108},J_{111},J_{112}) = \MC(J_{108})$.

The set $\SI{7t-5}{3,H_1}$ contains the remaining 94 ideals (from $J_{1}$ to $J_{94}$\footnotemark[1]) and 44 of them are irregular. Only 2 ideals have M-cone intersecting $\MC(J_{108},J_{111},J_{112})$ in a cone of maximal dimension.
The ideal $J_{93} = (x_3^3,x_2 x_3^2, x_2^3 x_3, x_1 x_2^2 x_3, x_1^2 x_3^2, x_1^2 x_2 x_3, x_1^4 x_3, x_2^7)_{\geqslant 16}$ leads to the cone $\MC(J_{93},J_{108},J_{111},J_{112})$ with extremal rays spanned by $(2,2,1,0)$, $(3,2,1,0)$, $(3,3,2,0)$, $(5,4,3,0)$ and $(1,1,1,1)$. The ideal $J_{94} = (x_3^3,x_2^2 x_3^2, x_2^3 x_3, x_1 x_2 x_3^2, x_1^2 x_3^2, x_1^2 x_2 x_3, x_1^3 x_3, x_2^7)_{\geqslant 16}$ leads to the cone $\MC(J_{94},$ $J_{108},J_{111},J_{112})$ with extremal rays spanned by $(1,1,1,0)$, $(3,3,2,0)$, $(5,4,3,0)$ and $(1,1,1,1)$.

In both cases, a term order $\Omega$ obtained from $\geq_\ooo$ and ties broken by $\mathtt{DegLex}$, where $\ooo$ is a vector in the interior of the maximality cones $\MC(J_{93},J_{108},J_{111},J_{112})$ and $\MC(J_{94},J_{108},J_{111},J_{112})$ makes the 4 ideals maximal elements with respect to $\ssucceq_{\Omega}$. Finally, $\mathfrak{m}_{7t-5}^3 = 4$. The same result has been showed in Example 8 of \cite{BCR-GG} using the graded reverse lexicographic order. The cone of $\GF(\Hilb{7t-5}{3})$ corresponding to $\mathtt{RevLex}$ is contained in $\MC(J_{94},J_{108},J_{111},J_{112})$.
\end{example}

\captionsetup[subfloat]{width=0.9\textwidth}
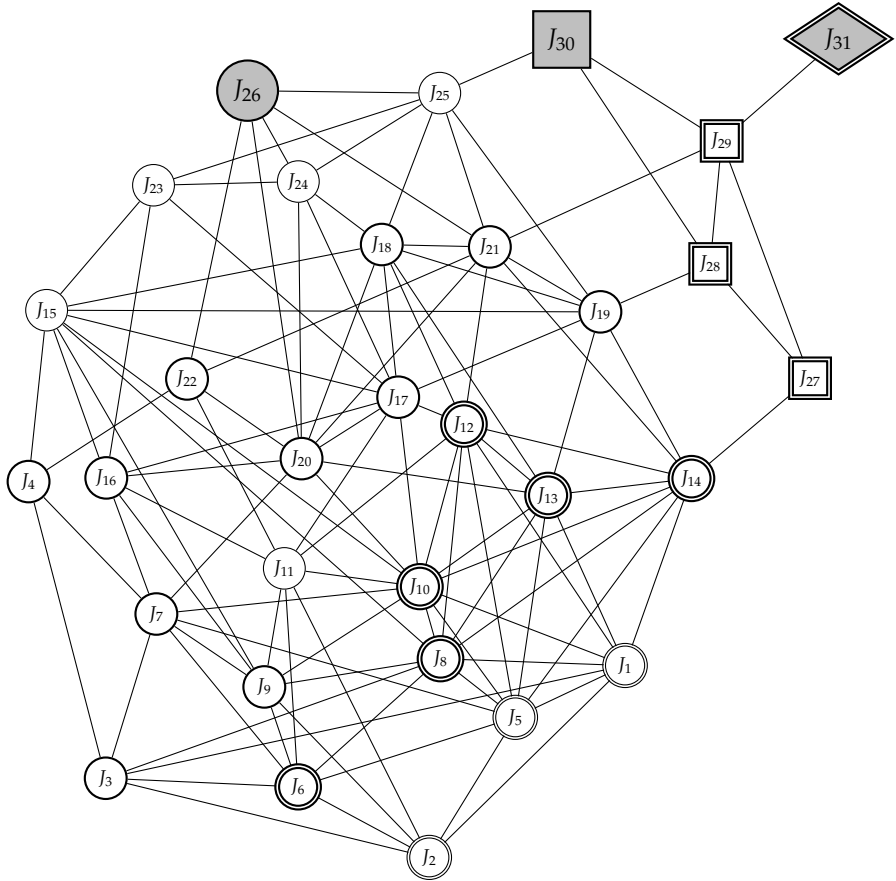
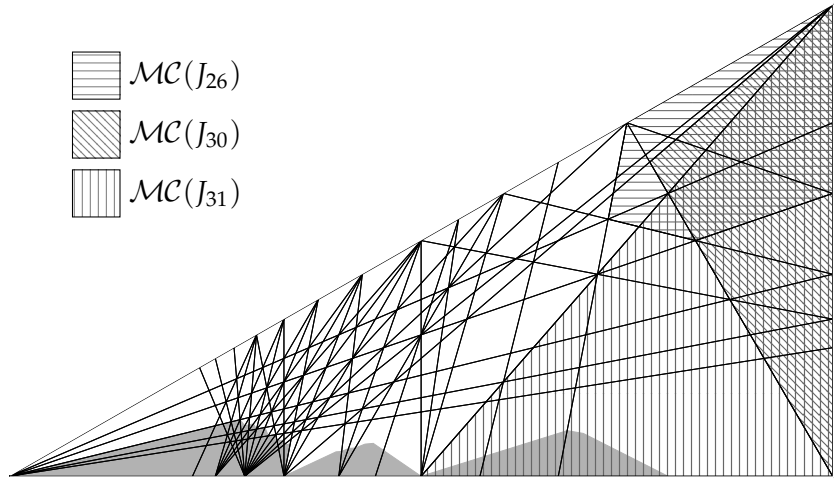
\begin{figure}
\begin{center}
\subfloat[][The Borel graph $\Sk{6t-3}{3}$. Regular ideals are drawn with a thin line, while irregular ideals are drawn with a thick line. Hilb-segment ideals are marked with a double line. Circle shaped vertices correspond to ideals with plane section $(x_3,x_2^6)$, squared shaped vertices correspond to ideals with plane section $(x_3^2,x_2x_3,x_2^5)$ and the diamond shaped vertex correspond to the ideal with plane section $(x_3^2,x_2^2 x_3, x_2^4)$. The gray bigger vertices are the maximal elements discussed in Example \ref{ex:6t-3 part 2}.]{\label{fig:6t-3 bg}
\begin{tikzpicture}[yscale=0.925,xscale=0.95]
\node (0) at (243.81bp,75.988bp) [double,circle,draw,inner sep=0pt,minimum size=0.55cm,thin] {\tiny $J_{1}$};
  \node (1) at (166.98bp,-2.0bp) [double,draw,circle,inner sep=0pt,minimum size=0.55cm,thin] {\tiny $J_{2}$};
  \node (2) at (40.01bp,30.169bp) [draw,circle,inner sep=0pt,minimum size=0.55cm,thick] {\tiny $J_{3}$};
  \node (4) at (200.73bp,54.911bp) [double,draw,circle,inner sep=0pt,minimum size=0.55cm,thin] {\tiny $J_{5}$};
  \node (7) at (171.38bp,78.734bp) [double,draw,circle,inner sep=0pt,minimum size=0.55cm,thick] {\tiny $J_{8}$};
  \node (9) at (163.31bp,108.03bp) [double,draw,circle,inner sep=0pt,minimum size=0.55cm,thick] {\tiny $J_{10}$};
  \node (11) at (180.58bp,174.05bp) [double,draw,circle,inner sep=0pt,minimum size=0.55cm,thick] {\tiny $J_{12}$};
  \node (12) at (213.59bp,145.13bp) [double,draw,circle,inner sep=0pt,minimum size=0.55cm,thick] {\tiny $J_{13}$};
  \node (13) at (269.87bp,151.96bp) [double,draw,circle,inner sep=0pt,minimum size=0.55cm,thick] {\tiny $J_{14}$};
  \node (5) at (115.51bp,26.658bp) [double,draw,circle,inner sep=0pt,minimum size=0.55cm,thick] {\tiny $J_{6}$};
  \node (8) at (102.24bp,67.28bp) [draw,circle,inner sep=0pt,minimum size=0.55cm,thick] {\tiny $J_{9}$};
  \node (10) at (110.1bp,115.46bp) [draw,circle,inner sep=0pt,minimum size=0.55cm,thin] {\tiny $J_{11}$};
  \node (3) at (9.725bp,150.72bp) [draw,circle,inner sep=0pt,minimum size=0.55cm,thick] {\tiny $J_{4}$};
  \node (6) at (59.879bp,96.907bp) [draw,circle,inner sep=0pt,minimum size=0.55cm,thick] {\tiny $J_{7}$};
  \node (14) at (16.81bp,220.36bp) [draw,circle,inner sep=0pt,minimum size=0.55cm,thin] {\tiny $J_{15}$};
  \node (21) at (71.915bp,192.43bp) [draw,circle,inner sep=0pt,minimum size=0.55cm,thick] {\tiny $J_{22}$};
  \node (15) at (40.227bp,152.22bp) [draw,circle,inner sep=0pt,minimum size=0.55cm,thick] {\tiny $J_{16}$};
  \node (19) at (116.85bp,160.05bp) [draw,circle,inner sep=0pt,minimum size=0.55cm,thick] {\tiny $J_{20}$};
  \node (16) at (154.78bp,185.01bp) [draw,circle,inner sep=0pt,minimum size=0.55cm,thick] {\tiny $J_{17}$};
  \node (17) at (148.34bp,247.05bp) [draw,circle,inner sep=0pt,minimum size=0.55cm,thick] {\tiny $J_{18}$};
  \node (20) at (190.75bp,246.09bp) [draw,circle,inner sep=0pt,minimum size=0.55cm,thick] {\tiny $J_{21}$};
  \node (18) at (234.13bp,219.7bp) [draw,circle,inner sep=0pt,minimum size=0.55cm,thick] {\tiny $J_{19}$};
  \node (26) at (316.4bp,192.72bp) [double,draw,rectangle,inner sep=0pt,minimum size=0.5cm,thick] {\tiny $J_{27}$};
  \node (22) at (58.762bp,271.14bp) [draw,circle,inner sep=0pt,minimum size=0.55cm,thin] {\tiny $J_{23}$};
  \node (23) at (115.58bp,272.65bp) [draw,circle,inner sep=0pt,minimum size=0.55cm,thin] {\tiny $J_{24}$};
  \node (24) at (171.08bp,308.59bp) [draw,circle,inner sep=0pt,minimum size=0.55cm,thin] {\tiny $J_{25}$};
  \node (27) at (277.24bp,239.19bp) [double,draw,rectangle,inner sep=0pt,minimum size=0.5cm,thick] {\tiny $J_{28}$};
  \node (25) at (95.676bp,309.58bp) [draw,circle,inner sep=0pt,minimum size=0.8cm,thick,fill=black!25] {\footnotesize $J_{26}$};
  \node (28) at (281.73bp,289.15bp) [double,draw,rectangle,inner sep=0pt,minimum size=0.5cm,thick] {\tiny $J_{29}$};
  \node (29) at (218.95bp,330.37bp) [draw,circle,rectangle,inner sep=0pt,minimum size=0.75cm,thick,fill=black!25] {\footnotesize $J_{30}$};
  \node (30) at (327.6bp,330.63bp) [double,draw,diamond,shape aspect=1.5,inner sep=3pt,thick,fill=black!25] {\footnotesize $J_{31}$};
  \draw [] (0) -- (1);
  \draw [] (0) -- (2);
  \draw [] (0) -- (4);
  \draw [] (0) -- (7);
  \draw [] (0) -- (9);
  \draw [] (0) -- (11);
  \draw [] (0) -- (12);
  \draw [] (0) -- (13);
  \draw [] (1) -- (2);
  \draw [] (1) -- (4);
  \draw [] (1) -- (5);
  \draw [] (1) -- (8);
  \draw [] (1) -- (10);
  \draw [] (2) -- (3);
  \draw [] (2) -- (5);
  \draw [] (2) -- (6);
  \draw [] (2) -- (7);
  \draw [] (3) -- (6);
  \draw [] (3) -- (14);
  \draw [] (3) -- (21);
  \draw [] (4) -- (5);
  \draw [] (4) -- (6);
  \draw [] (4) -- (7);
  \draw [] (4) -- (9);
  \draw [] (4) -- (11);
  \draw [] (4) -- (12);
  \draw [] (4) -- (13);
  \draw [] (5) -- (6);
  \draw [] (5) -- (7);
  \draw [] (5) -- (8);
  \draw [] (5) -- (10);
  \draw [] (6) -- (8);
  \draw [] (6) -- (9);
  \draw [] (6) -- (15);
  \draw [] (6) -- (19);
  \draw [] (7) -- (8);
  \draw [] (7) -- (9);
  \draw [] (7) -- (11);
  \draw [] (7) -- (12);
  \draw [] (7) -- (13);
  \draw [] (7) -- (14);
  \draw [] (8) -- (9);
  \draw [] (8) -- (10);
  \draw [] (8) -- (14);
  \draw [] (8) -- (15);
  \draw [] (9) -- (10);
  \draw [] (9) -- (11);
  \draw [] (9) -- (12);
  \draw [] (9) -- (13);
  \draw [] (9) -- (14);
  \draw [] (9) -- (16);
  \draw [] (9) -- (19);
  \draw [] (10) -- (11);
  \draw [] (10) -- (15);
  \draw [] (10) -- (16);
  \draw [] (10) -- (21);
  \draw [] (11) -- (12);
  \draw [] (11) -- (13);
  \draw [] (11) -- (16);
  \draw [] (11) -- (17);
  \draw [] (11) -- (20);
  \draw [] (12) -- (13);
  \draw [] (12) -- (17);
  \draw [] (12) -- (18);
  \draw [] (12) -- (19);
  \draw [] (13) -- (18);
  \draw [] (13) -- (20);
  \draw [] (13) -- (26);
  \draw [] (14) -- (15);
  \draw [] (14) -- (16);
  \draw [] (14) -- (17);
  \draw [] (14) -- (18);
  \draw [] (14) -- (22);
  \draw [] (15) -- (16);
  \draw [] (15) -- (19);
  \draw [] (15) -- (22);
  \draw [] (16) -- (17);
  \draw [] (16) -- (18);
  \draw [] (16) -- (19);
  \draw [] (16) -- (22);
  \draw [] (16) -- (23);
  \draw [] (17) -- (18);
  \draw [] (17) -- (19);
  \draw [] (17) -- (20);
  \draw [] (17) -- (23);
  \draw [] (17) -- (24);
  \draw [] (18) -- (20);
  \draw [] (18) -- (24);
  \draw [] (18) -- (27);
  \draw [] (19) -- (20);
  \draw [] (19) -- (21);
  \draw [] (19) -- (23);
  \draw [] (19) -- (25);
  \draw [] (20) -- (21);
  \draw [] (20) -- (24);
  \draw [] (20) -- (25);
  \draw [] (20) -- (28);
  \draw [] (21) -- (25);
  \draw [] (22) -- (23);
  \draw [] (22) -- (24);
  \draw [] (23) -- (24);
  \draw [] (23) -- (25);
  \draw [] (24) -- (25);
  \draw [] (24) -- (29);
  \draw [] (26) -- (27);
  \draw [] (26) -- (28);
  \draw [] (27) -- (28);
  \draw [] (27) -- (29);
  \draw [] (28) -- (29);
  \draw [] (28) -- (30);
\end{tikzpicture}
}

\subfloat[][The Gr\"obner fan $\GF(\Hilb{6t-3}{3})$. The maximality cones of $J_{26}$, $J_{30}$ and $J_{31}$ are highlighted with horizontal, oblique and vertical lines. The gray area corresponds to the union of all segment cones.]{ \label{fig:6t-3 gf}
\begin{tikzpicture}[scale=12.5]
\begin{scope}[scale=0.025,shift={(-32,-4)}]
\draw [black,pattern color=black!50,pattern=horizontal lines] (0,0) rectangle (2,2);
\draw [black,pattern color=black!50,pattern=north west lines] (0,-2.5) rectangle (2,-0.5);
\draw [black,pattern color=black!50,pattern=vertical lines] (0,-5) rectangle (2,-3);
\node at (4.75,1) [] {$\MC(J_{26})$};
\node at (4.75,-1.5) [] {$\MC(J_{30})$};
\node at (4.75,-4) [] {$\MC(J_{31})$};
\end{scope}

\draw [draw=black!30,fill=black!30,ultra thin] (-.866,-.5) -- (-.618571,-.5) -- (-.633659,-.463415) -- cycle;
\draw [draw=black!30,fill=black!30,ultra thin] (-.866,-.5) -- (-.632846,-.442308) -- (-.629818,-.454545) -- (-.633659,-.463415) -- cycle;
\draw [draw=black!30,fill=black!30,ultra thin] (-.618571,-.5) -- (-.622438,-.453125) -- (-.629818,-.454545) -- (-.633659,-.463415) -- cycle;
\draw [draw=black!30,fill=black!30,ultra thin] (-.62352,-.44) -- (-.632846,-.442308) -- (-.629818,-.454545) -- (-.622438,-.453125)  -- (-.618571,-.446429) -- cycle;
\draw [draw=black!30,fill=black!30,ultra thin] (-.618571,-.5) -- (-.6062,-.45) -- (-.609407,-.444444)   -- (-.618571,-.446429)-- (-.622438,-.453125) -- cycle;
\draw [draw=black!30,fill=black!30,ultra thin] (-.6062,-.45)  -- (-.618571,-.5)  -- (-.593829,-.457143)-- (-.597241,-.448276) --  cycle;
\draw [draw=black!30,fill=black!30,ultra thin] (-.618571,-.5) -- (-.58455,-.4625) -- (-.585824,-.455882) -- (-.593829,-.457143)  -- cycle;
\draw [draw=black!30,fill=black!30,ultra thin] (-.618571,-.5) -- (-.577333,-.466667) -- (-.577333,-.461538) -- (-.58455,-.4625)  -- cycle;
\draw [draw=black!30,fill=black!30,ultra thin] (-.577333,-.5) -- (-.618571,-.5) -- (-.577333,-.466667) -- (-.570773,-.465909) -- cycle;
\draw [draw=black!30,fill=black!30,ultra thin] (-.5196,-.5) -- (-.577333,-.5)  -- (-.522586,-.474138)  -- (-.516561,-.473684)-- cycle;
\draw [draw=black!30,fill=black!30,ultra thin] (-.5196,-.5) -- (-.50228,-.47) -- (-.509412,-.470588) -- (-.516561,-.473684) -- cycle;
\draw [draw=black!30,fill=black!30,ultra thin] (-.433,-.5) -- (-.483349,-.465116) -- (-.492045,-.465909) -- (-.50228,-.47) -- (-.5196,-.5)  -- cycle;
\draw [draw=black!30,fill=black!30,ultra thin] (-.433,-.5) -- (-.1732,-.5) -- (-.262424,-.454545) -- (-.279355,-.451613)  -- cycle;

	\draw [black,pattern color=black!60,pattern=horizontal lines] (0,0)-- (0,-.2) --  (-.144333,-.25)   --  (-.236182,-.227273)-- (-.2165,-.125)  --cycle;
	\draw [black,pattern color=black!60,pattern=north west lines] (0,0)--(0,-.5)  -- (-.192444,-.166667) --cycle;
	\draw [black,pattern color=black!60,pattern=vertical lines] (0,0)--(0,-.5) --	(-.433,-.5) -- cycle;

              \draw [-,thin]    (-.866,-.5) -- (-.673556,-.5) -- (-.658913,-.467391) -- cycle; 
              \draw [-,thin] (-.673556,-.5) -- (-.6495,-.5) --  (-.63919,-.464286) -- (-.658913,-.467391) -- cycle; 
              \draw [-,thin]    (-.866,-.5) -- (-.655351,-.459459) -- (-.658913,-.467391) -- cycle; 
              \draw [-,thin]    (-.6495,-.5) -- (-.633659,-.463415) -- (-.63919,-.464286) -- cycle; 
              \draw [-,thin]  (-.655351,-.459459) -- (-.636765,-.455882) -- (-.63919,-.464286) -- (-.658913,-.467391) -- cycle; 
              \draw [-,thin]    (-.866,-.5) -- (-.6495,-.446429) -- (-.655351,-.459459) -- cycle; 
              \draw [-,thin]    (-.6495,-.5) -- (-.633659,-.463415) -- (-.631458,-.46875) -- cycle; 
              \draw [-,thin]    (-.636765,-.455882) -- (-.633659,-.463415) -- (-.63919,-.464286) -- cycle; 
              \draw [-,thin]  (-.6495,-.446429) --(-.641481,-.444444) --  (-.636765,-.455882) -- (-.655351,-.459459) -- cycle; 
              \draw [-,thin] (-.6495,-.425) -- (-.866,-.5) --  (-.6495,-.446429) -- (-.644894,-.43617) -- cycle; 
              \draw [-,thin]    (-.62785,-.4625) -- (-.633659,-.463415) -- (-.631458,-.46875) -- cycle; 
              \draw [-,thin]    (-.6495,-.5) -- (-.631458,-.46875) -- (-.629818,-.472727) -- cycle; 
              \draw [-,thin]    (-.629818,-.454545) -- (-.636765,-.455882) -- (-.633659,-.463415) -- cycle; 
              \draw [-,thin]    (-.632846,-.442308) -- (-.641481,-.444444) -- (-.636765,-.455882) -- cycle; 
              \draw [-,thin]    (-.641481,-.444444) -- (-.6495,-.446429) -- (-.644894,-.43617) -- cycle; 
              \draw [-,thin]    (-.638105,-.421053) -- (-.6495,-.425) -- (-.644894,-.43617) -- cycle; 
              \draw [-,thin]    (-.866,-.5) -- (-.6495,-.425) -- (-.656061,-.409091) -- cycle; 
              \draw [-,thin] (-.626468,-.468085) -- (-.62785,-.4625) --  (-.631458,-.46875) -- (-.629818,-.472727) -- cycle; 
              \draw [-,thin]    (-.629818,-.454545) -- (-.62785,-.4625) -- (-.633659,-.463415) -- cycle; 
              \draw [-,thin]    (-.6495,-.5) -- (-.629818,-.472727) -- (-.628548,-.475806) -- cycle; 
              \draw [-,thin]    (-.632846,-.442308) -- (-.629818,-.454545) -- (-.636765,-.455882) -- cycle; 
              \draw [-,thin]  (-.638105,-.421053) -- (-.632846,-.442308) -- (-.641481,-.444444) -- (-.644894,-.43617) -- cycle; 
              \draw [-,thin] (-.6495,-.425) -- (-.638105,-.421053) -- (-.642516,-.403226) -- (-.656061,-.409091) -- cycle; 
              \draw [-,thin]    (-.866,-.5) -- (-.666154,-.384615) -- (-.656061,-.409091) -- cycle; 
              \draw [-,thin]    (-.621744,-.461538) -- (-.62785,-.4625) -- (-.626468,-.468085) -- cycle; 
              \draw [-,thin] (-.625444,-.472222) -- (-.626468,-.468085) --  (-.629818,-.472727) -- (-.628548,-.475806) -- cycle; 
              \draw [-,thin]    (-.622438,-.453125) -- (-.629818,-.454545) -- (-.62785,-.4625) -- cycle; 
              \draw [-,thin]    (-.618571,-.5) -- (-.6495,-.5) -- (-.628548,-.475806) -- cycle; 
              \draw [-,thin]    (-.62352,-.44) -- (-.632846,-.442308) -- (-.629818,-.454545) -- cycle; 
              \draw [-,thin]    (-.625444,-.416667) -- (-.638105,-.421053) -- (-.632846,-.442308) -- cycle; 
              \draw [-,thin]    (-.638105,-.421053) -- (-.627103,-.396552) -- (-.642516,-.403226) -- cycle; 
              \draw [-,thin]  (-.666154,-.384615) --(-.6495,-.375) --  (-.642516,-.403226) -- (-.656061,-.409091) -- cycle; 
              \draw [-,thin]   (-.621261,-.467391) -- (-.621744,-.461538) -- (-.626468,-.468085) -- (-.625444,-.472222) -- cycle; 
              \draw [-,thin]    (-.622438,-.453125) -- (-.621744,-.461538) -- (-.62785,-.4625) -- cycle; 
              \draw [-,thin]    (-.618571,-.5) -- (-.625444,-.472222) -- (-.628548,-.475806) -- cycle; 
              \draw [-,thin]    (-.62352,-.44) -- (-.622438,-.453125) -- (-.629818,-.454545) -- cycle; 
              \draw [-,thin]    (-.625444,-.416667) -- (-.62352,-.44) -- (-.632846,-.442308) -- cycle; 
              \draw [-,thin]    (-.625444,-.416667) -- (-.638105,-.421053) -- (-.627103,-.396552) -- cycle; 
              \draw [-,thin]  (-.6495,-.375) --(-.629818,-.363636) --  (-.627103,-.396552) -- (-.642516,-.403226) -- cycle; 
              \draw [-,thin]    (-.615316,-.460526) -- (-.621744,-.461538) -- (-.621261,-.467391) -- cycle; 
              \draw [-,thin]    (-.618571,-.5) -- (-.621261,-.467391) -- (-.625444,-.472222) -- cycle; 
              \draw [-,thin]    (-.614581,-.451613) -- (-.622438,-.453125) -- (-.621744,-.461538) -- cycle; 
              \draw [-,thin]    (-.613417,-.4375) -- (-.62352,-.44) -- (-.622438,-.453125) -- cycle; 
              \draw [-,thin]    (-.611294,-.411765) -- (-.625444,-.416667) -- (-.62352,-.44) -- cycle; 
              \draw [-,thin]    (-.625444,-.416667) -- (-.618571,-.392857) -- (-.627103,-.396552) -- cycle; 
              \draw [-,thin]    (-.6062,-.35) -- (-.629818,-.363636) -- (-.627103,-.396552) -- cycle; 
              \draw [-,thin]    (-.614581,-.451613) -- (-.615316,-.460526) -- (-.621744,-.461538) -- cycle; 
              \draw [-,thin]    (-.618571,-.5) -- (-.615316,-.460526) -- (-.621261,-.467391) -- cycle; 
              \draw [-,thin]    (-.613417,-.4375) -- (-.614581,-.451613) -- (-.622438,-.453125) -- cycle; 
              \draw [-,thin]    (-.611294,-.411765) -- (-.613417,-.4375) -- (-.62352,-.44) -- cycle; 
              \draw [-,thin]  (-.625444,-.416667) -- (-.611294,-.411765) --  (-.609407,-.388889) -- (-.618571,-.392857) -- cycle; 
              \draw [-,thin]    (-.6062,-.35) -- (-.618571,-.392857) -- (-.627103,-.396552) -- cycle; 
              \draw [-,thin]    (-.6062,-.45) -- (-.614581,-.451613) -- (-.615316,-.460526) -- cycle; 
              \draw [-,thin]    (-.618571,-.5) -- (-.608541,-.459459) -- (-.615316,-.460526) -- cycle; 
              \draw [-,thin]    (-.602435,-.434783) -- (-.613417,-.4375) -- (-.614581,-.451613) -- cycle; 
              \draw [-,thin]    (-.595375,-.40625) -- (-.611294,-.411765) -- (-.613417,-.4375) -- cycle; 
              \draw [-,thin]    (-.611294,-.411765) -- (-.599538,-.384615) -- (-.609407,-.388889) -- cycle; 
              \draw [-,thin]    (-.6062,-.35) -- (-.609407,-.388889) -- (-.618571,-.392857) -- cycle; 
              \draw [-,thin]    (-.6062,-.45) -- (-.608541,-.459459) -- (-.615316,-.460526) -- cycle; 
              \draw [-,thin]    (-.602435,-.434783) -- (-.6062,-.45) -- (-.614581,-.451613) -- cycle; 
              \draw [-,thin]    (-.618571,-.5) -- (-.601389,-.458333) -- (-.608541,-.459459) -- cycle; 
              \draw [-,thin]    (-.595375,-.40625) -- (-.602435,-.434783) -- (-.613417,-.4375) -- cycle; 
              \draw [-,thin]    (-.595375,-.40625) -- (-.611294,-.411765) -- (-.599538,-.384615) -- cycle; 
              \draw [-,thin]    (-.6062,-.35) -- (-.599538,-.384615) -- (-.609407,-.388889) -- cycle; 
              \draw [-,thin]  (-.6062,-.45) --(-.597241,-.448276) --  (-.601389,-.458333) -- (-.608541,-.459459) -- cycle; 
              \draw [-,thin]    (-.590455,-.431818) -- (-.602435,-.434783) -- (-.6062,-.45) -- cycle; 
              \draw [-,thin]    (-.618571,-.5) -- (-.593829,-.457143) -- (-.601389,-.458333) -- cycle; 
              \draw [-,thin]    (-.595375,-.40625) -- (-.602435,-.434783) -- (-.592526,-.421053) -- cycle; 
              \draw [-,thin]    (-.595375,-.40625) -- (-.58888,-.38) -- (-.599538,-.384615) -- cycle; 
              \draw [-,thin] (-.577333,-.333333) -- (-.6062,-.35) -- (-.599538,-.384615) -- cycle; 
              \draw [-,thin]   (-.597241,-.448276) --(-.587643,-.446429) -- (-.593829,-.457143) -- (-.601389,-.458333) -- cycle; 
              \draw [-,thin]    (-.590455,-.431818) -- (-.597241,-.448276) -- (-.6062,-.45) -- cycle; 
              \draw [-,thin]    (-.590455,-.431818) -- (-.602435,-.434783) -- (-.592526,-.421053) -- cycle; 
              \draw [-,thin]    (-.618571,-.5) -- (-.585824,-.455882) -- (-.593829,-.457143) -- cycle; 
              \draw [-,thin]    (-.577333,-.4) -- (-.595375,-.40625) -- (-.592526,-.421053) -- cycle; 
              \draw [-,thin]    (-.595375,-.40625) -- (-.577333,-.375) -- (-.58888,-.38) -- cycle; 
              \draw [-,thin]    (-.577333,-.333333) -- (-.58888,-.38) -- (-.599538,-.384615) -- cycle; 
              \draw [-,thin]    (-.587643,-.446429) -- (-.585824,-.455882) -- (-.593829,-.457143) -- cycle; 
              \draw [-,thin]    (-.590455,-.431818) -- (-.587643,-.446429) -- (-.597241,-.448276) -- cycle; 
              \draw [-,thin]    (-.577333,-.4) -- (-.590455,-.431818) -- (-.592526,-.421053) -- cycle; 
              \draw [-,thin]    (-.618571,-.5) -- (-.585824,-.455882) -- (-.58455,-.4625) -- cycle; 
              \draw [-,thin]    (-.577333,-.4) -- (-.595375,-.40625) -- (-.577333,-.375) -- cycle; 
              \draw [-,thin]    (-.577333,-.333333) -- (-.577333,-.375) -- (-.58888,-.38) -- cycle; 
              \draw [-,thin]    (-.577333,-.444444) -- (-.587643,-.446429) -- (-.585824,-.455882) -- cycle; 
              \draw [-,thin]    (-.577333,-.428571) -- (-.590455,-.431818) -- (-.587643,-.446429) -- cycle; 
              \draw [-,thin]    (-.577333,-.4) -- (-.590455,-.431818) -- (-.577333,-.416667) -- cycle; 
              \draw [-,thin]  (-.577333,-.454545) -- (-.585824,-.455882) -- (-.58455,-.4625) -- cycle; 
              \draw [-,thin]  (-.618571,-.5) -- (-.58455,-.4625) -- (-.583609,-.467391) -- cycle; 
              \draw [-,thin]  (-.577333,-.4) -- (-.564783,-.369565) -- (-.577333,-.375) -- cycle; 
              \draw [-,thin]  (-.54125,-.3125) -- (-.577333,-.333333) -- (-.577333,-.375) -- cycle; 
              \draw [-,thin]  (-.577333,-.444444) -- (-.577333,-.454545) -- (-.585824,-.455882) -- cycle; 
              \draw [-,thin]  (-.577333,-.428571) -- (-.577333,-.444444) -- (-.587643,-.446429) -- cycle; 
              \draw [-,thin]  (-.577333,-.428571) -- (-.590455,-.431818) -- (-.577333,-.416667) -- cycle; 
              \draw [-,thin]  (-.556714,-.392857) -- (-.577333,-.4) -- (-.577333,-.416667) -- cycle; 
              \draw [-,thin]  (-.577333,-.461538) -- (-.577333,-.454545) -- (-.58455,-.4625) -- (-.583609,-.467391) -- cycle; 
              \draw [-,thin]  (-.618571,-.5) -- (-.583609,-.467391) -- (-.582885,-.471154) -- cycle; 
              \draw [-,thin]  (-.577333,-.4) -- (-.551091,-.363636) -- (-.564783,-.369565) -- cycle; 
              \draw [-,thin]  (-.54125,-.3125) -- (-.564783,-.369565) -- (-.577333,-.375) -- cycle; 
              \draw [-,thin]  (-.566231,-.442308) -- (-.577333,-.444444) -- (-.577333,-.454545) -- cycle; 
              \draw [-,thin]  (-.5629,-.425) -- (-.577333,-.428571) -- (-.577333,-.444444) -- cycle; 
              \draw [-,thin]  (-.556714,-.392857) -- (-.577333,-.428571) -- (-.577333,-.416667) -- cycle; 
              \draw [-,thin]  (-.556714,-.392857) -- (-.577333,-.4) -- (-.551091,-.363636) -- cycle; 
              \draw [-,thin]  (-.568312,-.453125) -- (-.577333,-.454545) -- (-.577333,-.461538) -- cycle; 
              \draw [-,thin] (-.577333,-.466667) --  (-.577333,-.461538) --  (-.583609,-.467391) -- (-.582885,-.471154) -- cycle; 
              \draw [-,thin]  (-.577333,-.5) -- (-.618571,-.5) -- (-.582885,-.471154) -- cycle; 
              \draw [-,thin]  (-.54125,-.3125) -- (-.551091,-.363636) -- (-.564783,-.369565) -- cycle; 
              \draw [-,thin]  (-.5629,-.425) -- (-.566231,-.442308) -- (-.577333,-.444444) -- cycle; 
              \draw [-,thin]  (-.566231,-.442308) -- (-.568312,-.453125) -- (-.577333,-.454545) -- cycle; 
              \draw [-,thin]  (-.556714,-.392857) -- (-.5629,-.425) -- (-.577333,-.428571) -- cycle; 
              \draw [-,thin]  (-.556714,-.392857) -- (-.536095,-.357143) -- (-.551091,-.363636) -- cycle; 
              \draw [-,thin]  (-.569737,-.460526) -- (-.568312,-.453125) --  (-.577333,-.461538) -- (-.577333,-.466667) -- cycle; 
              \draw [-,thin]  (-.577333,-.5) -- (-.577333,-.466667) -- (-.582885,-.471154) -- cycle; 
              \draw [-,thin]  (-.494857,-.285714) -- (-.54125,-.3125) -- (-.551091,-.363636) -- cycle; 
              \draw [-,thin]  (-.546947,-.421053) -- (-.5629,-.425) -- (-.566231,-.442308) -- cycle; 
              \draw [-,thin]  (-.55424,-.44) -- (-.566231,-.442308) -- (-.568312,-.453125) -- cycle; 
              \draw [-,thin]  (-.532923,-.384615) -- (-.556714,-.392857) -- (-.5629,-.425) -- cycle; 
              \draw [-,thin]  (-.556714,-.392857) -- (-.5196,-.35) -- (-.536095,-.357143) -- cycle; 
              \draw [-,thin]  (-.494857,-.285714) -- (-.536095,-.357143) -- (-.551091,-.363636) -- cycle; 
              \draw [-,thin]  (-.577333,-.5) -- (-.569737,-.460526) -- (-.577333,-.466667) -- cycle; 
              \draw [-,thin]  (-.55871,-.451613) -- (-.568312,-.453125) -- (-.569737,-.460526) -- cycle; 
              \draw [-,thin]  (-.546947,-.421053) -- (-.55424,-.44) -- (-.566231,-.442308) -- cycle; 
              \draw [-,thin]  (-.532923,-.384615) -- (-.546947,-.421053) -- (-.5629,-.425) -- cycle; 
              \draw [-,thin]  (-.55424,-.44) -- (-.55871,-.451613) -- (-.568312,-.453125) -- cycle; 
              \draw [-,thin]  (-.532923,-.384615) -- (-.556714,-.392857) -- (-.5196,-.35) -- cycle; 
              \draw [-,thin]  (-.494857,-.285714) -- (-.5196,-.35) -- (-.536095,-.357143) -- cycle; 
              \draw [-,thin]  (-.577333,-.5) -- (-.55871,-.451613) -- (-.569737,-.460526) -- cycle; 
              \draw [-,thin]  (-.529222,-.416667) -- (-.546947,-.421053) -- (-.55424,-.44) -- cycle; 
              \draw [-,thin]  (-.505167,-.375) -- (-.532923,-.384615) -- (-.546947,-.421053) -- cycle; 
              \draw [-,thin]  (-.54125,-.4375) -- (-.55424,-.44) -- (-.55871,-.451613) -- cycle; 
              \draw [-,thin]  (-.532923,-.384615) -- (-.501368,-.342105) -- (-.5196,-.35) -- cycle; 
              \draw [-,thin]  (-.494857,-.285714) -- (-.5196,-.35) -- (-.499615,-.326923) -- cycle; 
              \draw [-,thin]  (-.577333,-.5) -- (-.548467,-.45) -- (-.55871,-.451613) -- cycle; 
              \draw [-,thin]  (-.505167,-.375) -- (-.529222,-.416667) -- (-.546947,-.421053) -- cycle; 
              \draw [-,thin]  (-.529222,-.416667) -- (-.54125,-.4375) -- (-.55424,-.44) -- cycle; 
              \draw [-,thin]  (-.505167,-.375) -- (-.532923,-.384615) -- (-.501368,-.342105) -- cycle; 
              \draw [-,thin]  (-.54125,-.4375) -- (-.548467,-.45) -- (-.55871,-.451613) -- cycle; 
              \draw [-,thin]  (-.501368,-.342105) -- (-.5196,-.35) -- (-.499615,-.326923) -- cycle; 
              \draw [-,thin]  (-.433,-.25) -- (-.494857,-.285714) -- (-.499615,-.326923) -- cycle; 
              \draw [-,thin]  (-.577333,-.5) -- (-.513185,-.444444) -- (-.548467,-.45) -- cycle; 
              \draw [-,thin]  (-.505167,-.375) -- (-.529222,-.416667) -- (-.507655,-.396552) -- cycle; 
              \draw [-,thin]  (-.509412,-.411765) -- (-.529222,-.416667) -- (-.54125,-.4375) -- cycle; 
              \draw [-,thin]  (-.505167,-.375) -- (-.481111,-.333333) -- (-.501368,-.342105) -- cycle; 
              \draw [-,thin]  (-.54125,-.4375) -- (-.511727,-.431818) --  (-.513185,-.444444) -- (-.548467,-.45) -- cycle; 
              \draw [-,thin]  (-.433,-.25) -- (-.501368,-.342105) -- (-.499615,-.326923) -- cycle; 
              \draw [-,thin]  (-.5196,-.5) -- (-.577333,-.5) -- (-.513185,-.444444) -- cycle; 
              \draw [-,thin]  (-.472364,-.363636) -- (-.505167,-.375) -- (-.507655,-.396552) -- cycle; 
              \draw [-,thin]  (-.509412,-.411765) -- (-.529222,-.416667) -- (-.507655,-.396552) -- cycle; 
              \draw [-,thin]  (-.509412,-.411765) -- (-.511727,-.431818) -- (-.54125,-.4375) -- cycle; 
              \draw [-,thin]  (-.505167,-.375) -- (-.458471,-.323529) -- (-.481111,-.333333) -- cycle; 
              \draw [-,thin]  (-.433,-.25) -- (-.481111,-.333333) -- (-.501368,-.342105) -- cycle; 
              \draw [-,thin]  (-.494857,-.428571) -- (-.511727,-.431818) -- (-.513185,-.444444) -- cycle; 
              \draw [-,thin]  (-.5196,-.5) -- (-.499615,-.442308) -- (-.513185,-.444444) -- cycle; 
              \draw [-,thin]  (-.472364,-.363636) --  (-.507655,-.396552) -- (-.509412,-.411765)  -- (-.481111,-.388889) -- cycle; 
              \draw [-,thin]  (-.472364,-.363636) -- (-.505167,-.375) -- (-.458471,-.323529) -- cycle; 
              \draw [-,thin]  (-.509412,-.411765) -- (-.487125,-.40625) --  (-.494857,-.428571) -- (-.511727,-.431818) -- cycle; 
              \draw [-,thin]  (-.433,-.25) -- (-.458471,-.323529) -- (-.481111,-.333333) -- cycle; 
              \draw [-,thin]  (-.494857,-.428571) -- (-.499615,-.442308) -- (-.513185,-.444444) -- cycle; 
              \draw [-,thin]  (-.5196,-.5) -- (-.48496,-.44) -- (-.499615,-.442308) -- cycle; 
              \draw [-,thin]  (-.487125,-.40625) -- (-.509412,-.411765) -- (-.481111,-.388889) -- cycle; 
              \draw [-,thin]  (-.433,-.35) -- (-.472364,-.363636) -- (-.481111,-.388889) -- cycle; 
              \draw [-,thin]  (-.472364,-.363636)-- (-.458471,-.323529) -- (-.433,-.3125)  -- (-.433,-.326923) -- cycle; 
              \draw [-,thin]  (-.461867,-.4) -- (-.487125,-.40625) -- (-.494857,-.428571) -- cycle; 
              \draw [-,thin]  (-.433,-.25) -- (-.458471,-.323529) -- (-.433,-.295455) -- cycle; 
              \draw [-,thin]   (-.494857,-.428571) --(-.4763,-.425) -- (-.48496,-.44) -- (-.499615,-.442308) -- cycle; 
              \draw [-,thin]  (-.5196,-.5) -- (-.481111,-.5) -- (-.460638,-.43617)  -- (-.48496,-.44) -- cycle; 
              \draw [-,thin]  (-.433,-.35) -- (-.461867,-.4) -- (-.487125,-.40625) -- (-.481111,-.388889) -- cycle; 
              \draw [-,thin]  (-.433,-.35) -- (-.472364,-.363636) -- (-.433,-.326923) -- cycle; 
              \draw [-,thin] (-.433,-.326923) -- (-.433,-.3125) --  (-.419032,-.306452) -- (-.422439,-.317073) -- cycle; 
              \draw [-,thin]  (-.433,-.3125) -- (-.458471,-.323529) -- (-.433,-.295455) -- cycle; 
              \draw [-,thin]  (-.461867,-.4) -- (-.4763,-.425) -- (-.494857,-.428571) -- cycle; 
              \draw [-,thin]  (-.433,-.25) -- (-.433,-.295455) -- (-.405938,-.265625) -- (-.402791,-.255814) -- cycle; 
              \draw [-,thin]  (-.455789,-.421053) -- (-.4763,-.425) -- (-.48496,-.44) -- (-.460638,-.43617) -- cycle; 
              \draw [-,thin]  (-.481111,-.5) -- (-.433,-.5)   -- (-.433,-.431818)  -- (-.460638,-.43617) -- cycle; 
              \draw [-,thin]  (-.433,-.35) -- (-.461867,-.4) -- (-.444103,-.384615) -- cycle; 
              \draw [-,thin]  (-.433,-.35) -- (-.433,-.326923) -- (-.422439,-.317073) -- cycle; 
              \draw [-,thin]  (-.404133,-.3) -- (-.419032,-.306452) -- (-.422439,-.317073) -- cycle; 
              \draw [-,thin]  (-.433,-.3125) -- (-.433,-.295455)  -- (-.405938,-.265625) -- (-.419032,-.306452) -- cycle; 
              \draw [-,thin] (-.4763,-.425) -- (-.461867,-.4)  -- (-.447931,-.396552) -- (-.455789,-.421053) --  cycle; 
              \draw [-,thin]  (-.405938,-.265625) -- (-.397892,-.256757) -- (-.402791,-.255814) -- cycle; 
              \draw [-,thin]  (-.433,-.25) -- (-.393636,-.227273) -- (-.402791,-.255814) -- cycle; 
              \draw [-,thin] (-.455789,-.421053)  -- (-.433,-.416667) -- (-.433,-.431818) -- (-.460638,-.43617) -- cycle; 
              \draw [-,thin]  (-.433,-.5) -- (-.433,-.431818) -- (-.42293,-.430233) -- cycle; 
              \draw [-,thin]  (-.461867,-.4) -- (-.447931,-.396552) -- (-.444103,-.384615) -- cycle; 
              \draw [-,thin]  (-.433,-.35) -- (-.433,-.375) -- (-.444103,-.384615) -- cycle; 
              \draw [-,thin]  (-.433,-.35) -- (-.404133,-.3) -- (-.422439,-.317073) -- cycle; 
              \draw [-,thin]  (-.404133,-.3) -- (-.419032,-.306452) -- (-.405938,-.265625) -- (-.397892,-.256757) -- cycle; 
              \draw [-,thin]  (-.433,-.392857) -- (-.433,-.416667) -- (-.455789,-.421053) -- (-.447931,-.396552) -- cycle; 
              \draw [-,thin]  (-.393636,-.227273) -- (-.397892,-.256757) -- (-.402791,-.255814) -- cycle; 
              \draw [-,thin]  (-.433,-.416667) -- (-.433,-.431818)  -- (-.42293,-.430233) -- (-.420629,-.414286) -- cycle; 
              \draw [-,thin]  (-.433,-.5) -- (-.412381,-.428571) -- (-.42293,-.430233) -- cycle; 
              \draw [-,thin]  (-.433,-.392857) -- (-.433,-.375) -- (-.444103,-.384615) -- (-.447931,-.396552)  -- cycle; 
              \draw [-,thin]  (-.433,-.35)  -- (-.433,-.375) -- (-.412381,-.357143) -- (-.410211,-.342105)-- cycle; 
              \draw [-,thin]  (-.433,-.35) -- (-.404133,-.3) -- (-.408491,-.330189) -- cycle; 
              \draw [-,thin]  (-.404133,-.3) -- (-.38104,-.26) -- (-.397892,-.256757) -- cycle; 
              \draw [-,thin]  (-.433,-.392857) -- (-.433,-.416667) -- (-.420629,-.414286)  -- (-.416963,-.388889) -- cycle; 
              \draw [-,thin]  (-.3464,-.2) -- (-.393636,-.227273) -- (-.397892,-.256757) -- cycle; 
              \draw [-,thin]  (-.407529,-.411765) -- (-.412381,-.428571)  -- (-.42293,-.430233) -- (-.420629,-.414286) -- cycle; 
              \draw [-,thin]  (-.433,-.5) -- (-.364632,-.421053) -- (-.412381,-.428571) -- cycle; 
              \draw [-,thin]  (-.433,-.392857) -- (-.433,-.375)  -- (-.412381,-.357143) -- (-.416963,-.388889) -- cycle; 
              \draw [-,thin]  (-.384889,-.333333) -- (-.410211,-.342105) -- (-.412381,-.357143) -- cycle; 
              \draw [-,thin]  (-.433,-.35) -- (-.410211,-.342105) -- (-.408491,-.330189) -- cycle; 
              \draw [-,thin]  (-.371143,-.285714) -- (-.404133,-.3)  -- (-.408491,-.330189)-- (-.376522,-.304348) -- cycle; 
              \draw [-,thin]  (-.3464,-.2) -- (-.38104,-.26) -- (-.397892,-.256757) -- cycle; 
              \draw [-,thin]  (-.404133,-.3) -- (-.364632,-.263158) -- (-.38104,-.26) -- cycle; 
              \draw [-,thin]  (-.399692,-.384615) -- (-.407529,-.411765)  -- (-.420629,-.414286) -- (-.416963,-.388889) -- cycle; 
              \draw [-,thin]  (-.3464,-.4) -- (-.364632,-.421053) -- (-.412381,-.428571) -- (-.407529,-.411765) -- cycle; 
              \draw [-,thin]  (-.433,-.5) -- (-.371143,-.5) -- (-.351081,-.418919) -- (-.364632,-.421053) -- cycle; 
              \draw [-,thin]  (-.384889,-.333333) -- (-.399692,-.384615) -- (-.416963,-.388889) -- (-.412381,-.357143) -- cycle; 
              \draw [-,thin]  (-.384889,-.333333) -- (-.376522,-.304348)  -- (-.408491,-.330189) -- (-.410211,-.342105) --  cycle; 
              \draw [-,thin]  (-.333077,-.269231) -- (-.371143,-.285714) -- (-.376522,-.304348) -- cycle; 
              \draw [-,thin]  (-.371143,-.285714) -- (-.404133,-.3) -- (-.364632,-.263158) -- cycle; 
              \draw [-,thin]  (-.3464,-.2) -- (-.364632,-.263158) -- (-.38104,-.26) -- cycle; 
              \draw [-,thin]  (-.399692,-.384615)  -- (-.407529,-.411765)-- (-.3464,-.4) -- (-.33887,-.369565) -- cycle; 
              \draw [-,thin]  (-.3464,-.4) -- (-.364632,-.421053) -- (-.351081,-.418919) -- cycle; 
              \draw [-,thin]  (-.288667,-.5) -- (-.270625,-.40625)  -- (-.351081,-.418919) -- (-.371143,-.5) -- cycle; 
              \draw [-,thin]  (-.32475,-.3125) -- (-.384889,-.333333) -- (-.399692,-.384615) -- (-.33887,-.369565) -- cycle; 
              \draw [-,thin]  (-.384889,-.333333) -- (-.376522,-.304348) -- (-.333077,-.269231) -- (-.314909,-.272727) -- cycle; 
              \draw [-,thin]  (-.333077,-.269231) -- (-.371143,-.285714) -- (-.364632,-.263158) -- cycle; 
              \draw [-,thin]  (-.3464,-.2) -- (-.364632,-.263158) -- (-.30729,-.209677) -- cycle; 
              \draw [-,thin]  (-.314909,-.363636) -- (-.3464,-.4) -- (-.33887,-.369565) -- cycle; 
              \draw [-,thin]  (-.266462,-.384615) -- (-.270625,-.40625) -- (-.351081,-.418919) -- (-.3464,-.4) -- cycle; 
              \draw [-,thin]  (0,-.5) -- (-.288667,-.5) --  (-.270625,-.40625) -- (-.0721667,-.375) -- cycle; 
              \draw [-,thin]  (-.247429,-.285714) -- (-.32475,-.3125)  -- (-.33887,-.369565)-- (-.314909,-.363636) -- cycle; 
              \draw [-,thin]  (-.32475,-.3125) -- (-.384889,-.333333) -- (-.314909,-.272727) -- cycle; 
              \draw [-,thin]  (-.333077,-.269231) -- (-.31176,-.26) -- (-.314909,-.272727) -- cycle; 
              \draw [-,thin]  (-.333077,-.269231) -- (-.364632,-.263158)-- (-.30729,-.209677) -- (-.299769,-.211538)  -- (-.309286,-.25) -- cycle; 
              \draw [-,thin]  (-.3464,-.2) -- (-.288667,-.166667) --  (-.296914,-.2) -- (-.30729,-.209677) -- cycle; 
              \draw [-,thin]  (-.2598,-.35) -- (-.314909,-.363636)  -- (-.3464,-.4)-- (-.266462,-.384615) -- cycle; 
              \draw [-,thin]  (-.0866,-.35) -- (-.266462,-.384615)  -- (-.270625,-.40625)-- (-.0721667,-.375) -- cycle; 
              \draw [-,thin]  (0,-.5) -- (0,-.363636) -- (-.0721667,-.375) -- cycle; 
              \draw [-,thin]  (-.247429,-.285714) -- (-.2598,-.35) -- (-.314909,-.363636) -- cycle; 
              \draw [-,thin]  (-.247429,-.285714) -- (-.32475,-.3125) -- (-.314909,-.272727) -- cycle; 
              \draw [-,thin]  (-.288667,-.25) -- (-.31176,-.26) -- (-.314909,-.272727) -- cycle; 
              \draw [-,thin]  (-.333077,-.269231) -- (-.31176,-.26) -- (-.309286,-.25) -- cycle; 
              \draw [-,thin]  (-.270625,-.21875) -- (-.299769,-.211538) -- (-.309286,-.25) -- cycle; 
              \draw [-,thin]  (-.299769,-.211538) -- (-.30729,-.209677) -- (-.296914,-.2) -- cycle; 
              \draw [-,thin]  (-.2165,-.125) -- (-.288667,-.166667) -- (-.296914,-.2) -- cycle; 
              \draw [-,thin]  (-.10825,-.3125) -- (-.0866,-.35)  -- (-.266462,-.384615) -- (-.2598,-.35)  -- cycle; 
              \draw [-,thin]  (0,-.333333)  -- (0,-.363636) -- (-.0721667,-.375) -- (-.0866,-.35) -- cycle; 
              \draw [-,thin]  (-.247429,-.285714) -- (-.10825,-.3125) -- (-.2598,-.35) -- cycle; 
              \draw [-,thin]  (-.247429,-.285714) -- (-.236182,-.227273) -- (-.288667,-.25) -- (-.314909,-.272727) -- cycle; 
              \draw [-,thin]   (-.31176,-.26)  -- (-.309286,-.25)  -- (-.270625,-.21875) -- (-.256593,-.222222) -- (-.288667,-.25)  -- cycle; 
              \draw [-,thin]  (-.2165,-.125)  -- (-.227895,-.184211)-- (-.270625,-.21875) -- (-.299769,-.211538) -- (-.296914,-.2) -- cycle; 
              \draw [-,thin]  (-.10825,-.3125) -- (0,-.333333) -- (-.0866,-.35) -- cycle; 
              \draw [-,thin]  (-.144333,-.25) -- (-.247429,-.285714) -- (-.10825,-.3125) -- cycle; 
              \draw [-,thin]  (-.247429,-.285714) -- (-.236182,-.227273) -- (-.203765,-.235294) -- cycle; 
              \draw [-,thin]  (-.236182,-.227273) -- (-.288667,-.25) -- (-.256593,-.222222) -- cycle; 
              \draw [-,thin]  (-.230933,-.2)   -- (-.256593,-.222222) -- (-.270625,-.21875) -- (-.227895,-.184211) -- cycle; 
              \draw [-,thin]  (-.2165,-.125) -- (-.227895,-.184211) -- (-.196818,-.159091) -- cycle; 
              \draw [-,thin]  (0,-.285714) -- (-.10825,-.3125) -- (0,-.333333) -- cycle; 
              \draw [-,thin]  (-.144333,-.25) -- (0,-.285714) -- (-.10825,-.3125) -- cycle; 
              \draw [-,thin]  (-.144333,-.25) -- (-.247429,-.285714) -- (-.203765,-.235294) -- cycle; 
              \draw [-,thin]  (-.1732,-.2) -- (-.236182,-.227273) -- (-.203765,-.235294) -- cycle; 
              \draw [-,thin]  (-.236182,-.227273) -- (-.230933,-.2) -- (-.256593,-.222222) -- cycle; 
              \draw [-,thin]  (-.230933,-.2) -- (-.192444,-.166667) --  (-.196818,-.159091) -- (-.227895,-.184211) -- cycle; 
              \draw [-,thin]  (-.2165,-.125) -- (-.196818,-.159091) -- (-.1732,-.14) -- cycle; 
              \draw [-,thin]  (0,-.2) -- (-.144333,-.25) -- (0,-.285714) -- cycle; 
              \draw [-,thin]  (-.144333,-.25) -- (-.1732,-.2) -- (-.203765,-.235294) -- cycle; 
              \draw [-,thin]  (-.1732,-.2) -- (-.236182,-.227273) -- (-.230933,-.2) -- (-.192444,-.166667) -- cycle; 
              \draw [-,thin]  (-.192444,-.166667)  -- (-.196818,-.159091) -- (-.1732,-.14)-- (-.164952,-.142857) -- cycle; 
              \draw [-,thin]  (0,0) -- (-.2165,-.125) -- (-.1732,-.14) -- cycle; 
              \draw [-,thin]  (0,-.2) -- (-.144333,-.25)  -- (-.1732,-.2) -- (-.0962222,-.166667)-- cycle; 
              \draw [-,thin]  (-.1732,-.2)  -- (-.192444,-.166667)  -- (-.164952,-.142857) -- (-.133231,-.153846) -- cycle; 
              \draw [-,thin]  (0,0) -- (-.164952,-.142857) -- (-.1732,-.14) -- cycle; 
              \draw [-,thin]  (0,-.2) -- (0,-.125) -- (-.0962222,-.166667) -- cycle; 
              \draw [-,thin]  (-.0962222,-.166667) -- (-.1732,-.2) -- (-.133231,-.153846) -- cycle; 
              \draw [-,thin]  (0,0) -- (-.133231,-.153846) -- (-.164952,-.142857) -- cycle; 
              \draw [-,thin]  (0,0) -- (0,-.125) -- (-.0962222,-.166667) -- (-.133231,-.153846) -- cycle; 

	\draw[white,ultra thick,yshift=-0.068pt] (-0.866,-0.5) -- (0.001,-0.5);
	\draw[white,ultra thick,xshift=0.068pt] (0,0) -- (0,-0.501);
	\draw[white,ultra thick,yshift=0.076pt] (-0.866,-0.5) -- (0,0);
	
\end{tikzpicture}
}
\caption{The Borel graph and the Gr\"obner fan of the Hilbert scheme $\Hilb{6t-3}{3}$.}
\label{fig:6t-3}
\end{center}
\end{figure}

\bibliographystyle{amsplain}

\providecommand{\bysame}{\leavevmode\hbox to3em{\hrulefill}\thinspace}
\providecommand{\MR}{\relax\ifhmode\unskip\space\fi MR }
\providecommand{\MRhref}[2]{%
  \href{http://www.ams.org/mathscinet-getitem?mr=#1}{#2}
}
\providecommand{\href}[2]{#2}

\end{document}